\documentclass{amsart}
\usepackage[margin=1.3in]{geometry}
\usepackage[utf8]{inputenc}
\usepackage{color}
\usepackage{tikz-cd}
\usepackage[hidelinks]{hyperref}
\usepackage{enumerate}
\usepackage{amssymb, amsthm, amsmath, bbm }
\usepackage{braket}
\usepackage{array} 
\usepackage[capitalise]{cleveref} 
\usepackage{diagbox}
\usepackage{multirow}
\usepackage{makecell} 
\usepackage{comment}
\newcolumntype{C}[1]{>{\centering\let\newline\\\arraybackslash\hspace{0pt}}m{#1}}

\theoremstyle{theorem}
\newtheorem{theorem}{Theorem}[section]
\newtheorem{corollary}[theorem]{Corollary} 
\newtheorem{lemma}[theorem]{Lemma}
\newtheorem{proposition}[theorem]{Proposition}

\usepackage{array} 
\usepackage{amssymb}

\numberwithin{equation}{section}

\theoremstyle{definition}

\newtheorem{example}[theorem]{Example}
\newtheorem{definition}[theorem]{Definition}
\newtheorem{remark}[theorem]{Remark}

\theoremstyle{definition}

\DeclareMathOperator{\Cone}{Cone}
\DeclareMathOperator{\Bl}{Bl}

\def\s{\sigma}

\def\cA{\mathcal A}\def\cB{\mathcal B}
\def\cF{\mathcal F}
\def\cI{\mathcal I}\def\cJ{\mathcal J}\def\cK{\mathcal K}\def\cL{\mathcal L}
\def\cM{\mathcal M}\def\cN{\mathcal N}\def\cO{\mathcal O}\def\cP{\mathcal P}
\def\cQ{\mathcal Q}

\def\cY{\mathcal Y}

\def\AA{\mathbb A}\def\CC{\mathbb C}
\def\GG{\mathbb G}

\def\NN{\mathbb N}\def\PP{\mathbb P}
\def\RR{\mathbb R}

\def\ZZ{\mathbb Z}

\def\fm{\mathfrak m}

\def\fs{\mathfrak s}\def\fu{\mathfrak u}

\newcommand{\GL}{\operatorname{GL}}
\newcommand{\PGL}{\operatorname{PGL}}
\newcommand{\Lie}{\operatorname{Lie}}

\newcommand{\gitq}{/\!/}
\newcommand{\hU}{\widehat{U}}
\newcommand{\hX}{\widehat{X}}
\newcommand{\Hom}{\operatorname{Hom}}
\newcommand{\Stab}{\operatorname{Stab}}

\newcommand{\Proj}{\operatorname{Proj}}
\newcommand{\Spec}{\operatorname{Spec}}
\newcommand{\Lc}{\mathcal{L}}

\newcommand{\Aut}{\operatorname{Aut}}

\newcommand{\conv}{\operatorname{conv}}

\DeclareMathOperator{\Sym}{Sym}

\newcommand{\rar}{\rightarrow}
\newcommand{\Gm}{\mathbb{G}_m}
\newcommand{\Ga}{\mathbb{G}_a}
\newcommand{\mc}{\mathcal}
\newcommand{\sli}{\operatorname{sl}}

\DeclareMathOperator{\rSpec}{\underline{Spec}}
\DeclareMathOperator{\rProj}{\underline{Proj}}

\title{Relative Geometric Invariant Theory: reductive and non-reductive}
\author{Eloise Hamilton, Victoria Hoskins and Joshua Jackson}
\date{}

\begin{document}

\begin{abstract} We construct good quotients for equivariant actions of group homomorphisms on morphisms of schemes. Using Geometric Invariant Theory, we obtain explicit open semistable loci in the source with Hilbert--Mumford descriptions admitting good quotients relative to a given good quotient of the target. In particular, we obtain quotients for reductive groups acting on projective-over-affine morphism. In the non-reductive case, we consider equivariant actions on affine morphisms which are \lq graded' by a multiplicative group and satisfy certain unipotent stabiliser assumptions. This recovers known results in projective non-reductive GIT as a special case, which also proves the Hilbert-Mumford criterion in that setting. As applications, we consider moduli of unstable objects, representations of quivers with multiplicities and jets.
\end{abstract}
	\maketitle

\setcounter{tocdepth}{1}
\tableofcontents

\section{Introduction}

The starting point for this paper is the observation that many moduli problems can be formulated in terms of an equivariant group action. Suppose we have an equivariant action of a homomorphism $\varphi \colon G \rightarrow H$ of algebraic groups on a morphism $f \colon X\rar Z$ of schemes. The central question of this paper is whether a good quotient of $X$ by $G$ can be constructed using Geometric Invariant Theory (GIT) in such a way that it admits a natural morphism to a good quotient of $Z$ by $H$; we call this a relative quotient. 
Our main results are constructions of relative GIT quotients for equivariant actions on projective-over-affine morphisms between $k$-schemes in both the reductive and non-reductive setting, with Hilbert--Mumford descriptions for their semistable loci. 

When both groups are reductive, our results provide a unified approach that generalises previous approaches to relative reductive GIT. For non-reductive group actions, our relative approach generalises (and recovers) all previously known results on the construction of non-reductive GIT  quotients. Our approach is conceptually simpler and better adapted to various applications. See $\S$\ref{sec relation with others} for more details on the relation between our work and previous results in these areas. 

Applications of our results include moduli of representations of a quiver with multiplicities \cite{HJV}, moduli of unstable objects and jet moduli (for an overview of applications, see $\S$\ref{sec applications}).

\subsection{Overview of the main results}

Let us sketch how we construct our relative quotients. For an equivariant action of $\varphi \colon G \rar H$ on a projective-over-affine morphism $f \colon X\rar Z$, we assume that we have a good quotient $Z'\rar W$ of an open set $Z' \subseteq Z$, which may be obtained via GIT, or otherwise. We assume that $H$ is reductive (see Remarks \ref{rmk base group can be nonred} and \ref{rmk why R reductive} for further discussion on this assumption).  By passing to an affine cone if necessary, we may assume that the map $f$ is affine. We construct our quotient locally over $W$ as the projective spectrum of semi-invariants, and glue. 

When $G$ is reductive, we obtain relative reductive GIT quotient as follows.

\begin{theorem}[Relative reductive GIT, {\cref{thm proj equivariant qnts}}] \label{reductive main theorem}
Suppose that a homomorphism $ G \to H$  of reductive groups over an algebraically closed field $k$ acts equivariantly on a projective-over-affine morphism $f \colon X  \to Z$ of $k$-schemes, and the $G$-action on $X$ is linearised with respect to a $f$-relatively ample line bundle $\Lc$. Assume that $q \colon Z' \to W$ is a good $H$-quotient  of an open subset $j \colon Z' \hookrightarrow Z$.   Then there are open (semi)stable sets $X^{s}(f;q,\Lc) \subseteq X^{ss}(f;q,\Lc) \subseteq X$ and a good $G$-quotient $$X^{ss}(f;q,\Lc) \to X \gitq^f_{\hspace{-2pt} q,\Lc \hspace{2pt}} G:= \rProj_{W} \bigoplus_{n \geq 0} (q_*j^*f_{\ast} \Lc^{n})^{G}$$ that is projective-over-affine over $W$ and restricts to a geometric quotient on the stable locus $X^{s}(f;q,\Lc)$. Moreover, this quotient has the following properties: 
\begin{enumerate}[(i)]
    \item  The sheaf of $\cO_W$-algebras $\bigoplus_{n \geq 0} (q_*j^*f_{\ast} \Lc^{n})^{G}$ is finitely generated, and the good quotient is induced by the inclusion of invariants $(q_*j^*f_{\ast} \Lc^{n})^{G} \subseteq q_{\ast} j^* f_{\ast} \Lc^{n}$ for each $n$;
    \item  The loci $X^{(s)s}(f;q,\Lc)$ admit explicit Hilbert--Mumford descriptions (see  \cref{eq HM descr red}).
\end{enumerate}
\end{theorem}

We then turn our attention to a non-reductive group $G$ with an equivariant action 
\[ (G \twoheadrightarrow R) \curvearrowright (X \stackrel{f}{\rightarrow} Z)\]
on a morphism $f$ such that $R = G/U$ is the quotient of $G$ by the unipotent radical $U$ of $G$. As is usual in the non-reductive setting, we work in characteristic zero and assume that (i)  there is a multiplicative group $\GG_m$ that \lq grades' the equivariant action, and (ii) appropriate unipotent stabiliser assumptions hold. These assumptions are relative versions of the central ideas in non-reductive GIT developed by B\'{e}rczi--Doran--Hawes--Kirwan \cite{Berczi2016,Berczi2023, Berczi2024}.  The grading assumption (i) is given in Definition \ref{definition graded} and the unipotent stabiliser assumption (ii) involves a certain open subscheme $Z^\circ(U)$ where \lq the $U$-stabilisers are minimal' (see \cref{def Zcirc for Z arbitrary}). We construct our quotient in two stages. First, locally over $Z$, we use these assumptions to construct a geometric $U$-quotient via slices. As this quotient is affine over $Z$, we can apply \cref{reductive main theorem} to the resulting action of $G/U$ to obtain the $G$-quotient. Finally, we interpret the semistable locus by proving a non-reductive Hilbert--Mumford criterion in this relative setting (see also \cite{Jackson2026} for a proof).

Our second result concerns an action that is \emph{internally} graded by a central subgroup $\GG_m \subseteq R$.

\begin{theorem}[Relative internally graded non-reductive GIT, $\S$\ref{subsec:internallygradedactions}]\label{main theorem}
   Let $G = U \rtimes R$ be a linear algebraic group with unipotent radical $U$ over an algebraically closed field of characteristic zero. Suppose that $ G \twoheadrightarrow R$ acts equivariantly on an affine morphism $f \colon X \rightarrow Z$ of $k$-schemes and this action is internally graded. Assume that $q \colon Z' \rightarrow W$ is a good $R$-quotient of an open set $j \colon Z' \subseteq Z$ such that $Z' \subseteq Z^\circ(U) \subseteq Z$. Then for any character $\rho \colon G \rightarrow \GG_m$, there are open subsets $X^{(s)s}(f;q,\rho) \subseteq X$ and a good $G$-quotient relative to $W$
      \[ X^{ss}(f;q,\rho) \rightarrow X \gitq^f_{\hspace{-2pt} q,\rho \hspace{2pt}} G:= \rProj_{W} \bigoplus_{n \geq 0} (q_*j^*\hspace{-2pt}f_*\cO_X)^{G}_{\rho^n}  \] 
    that restricts to a geometric quotient on $X^{s}(f;q,\rho)$ with the following properties: 
   \begin{enumerate}[(i)]
        \item the sheaf of graded $\cO_W$-algebras $q_*j^*f_*\cO_{X}^{G,\rho}:=\bigoplus_{n \geq 0} (q_{\ast} j^{\ast} f_*\cO_X)^G_{\rho^n}$ is finitely generated, and the good quotient is induced by the inclusions of semi-invariants;
        \item  the loci $X^{(s)s}(f;q,\rho)$ admit explicit Hilbert--Mumford descriptions (see \cref{eq HM descr nonred}); 
       \item there is an equality $\rSpec_W (q_*j^*\hspace{-2pt}f_*\cO_X)^{G}=W$, and so $X \gitq^f_{\hspace{-2pt} q,\rho \hspace{2pt}} G$ is projective over $W$;
       \item the closed points of $X \gitq^{f}_{\hspace{-2pt} q, \rho \hspace{2pt}} G$ are in bijection with the closed $G$-orbits in $X^{ss}(f;q,\rho)$ or equivalently S-equivalence classes of orbits in $X^{ss}(f;q,\rho)$. Furthermore, an orbit is closed in $X^{ss}(f;q,\rho)$ if and only if it is closed under all flows along one-parameter subgroups of $G$.
   \end{enumerate}
\end{theorem}

An important family of applications of \cref{main theorem} is to provide relative constructions of moduli spaces of unstable objects (such as objects of a fixed Harder--Narasimhan type in an abelian moduli problem such as sheaves or quiver representations), or more generally to give relative quotients of unstable strata appearing in instability stratifications of Hesselink, Kempf, Kirwan and Ness \cite{Hesselink,Kempf1978,KirwanThesis,Ness} for a reductive group $G$ acting on a projective-over-affine scheme. Under certain unipotent stabiliser assumptions, we construct quotients of these unstable strata  (see \cref{thm app unstable quotients}) and show that the non-reductive Hilbert--Mumford criterion can be used to formulate notions of (semi)stability for objects of fixed Harder--Narasimhan type (see $\S$\ref{sec moduli fixed HN type}).

As another application, we explain in $\S$\ref{sec jets} how to simultaneously generalise and simplify the construction of moduli spaces of jets of B\'{e}rczi and Kirwan \cite{Berczi2024} using our relative viewpoint.

Our third result concerns a non-reductive action that is \emph{externally} graded by $\GG_m \subseteq \Aut(G)$.

\begin{theorem}[Relative externally graded non-reductive GIT, $\S$\ref{sec ext graded}]\label{main theorem externally graded case}
Let $G = U \rtimes R$ be a linear algebraic group with unipotent radical $U$ over an algebraically closed field of characteristic zero. Suppose that $ G \twoheadrightarrow R$ acts equivariantly on an affine morphism $f \colon X  \rightarrow Z$ of $k$-schemes and this action is externally graded by $\GG_m \subseteq  \Aut(G)$. Assume that there is a good $R$-quotient $q \colon Z' \rightarrow W$ of an open set $j \colon Z' \hookrightarrow Z$ such that $Z' \subseteq Z^\circ \subseteq Z$. Then for any character $\rho \colon R \rightarrow \GG_m$, there are open subsets $X^{(s)s}(f;q,\rho)$ and a good quotient $X^{ss}(f;q,\rho) \rightarrow X \gitq^{f}_{\hspace{-2pt} q, \rho \hspace{2pt}} G$ which restricts to a geometric quotient on $X^{s}(f;q,\rho)$ with the following properties:
    \begin{enumerate}[(i)]
        \item the sheaf of graded $\cO_W$-algebras $q_*j^*f_*\cO_{X}^{G,\rho}:= \bigoplus_{n \geq 0} (q_*j^*f_*\cO_{X})^G_{\rho^n}$ is finitely generated, and $X \gitq^{f}_{\hspace{-2pt} q, \rho \hspace{2pt}} G = \rProj_W q_*j^*f_*\cO_{X}^{G,\rho}$;
        \item the structure morphism factors as $X \gitq^{f}_{\hspace{-2pt} q, \rho \hspace{2pt}} G \rightarrow X \gitq^{f}_{\hspace{-2pt} q \hspace{2pt}} G:= \rSpec q_*j^*f_*\cO_{X}^{G}  \rightarrow W$, where the first map is projective and the second map is affine;
        \item there is a relative projective completion $X \gitq^{f}_{\hspace{-2pt} q, \rho \hspace{2pt}} G \hookrightarrow \widetilde{X} \gitq^{\widetilde{f}}_{\hspace{-2pt} q, \widetilde{\rho} \hspace{2pt}} \widetilde{G}$ over $W$, where $\widetilde{G} := G \rtimes \GG_m$ acts on $\widetilde{X} := X \times \AA^1$ by the scaling $\GG_m$-action on $\AA^1$ and $\widetilde{\rho} = (\rho, m)$ for $m >\!> 0$, as well as an explicit boundary (see \cref{externally graded});
        \item the loci $X^{(s)s}(f;q,\rho)$ admit explicit Hilbert--Mumford descriptions. 
        \end{enumerate}
    \end{theorem}

A key application of this result is to construct moduli spaces of representations of a quiver with multiplicities \cite{HJV}. A representation of a quiver with multiplicities is given by a free module over a truncated polynomial ring $k[\epsilon]/\epsilon^{m_i}$ for each vertex $i$ in the quiver (where the multiplicities specify the order $m_i$) and appropriate linear maps between these for each arrow. There is a forgetful map from an affine representation space for a quiver with multiplicities to a representation space for the quiver (without multiplicities), which is equivariant with respect to the natural groups acting by change of basis.  In  \cite{HJV}, an external grading for this action is found and  \cref{main theorem externally graded case} is applied to construct moduli spaces of semistable representations of a quiver with multiplicities that are projective-over-affine over King's moduli spaces of semistable representations of the quiver without multiplicities \cite{King1994} (see the summary in $\S$\ref{sec quivers with mult}).

Finally, we also apply our techniques to the situation of an absolute action of a graded non-reductive group $G = U \rtimes R$, rather than an equivariant action. If the grading $\GG_m$-action admits an open Bia{\l}ynicki-Birula stratum, then there is an equivariant action of $G \twoheadrightarrow R$ on the affine retraction morphism from this stratum to its fixed locus given by flowing under the $\GG_m$-action. In this case, we can apply \cref{main theorem} to construct a quotient of an open subscheme (see \cref{Main thm on quotienting BB strata}), and in particular, we recover the results of \cite{Berczi2016} using our relative approach (see the new proof of \cref{thm Uhat} using our relative approach). 

\subsection{Relation with other works}\label{sec relation with others} 
\cref{reductive main theorem} generalises several previous approaches to relative reductive GIT \cite{Mumford1994,Seshadri1977, Reichstein1989,Hu1996,Grulbrandsen2015, Schmitt2017, Schmitt2021}. Mumford introduces group actions and quotients in a relative framework and works with a group scheme over a general base in \cite[Chapter 0]{Mumford1994}, although in later chapters he works over a field. Seshadri extends many of Mumford results to the case of a reductive group scheme $G/S$ over a locally Noetherian base scheme acting on a scheme $X/S$, and establishes a Hilbert--Mumford criterion for a projective scheme over an affine base scheme \cite{Seshadri1977}. If $S$ is of finite type over a universal Japanese ring, then Seshadri shows that the quotient is of finite type over the base. Alper also considers a relative framework in his work on good and adequate moduli spaces \cite{Alper,AlperAMS}, which extends ideas of reductive GIT to the world of algebraic stacks. This should be considered an extension of Seshadri's relative approach, however: in all of this work there is no additional action or stacky structure on the base. In our terminology (\S \ref{subsec:equivariant and fibrewise}), they are fibrewise rather than equivariant actions. 

For equivariant actions, in the case where $G = H$ acts on a morphism, Reichstein constructs relative good quotients and establishes a Hilbert--Mumford criterion under the assumption that $X$ and $Z$ are projective varieties over an algebraically closed field. Hu considers equivariant actions of a surjective homomorphism of reductive groups on morphisms of quasi-projective varieties \cite{Hu1996}, but \cite{Schmitt2017} shows that the Hilbert--Mumford criterion claimed in \cite{Hu1996} is false in this level of generality. Finally, in the case where $G = H$ is linearly reductive and acts equivariantly on a projective morphism $f \colon X \rightarrow Z$ to a Noetherian affine scheme over an arbitrary field, a relative quotient to the affine GIT quotient $q \colon X \rightarrow X \gitq G$ is constructed in \cite{Grulbrandsen2015}, and they establish a relative Hilbert--Mumford criterion. Thus, special cases of Theorem \ref{reductive main theorem} have been established piecemeal in the literature, but to the best of the authors' knowledge, relative GIT, even for the classical case of reductive groups, had not been systematically treated prior to this paper.

Our other main results, \cref{main theorem,main theorem externally graded case}, generalise existing approaches to non-reductive GIT in \cite{Berczi2016,Jackson2021,Hoskins2021,Berczi2023, Berczi2024} to the relative setting. Aside from its conceptual simplicity, our approach has three major advantages. Firstly, it does not require projectivity. Up until this paper, all previous work in  non-reductive GIT required a projective set-up. This made the proofs significantly harder, since at every stage in the process one had to show that one still had a projective scheme to work with. Moreover, many interesting problems do not only involve projective schemes, and in this case one had no choice but to embed (often somewhat artificially) the relevant schemes into projective ones, with the usual cost that nothing better than a quasi-projective quotient could be obtained. As an example, take the construction of moduli of representations of quivers with multiplicities, which is naturally an affine problem. The approach of embedding into projective space was explored by the authors in \cite{HHJ25}, but with T. Vernet in \cite{HJV} we were able to improve these results substantially using our relative approach. A second advantage of our approach is that, as we use twisted affine GIT on each local patch, we overcome some technicalities in \cite{Berczi2016,Jackson2021,Hoskins2021,Berczi2023, Berczi2024}\ related to blow-up procedures and finding appropriately small so-called \emph{well-adapted} perturbations of linearisations. Thirdly, our relative approach makes proving the Hilbert-Mumford criterion a great deal easier.

Alper's ideas on good moduli spaces (i.e.\ stacky versions of reductive GIT quotients) have been extended by David Rydh to give stacky generalisations to certain aspects of non-reductive GIT, where also a relative approach plays an important r\^{o}le. Some of the results of NRGIT have been generalised to stacks by Modin \cite{Modin}, and Cooper and Modin are preparing a paper where they obtain stacky analogues of our results and new applications to moduli of principal bundles \cite{CooperModin}. 

\subsection*{Conventions} 
Throughout $k$ is an algebraically closed field, which is assumed to be of characteristic zero from $\S$\ref{sec:rel quotients for unipotent actions} onwards. By a scheme we mean a separated scheme of finite type over $k$, and by a variety, we mean a reduced scheme of finite type over $k$, which is not necessarily irreducible.

\subsection*{Acknowledgments} 
The authors would like to thank the Isaac Newton Institute for Mathematical Sciences, Cambridge, where much of this research was carried out during the programme \emph{New equivariant methods in algebraic and differential geometry}; this work was supported by EPSRC grant EP/R014604/1. We would also like to thank David Rydh for insightful discussions on non-reductive GIT, and the original idea to work in the relative setting.

\section{Preliminaries on group actions on morphisms}

In this section we let $f \colon X \rar Z$ be a morphism of schemes and $\varphi \colon G \rar H$ be a homomorphism of linear algebraic groups, 
and we define equivariant and fibrewise actions of $\varphi$ on $f$ in $\S$\ref{subsec:equivariant and fibrewise}. We describe properties of such actions in the case where $f$ is affine in $\S$\ref{subsec:fibrewise actions affine} and projective-over-affine in $\S$\ref{subsec:fibrewise actions on poa}. We introduce the notion of a graded equivariant action in $\S$\ref{subsec:gradedactions}, which plays an important r\^{o}le in (relative) non-reductive GIT quotients. Finally in $\S$\ref{subsec:relative quotients}, we explain what we mean by a relative quotient, and in  $\S$\ref{subsec: def rel GIT reductive} we give the basic definitions concerning relative GIT quotients.

\subsection{Equivariant and fibrewise actions} \label{subsec:equivariant and fibrewise}
Let us introduce the central notions of this paper.

\begin{definition}[Equivariant and fibrewise actions]
 Let $\varphi \colon G \to H$ be a homomorphism of linear algebraic groups and $f \colon  X \to Z$  be a morphism of schemes. 
 \begin{enumerate}
     \item An \emph{equivariant action of $\varphi$ on $f$} is a pair of actions $(\sigma_X \colon G \times X \rightarrow X, \sigma_Z \colon H \times Z \rightarrow Z)$ such that the following diagram commutes: 
    \begin{center}  \begin{tikzcd}
G \times X \arrow[r,"\sigma_X"] \ar[d, "\varphi\times f" ] & X \ar[d, "f"] \\
H \times Z  \arrow[r,"\sigma_Z"] & Z. 
\end{tikzcd} 
\end{center} 
\item A \emph{fibrewise action of $G$ on $f$} is an equivariant action of $\varphi \colon G \to \{e\}$ on $f$.
 \end{enumerate}
\end{definition}

Equivalently, a fibrewise $G$-action on $f$ is an action of $G$ on $X$ such that $f$ is $G$-invariant.

One can produce fibrewise actions from equivariant actions, as per the following remark.

\begin{remark}[Fibrewise actions from equivariant actions] \label{lem fibrewise actions from eqvnt actions}
    Suppose that a group homomorphism $\varphi \colon G \rar H$ acts equivariantly on a morphism of schemes $f \colon X \rar Z$. For any normal subgroup $H' \subseteq  H$ acting trivially on $Z$, there is a fibrewise action of $\ker(G \rightarrow H \rightarrow H/H')$ on $f$. Moreover, $G$ acts fibrewise on $f$ if and only if $H$ acts trivially on $Z$. 
\end{remark}

Given a fibrewise $G$-action on $f \colon X \to Z$ and a $G$-equivariant sheaf $\mathcal{F}$ on $X$, there is an induced $G$-action on $f_*\cF(U) =\cF(f^{-1}(U))$ for any open $U \subseteq Z$. We define invariant subsheaves as follows.

\begin{definition}[Invariants and semi-invariants under fibrewise action]\label{definition induced action on pushfoward}
    Suppose that $G$ acts fibrewise on $f \colon X \rightarrow Z$ and $\cF$ is a $G$-equivariant sheaf on $X$. For a character $\rho$ of $G$, we let $\cF_{\rho}$ denote the sheaf $\cF$ with $G$-equivariant structure twisting by $\rho$. Then:
    \begin{enumerate}[(i)]
    \item the \emph{subsheaf of $G$-invariants} is the subsheaf $(f_{\ast} \cF)^G$ of $f_{\ast} \cF$ of $\cO_Z$-modules defined by $(f_*\cF)^G(U):= (f_*\cF(U))^G$ for any open $U \subseteq Z$;
    \item if $\cF$ is a sheaf of $\cO_X$-algebras, the \emph{graded sheaf of $\rho$-twisted semi-invariants} is defined by $(f_*\cF)^{G,\rho}:= \bigoplus_{n \geq 0} (f_*\cF_{\rho^n})^G.$ We refer to sections of $(f_{\ast} \mathcal{F}_{\rho^n})^G$ as \emph{semi-invariants of weight $\rho^n$} (or $\rho$ semi-invariants of weight $n$).
    \end{enumerate} 
\end{definition}

We will only work with semi-invariants when $\mathcal{F} = \mathcal{O}_X$. A section $\sigma \in f_*\cO_X (U)$ over an open $U \subseteq Z$ is a semi-invariant of weight $\rho^n$ if $\sigma(g \cdot x) = \rho^n(g) \sigma(x)$ for all $g \in G$ and $x \in f^{-1}(U)$.

\begin{remark}\label{rmk cannot define invariant sections in equiv case}
If $\varphi \colon G \to H$ acts equivariantly on $f \colon X \to Z$, then for a $G$-equivariant sheaf $\cF$ on $X$ we cannot in general define an invariant subsheaf of $f_{\ast} \cF$, as $f^{-1}(U)$ need not be $G$-invariant for an open set $U \subseteq Z$. However, we can define $(f_{\ast} \cF (U))^G$ for a $H$-invariant open set $U \subseteq Z$, as $f^{-1}(U)$ is $G$-invariant. 
\end{remark}

\subsection{Fibrewise actions on affine morphisms}  \label{subsec:fibrewise actions affine}

We now suppose $f \colon X \rar Z$ is an affine morphism. Recall that for a fixed base $Z$, the relative spectrum functor $\rSpec_Z $ gives an equivalence of categories between sheaves of quasi-coherent $\mathcal{O}_Z$-algebras and affine morphisms of schemes $f \colon X \to Z$, whose inverse is given by $X \mapsto f_* \mc{O}_X$ \cite[Lem.\ 29.11.5]{stacks-project}.

\subsubsection{Correspondence between fibrewise actions and co-actions} In the absolute case, there is a well-known correspondence between $G$-actions on an affine scheme $X$ and co-actions of the associated Hopf algebra $\mathcal{O}_G(G)$ on $\cO(X)$. This extends to the relative setting as follows.

\begin{definition}[Co-actions] \label{defcoaction}
A \emph{co-action} of Hopf algebra $\cO_G(G)$ of a linear algebraic group $G$ on a sheaf of $\cO_Z$-algebras $\cA$ is a morphism $\sigma^{\ast} \colon \cA \to \cA \otimes_{k} \mathcal{O}_G(G)$ of sheaves of $\cO_Z$-algebras satisfying the co-associativity and co-unitality axioms, where the $\mathcal{O}_Z$-module structure on the tensor product is given by multiplication on the left. 
\end{definition}

\begin{proposition}[Relating fibrewise actions and co-actions, {\cite[Exp.\ I, Prop.\ 4.7.2]{SGA3}}] \label{equivalence}
For an affine morphism $f \colon X \to Z$ and a linear algebraic group $G$, there is an equivalence of categories between fibrewise $G$-actions on $f$ and co-actions of $\mathcal{O}_G(G)$ on $\cA = f_* \cO_X$. 
\end{proposition}

\subsubsection{Fibrewise $\GG_m$- and $\GG_a$-actions}
In the relative affine setting, we review how fibrewise actions of these 1-dimensional groups can be equivalently characterised using their co-actions, which generalises the well-known descriptions of these group actions in the absolute affine setting.

\begin{proposition}[Fibrewise $\GG_m$-actions, {\cite[Exp.\ I Prop.\ 4.7.3]{SGA3}}] \label{gmdictionary}
    The functor $\cA \mapsto \rSpec_Z \cA$ induces a contravariant equivalence from the category of quasi-coherent $\ZZ$-graded $\cO_Z$-algebras to the category of fibrewise $\GG_m$-actions on affine $Z$-schemes. 
\end{proposition}

In the absolute setting, $\GG_a$-actions on an affine scheme $X = \Spec A$ over a field $k$ of characteristic zero are in one-to-one correspondence with locally nilpotent derivations of $A$ (see \cite[Section 1.5]{Freudenburg2017} for details). We now formulate an analogous notion in the relative affine setting.

\begin{definition}[Derivations of sheaves of algebras] Given a sheaf of $\mathcal{O}_Z$-algebras $\cA$ on a scheme $Z$, a \emph{derivation of $\cA$ relative to $Z$} is a map of sheaves of $\cO_Z$-modules $D \colon \cA \rightarrow \cA $ satisfying the Leibniz rule, i.e.\ such that for all open sets $U \subseteq Z$ and $g,h \in \cA(U)$ we have \[D(gh) = gD(h) + hD(g).\] A derivation $D$ is \emph{locally nilpotent} if for every open set $U\subseteq Z$ and every $f \in \cA(U)$, there exists some $n \in \NN_{>0}$ such that $D^n(f)=0$. \end{definition}

\begin{proposition}[Fibrewise $\GG_a$-actions] \label{corresp} Let $Z$ be a scheme over a field $k$ of characteristic zero. There is a one-to-one correspondence between fibrewise $\GG_a$-actions on affine $Z$-schemes and  locally nilpotent derivations of sheaves of $\cO_Z$-algebras relative to $Z$.  Moreover, if $D \colon \cA \rightarrow \cA$ is the derivation corresponding to a fibrewise $\GG_a$-action on $\rSpec_Z \cA$, then $\cA^{\GG_a} = \ker (D)$.  
\end{proposition}
\begin{proof}
By \cref{equivalence}, a fibrewise $\GG_a$-action on an affine morphism $f \colon X \to Z$ is equivalent to a co-action $\sigma^{\ast} \colon \cA \to \cA[w]$ which is $\cO_Z$-linear. Given an open set $U \subseteq Z$, there is an action $\sigma_U$ of $\GG_a$ on $X_U := \Spec A_U$ where $A_U : = \mathcal{A}(U)$, with co-action $\sigma_U^{\ast} \colon \cA_U \to \cA_U[w]$. The associated locally nilpotent derivation $D_U \colon A_U \to A_U$ is given by $h \mapsto \frac{d}{dw} \sigma_U^{\ast}(h)$. Choose an open cover $\{U_i\}_{i \in I}$ of $Z$. Then the restrictions of $D_{U_i}$ and $D_{U_j}$ to $U_i \cap U_j$ must coincide, since they represent locally nilpotent derivations associated to the same $\GG_a$-action on $U_i \cap U_j$. Therefore the locally nilpotent derivations $D_{U_i}$ of $A_i$ glue to give a locally nilpotent derivation of $\cA$ relative to $Z$.

Conversely, suppose that $D \colon \mathcal{A} \to \mathcal{A}$ is a locally nilpotent derivation on $\mathcal{A}$ relative to $Z$, and choose an open affine cover $\{U_i\}_{i \in I}$ of $Z$. Then for each affine $U= U_i$ we can define an $\cO_Z$-linear co-action $ \sigma_U^{\ast} \colon \mathcal{A}_U \to \mathcal{A}_U[w]$ given by $ \sigma_U^{\ast}(h) = e^{w D_U(h)}.$ These co-actions agree on $U_i \cap U_j$ since the associated locally nilpotent derivations agree on $U_i \cap U_j$ and the absolute correspondence is one-to-one. Hence the co-actions $\sigma_{U_i}^{\ast}$ glue to give an $\cO_Z$-linear co-action $ \sigma^{\ast} \colon \mathcal{A} \to \mathcal{A}[w]$, and by \cref{equivalence} this data determines a unique fibrewise $\GG_a$-action on $X : = \rSpec_Z \cA \to Z$. 

These constructions are inverse to each other, as this is true on an open cover of $Z$. The final claim follows as in the absolute case (for example, see \cite[Section 1.5]{Freudenburg2017}).
\end{proof}

\subsubsection{Relative Reynolds operators for fibrewise actions of linearly reductive groups}\label{subsec: Rel Reynolds}

\begin{definition}
    For a fibrewise action of a linear algebraic group $G$  on an affine morphism $f \colon X = \rSpec_Z \cA \rightarrow Z$, a \emph{relative Reynolds operator} is an $\cO_Z$-linear projection $R \colon \cA \rightarrow \cA^G$ that is $\cA^G$-linear (i.e.\ for all open subsets $U \subseteq Z$ and $a \in \cA^G(U)$ and $b \in \cA(U)$, we have $R(ab) =a R(b)$).
\end{definition}

For $Z = \Spec k$ and $X = \Spec A$, the above notion coincides with the classical notion. 

\begin{example}
    For a finite group $G$ acting fibrewise on an affine morphism $f \colon X = \rSpec_Z \cA \rightarrow Z$, there is a relative Reynolds operator given by averaging over $G$, provided $|G|$ is invertible in $k$.
\end{example}

In the absolute setting of an action of a linear reductive group $G$ on an affine scheme $\Spec A$ (not necessarily of finite type), there is a unique Reynolds operator $R \colon A \rightarrow A^G$ which is a $k$-linear projection onto $A^G$ which is also $A^G$-linear. This is because the action on $A$ is \emph{rational} (that is, every $a \in A$ is contained in a $G$-invariant finite dimensional $k$-vector subspace of $A$). 

\begin{lemma}\cite[\S1.1]{Mumford1994}\label{lemma abs Reynolds op}
Let $G$ be a linearly reductive group acting rationally on a $k$-algebra $A$. Then there is a unique $k$-linear Reynolds operator $R \colon A \rightarrow A^G$. 
\end{lemma}

In the fibrewise setting, we obtain a relative Reynolds operator by gluing.

\begin{proposition}
 For a fibrewise action of a linearly reductive group $G$ on an affine morphism $f \colon X  \rightarrow Z$, there is a relative Reynolds operator.
\end{proposition}
\begin{proof}
We can base change $f$ along an open affine scheme $U= \Spec B \subseteq Z$ to obtain a morphism $f_U \colon f^{-1}(U) = \Spec A \rightarrow U =\Spec B$ of affine schemes, and as $G$ acts trivially on $U$, the homomorphism $f_U^* \colon B \rightarrow A$ has image in $A^G$. Since $G$ acts rationally on the $k$-algebra $A$, there is a $k$-linear Reynolds operator $R_U \colon A \rightarrow A^G$ which is $A^G$-linear by Lemma \ref{lemma abs Reynolds op}. By the above observation that $f_U^*(B) \subseteq A^G$ and the fact that $R_U$ is $A^G$-linear, we conclude that $R_U$ is $B$-linear. 
    
To glue these local Reynolds operators to give a morphism of sheaves $R \colon \cA \rightarrow \cA^G$ where $\cA = f_* \cO_X$, we need to define $R_U \colon \cA(U) \rightarrow \cA^G(U)$ over any (not necessarily affine) open set $U \subseteq Z$. For this, we take a cover $U = \cup_i U_i$ by open affines and then for $a \in \cA(U)$ define $R_U(a)$ by gluing $R_{U_i}(a|_{U_i}) \in \cA^G(U_i)$, where the Reynolds operators $R_{U_i}$ are defined above. Indeed by the uniqueness of the Reynolds operator, one can glue these sections and the maps $\cA(U) \rightarrow \cA^G(U)$ are compatible with restrictions for $V \hookrightarrow U$. This morphism of sheaves is $\cO_Z$-linear, as by the same gluing argument it suffices to check this when $U$ is affine, which we did above.
\end{proof}

The above proposition actually holds without any finite type assumptions on $f$, $X$ or $Z$, but for the next proposition it is essential (as in the absolute case).

\begin{proposition}
     Suppose a linear algebraic group $G$ acts fibrewise on an affine morphism $f \colon X = \rSpec_Z \cA \rightarrow Z$ and has a relative Reynolds operator. Then $\rSpec_Z \cA^G \rightarrow Z$ is of finite type. 
\end{proposition}
\begin{proof}
The proof is obtained from the proof in the absolute setting. By assumption, $f$ is affine and of finite type, so for every open affine scheme $ U = \Spec B \subseteq Z$ with $V=f^{-1}(U)=\Spec A$, the induced algebra homomorphism $f^*|_V\colon B = \cO_Z(U) \rightarrow A = \cO_X(V) = \cA(U)$ is of finite type, i.e.\ we have $B[x_1,\dots,x_n] \twoheadrightarrow A$. We want to prove that the map $B \rightarrow A^G$ is of finite type. In fact, as $Z$ is of finite type, $B$ and thus also $A$ are finite type $k$-algebras, and it suffices to show the same is true for $A^G$. Here we will use the fact that we have a Reynolds operator $R \colon A \rightarrow A^G$ (although we only use $k$-linearity, it is actually $B$-linear) and apply the absolute argument (see for example, \cite[Theorem 3.4]{Newstead1978}): one first uses $R$ to show that $A$ being Noetherian implies $A^G$ is Noetherian, and then after reducing to the case where $A$ is a polynomial ring (which also uses $R$), one can show the finitely many generators of the ideal in $A^G$ generated by $G$-invariant polynomials of positive degree actually generate $A^G$ as an algebra. 
\end{proof}

\subsection{Fibrewise actions on projective-over-affine morphisms} \label{subsec:fibrewise actions on poa}

We begin by recalling some preliminaries on projective-over-affine morphisms. 

\begin{definition} \label{projoveraffine} A morphism $f \colon X \to Z$ of schemes is \emph{projective-over-affine} if $X = \rProj_Z \cA$ for some sheaf $\cA$ of non-negatively graded\footnote{If $\cA$ is non-positively graded, we reverse the sign of the grading to define $\rProj_Z \cA$.} $\cO_Z$-algebras such that:  \begin{enumerate}[(i)]
    \item the $\cO_Z$-algebra structure is given by a map of sheaves of rings $\cO_Z \hookrightarrow \cA_0$;
    \item $\cA$ is generated by $\cA_1$ over $\cA_0$; \label{secondcond}
    \item $\cA_1$ is of finite type as an $\cA_0$-module.
\end{enumerate}  
We say $f$ is \emph{projective} if in addition $\cA_0$ is a finite $\cO_Z$-module. 

\end{definition}

A projective-over-affine morphism factorises as $ f \colon X:=\rProj_Z \cA \to Y:=\rSpec_{Z}\cA_0 \rightarrow Z$,
where the first map is projective and the second is affine. Moreover, the scheme $X = \rProj_Z \cA$ admits an invertible sheaf $\cO_X(1)$, relatively ample over both $Y$ and $Z$. Since $\cA$ is generated in degree $1$, we can describe this invertible sheaf in two different ways. First, by \cref{projoveraffine} (ii) the map \[\bigoplus_{n\geq 0}\Sym^n_{\cA_0} \cA_1 \rar \cA \] is surjective, and so yields a closed immersion $ X\hookrightarrow \PP^m_Y$ of $Y$-schemes such that $\cO_X(1)$ is the pullback of $\cO_{\PP^m_Y}(1)$. Second, we can obtain $\cO_X(1)$ from the graded module $\cA(1)$, equal to $\cA$ with grading shifted by $1$, by gluing: for a covering of $Z$ by open affines $U$, we glue together the sheaves on $X_U = f^{-1}(U)$ associated to the graded modules $\cA(U)(1)$. In fact, we have invertible sheaves $\cO_X(n)$ defined analogously and $f_*\cO_X(n) = \cA_n$ as in the classical case.

In the absolute case, a choice of ample line bundle $\Lc$ on a projective scheme $X$ determines a homogeneous coordinate ring for $X$ and a corresponding affine cone. The same is true in the relative setting. To simplify notation, given $n \in \mathbb{N}_{\geq 0}$ we let $\Lc^n := \Lc^{\otimes n}$.

\begin{definition}\label{def affine cone relative to L}
For an $f$-ample line bundle $\cL$ on a projective-over-affine morphism $f \colon X \rar Z$, the \emph{affine cone relative to $\Lc$} is 
\begin{equation}\label{affine cone with respect to L} \widehat{f}_{\cL} \colon \widehat{X}_{\cL} := \rSpec_Z \cA(\cL) \rightarrow Z 
\end{equation} where $\cA(\cL) := \bigoplus_{n \geq 0} f_*\cL^{n}$. We let $\pi_{\cL} \colon \widehat{X}_{\cL} \dashrightarrow X = \rProj_Z \cA(\cL)$ denote the natural projection. 

If moreover we have a fibrewise $G$-action on $f$ and $\Lc$ is $G$-equivariant, the associated \emph{graded subalgebra of invariants} is \[ \cA(\cL)^G := \bigoplus_{n \geq 0} (f_*\cL^{n})^G \subseteq \cA(\cL). \]
\end{definition}

Since $\cL$ is $f$-ample, for any section $\sigma \in f_*\cL(U) = \cL(f^{-1}(U))$ over an open set $U \subseteq Z$, we note that the non-vanishing locus $f^{-1}(U)_\sigma$ is affine over $U$. Indeed, using a power of $\cL$, there is a closed immersion $X \hookrightarrow \PP^m_Y$ where $Y = \rSpec_Z f_*\cO_X$ and $\cO_{\PP^m_Y}(1)$ pulls back to a positive power $\cL^{N}$ of $\cL$. If $h \colon \PP^m_Y \rightarrow Z$ denotes the structure map, the non-vanishing locus of a section of $h_*\cO_{\PP^m_Y}(1)(U)$ is affine over $U$, and as $\sigma^N$ is obtained as the pullback of a section of $\cO_{\PP^m_Y}(1)$, we conclude $f^{-1}(U)_{\sigma^N} = f^{-1}(U)_\sigma$ is also affine over $U$.

\begin{remark}\
\begin{enumerate}
\item If $X = \rProj_Z \cA$ and $\cL = \cO_X(1)$, we have $\cA(\cL) = \cA$. If $X = \PP^n_Z$, so $f$ is projective and $Y = Z$, then $\widehat{X}_{\cO_X(1)} = \AA^{n+1}_Z$ and $\pi_{\cO_X(1)}$ is defined outside of the zero section.
    \item If $f$ is affine and $\cL = \cO_X$, then $X = Y$ and $\widehat{X}_{\cL} = \AA^1_X$, so $\pi_{\cL}$ is defined on all of $X$ (this perspective is useful when using a twisted affine set-up, see \cref{rmk equiv descr tw aff quot} below).
\end{enumerate}
In general, $\widehat{X}_{\cL}$ is affine over $Y$ (and $Z$) and the domain of $\pi_{\cL}$ contains $\widehat{X}_{\cL} \setminus Y$.
\end{remark}

Given a $G$-equivariant structure on $\cL$, in the absolute case there is an induced co-action of $\mathcal{O}_G(G)$ on $\bigoplus_{i \geq 0} H^0(X,\mathcal{L}^{i})$, see \cite[Chap 1 \S 3.]{Mumford1994}. This easily generalises to the relative setting, so that a $G$-equivariant structure on $\cL$ determines a co-action of $\mathcal{O}_G(G)$ on $\cA(\mathcal{L})$ in the sense of \cref{defcoaction}, and in turn a fibrewise $G$-action on $\widehat{f}_{\Lc}$ by \cref{equivalence}.

\subsection{Graded equivariant actions} \label{subsec:gradedactions}  

In this section we introduce the notion of a graded equivariant action, which will be used to construct relative non-reductive GIT quotients in $\S$\ref{sec:rel quotients for unipotent actions} and $\S$\ref{sec:relative GIT non-reductive}. 
The grading concerns a subgroup $\GG_m \subseteq  \Aut(G)$ which can be inner (called an \emph{internal} grading) or outer (called an \emph{external} grading).

\begin{definition}[Graded morphisms and actions]\label{definition graded}\
\begin{enumerate}[(1)]
    
    \item A morphism $f \colon X \rightarrow Z$ of schemes is 
    \emph{graded} by a fibrewise $\GG_m$-action on $f$ if $X^{\GG_m} \cong Z$ as $Z$-schemes and for all $x \in X$, the limit $t \cdot x$ as $t \rightarrow 0$ exists and equals $f(x)$.     
    
    \item An equivariant action of a homomorphism $\varphi \colon G \rightarrow H$ on $f \colon X \rightarrow Z$ is \emph{internally graded} by a $\GG_m \subseteq  G$ acting by conjugation on $G$ if this $\GG_m$ grades $\varphi$ and $f$.
    \item\label{def ext graded} An equivariant action of a homomorphism $\varphi \colon G \rightarrow H$ on $f \colon X \rightarrow Z$ is \emph{externally  graded} if there is a $\GG_m$ acting fibrewise on both $\varphi$ and $f$ that grades $\varphi$ and $f$.
    \item A linear algebraic group $G$ is \emph{graded} by a $\GG_m \subseteq G$ acting by conjugation if the quotient by the unipotent radical $G \rightarrow G/U$ is \emph{graded} by this $\GG_m$. We say $G$ is \emph{externally graded} by a $\GG_m$ acting on $G$ if the $\GG_m$-action on $G$ preserves $U$ and $G \rightarrow G/U$ is graded by this $\GG_m$. 
     \end{enumerate}
\end{definition}

We note the following properties of graded actions. 

\begin{remark}\label{remark graded}\
\begin{enumerate}[(i)]
   
   \item \label{classicaldef} Recall that a fibrewise $\GG_m$-action on an affine morphism $f \colon X = \rSpec_Z \cA \to Z$ is equivalent to a grading $\bigoplus_{n \in \mathbb{Z}} \cA_n$ on the sheaf of $\cO_Z$-algebras $\cA$. Our definition of $f$ being graded is equivalent (up to sign convention) to the standard definition of a cone, namely that $\cA$ is non-positively graded and $\cA_0$ is a finite type $ \cO_Z$-module (see \cite[\href{https://stacks.math.columbia.edu/tag/062P}{Tag 062P}]{stacks-project}). 
    \item If $f \colon X \rightarrow Z$ is graded by $\GG_m$, then work of Bia{\l}ynicki-Birula shows that the morphism $f$ is affine under our assumptions on $X$; indeed as $X$ is separated, the $\GG_m$-action on $X$ is \'{e}tale locally linearisable and so we can apply \cite[Theorem A (ii)]{Richarz2019}.
        \item \label{equivalent definition of graded} A linear algebraic group $G= U \rtimes R$ with unipotent radical $U$ is graded by $\GG_m \subseteq G$ acting by conjugation if and only if $\GG_m$ is central in $R$ and the action of $\GG_m$ on the Lie algebra of $U$ has strictly positive weights. Hence, \cref{definition graded} (4) agrees with the usual definition of $G$ being graded by $\GG_m$, as given for example in \cite{Berczi2016}. 
        \item For an equivariant action of $\varphi \colon G \to H$ on $f \colon X \to Z$ that is externally graded, there is an (internally) graded equivariant action of $\widetilde{\varphi} \colon \widetilde{G} : = G \rtimes \GG_m \twoheadrightarrow \widetilde{R} : = R \times \GG_m$ on $f$ and also on $\widetilde{f} \colon \widetilde{X}:=X \times \mathbb{A}^1 \to Z$, where $G$ acts trivially on $\mathbb{A}^1_Z$ and $\GG_m$ acts by scaling. \label{fromexternaltointernal}
\end{enumerate}
\end{remark}

\begin{example}[Examples of graded equivariant actions] \label{examples of graded actions} \
\begin{enumerate}
    \item\label{BBexample} Let $\GG_m$ act on a scheme $X$. For any connected component $Z_i$ of the fixed point locus $X^{\GG_m}$, the retraction map $f_i: X_i^+ \to Z_i$ from the corresponding Bia{\l}ynicki-Birula stratum $X_i^+$ is affine, and the fibrewise $\GG_m$-action on $f_i$ is graded; we will revisit this in $\S$\ref{sec:quotients by non-reductive absolute}.
    \item Let $P:= P_\lambda \subseteq G$ be a parabolic subgroup of a reductive group $G$ associated to a 1-PS $\lambda$ of $G$ as in \cref{def:instab strat} (3). Let $L$ denote the Levi subgroup of $P$ and $\varphi \colon P \to L$ the natural quotient map. Then the equivariant action of $\varphi$ on $\operatorname{Lie} P \to \operatorname{Lie} L$ is graded by $\lambda(\GG_m)$.  
\end{enumerate}
   
\end{example}

The grading condition on $\varphi \colon G \to H$ imposes restrictions on $G$ and $H$ as follows.

\begin{lemma} \label{reductive H} 
 If $\varphi \colon G \to H$ is graded by $\GG_m \subseteq  G$ acting by conjugation, then $G \cong \ker \varphi \rtimes H$. Moreover, $\ker \varphi $ is unipotent, so that $H$ is reductive if and only if $\ker \varphi $ is the unipotent radical of $G$. In particular, if $G$ and $H$ are both reductive, then $G \cong H$. 
\end{lemma}

\begin{proof}
  
By \cref{remark graded}~\eqref{classicaldef}, there is a section of $\varphi$, which exhibits $G \cong \ker \varphi \rtimes H$. For the unipotence of the kernel, we can choose coordinates so that $\Gm \subseteq  G \subseteq  \GL_r$ is diagonal with decreasing weights. Then the kernel of $\varphi$ consists of unipotent matrices, so it is a normal unipotent group, hence contained in the unipotent radical of $G$, and equals the unipotent of radical of $G$ precisely when $H$ is reductive.  
\end{proof}

An important property of graded equivariant actions is the relationship that they impose on invariants on $X$ and on $Z$ for suitable equivariant line bundles. 

\begin{proposition}  \label{properties of graded actions and linearisations}    
Suppose that there is an internally graded equivariant action of $\varphi\colon G \to H$ on an affine morphism $f \colon X \to Z$. Let $\Lc$ denote an $H$-equivariant line bundle on $Z$ whose restricted equivariant structure for the grading $\GG_m$ is trivial. Then there is an isomorphism of invariants $$\bigoplus_{r \geq 0} H^0(X, f^*\mathcal{L}^{r})^G \cong \bigoplus_{r \geq 0} H^0(Z, \mathcal{L}^{r})^H.$$ 
Moreover, if these invariant rings are finitely generated, then we have an isomorphism
\[  X \gitq_{\hspace{-2pt} f^{\ast} \Lc \hspace{2pt}} G := \Proj \oplus_{r \geq 0} H^0(X, f^*\mathcal{L}^{r})^G \cong Z \gitq_{\hspace{-2pt} \cL \hspace{2pt}} H := \Proj \oplus_{r \geq 0} H^0(Z, \mathcal{L}^{r})^H. \]
In the special case where $Z$ is affine and $\Lc = \mathcal{O}_Z$, we have $\mathcal{O}(X)^G \cong \mathcal{O}(Z)^H$. 
\end{proposition} 
\begin{proof} It suffices to prove the statement for fixed $r$. The pullback on global sections gives an injection $f^*\colon H^0(Z, \mathcal{L}^{r}) \hookrightarrow   H^0(X, f^*\mathcal{L}^{r}).$ By the equivariance of $f$ and the definition of the pullback linearisation, this restricts to an injection on invariants $f^*\colon H^0(Z, \mathcal{L}^{r})^H \hookrightarrow   H^0(X, f^*\mathcal{L}^{r})^G.$ To show surjectivity, note that $H^0(X, f^*\mathcal{L}^{r})^G \subseteq  H^0(X, f^*\mathcal{L}^{r})^{\GG_m},$ i.e.\ the $G$-invariants are $\GG_m$-invariant, where this $\GG_m$ denotes the grading multiplicative subgroup. We claim that there is an isomorphism \begin{equation}\label{eq weight zero sections} f^*H^0(Z, \mathcal{L}^{r}) \cong H^0(X, f^*\mathcal{L}^{r})^{\GG_m}. \end{equation} Assuming this is the case, the required surjectivity follows. To prove \eqref{eq weight zero sections}, we note the forward inclusion holds, as the $\Gm$-equivariant structure on $\cL$ is trivial. Conversely, as the action is graded, any $\GG_m$-invariant section is determined by its behaviour on $Z$. 
\end{proof}

In this graded setting, when these invariant rings are finitely generated, the corresponding GIT quotients are collapsed, because the grading condition means every point in $x$ contains the point $f(x) \in Z$ in its orbit closure. To avoid this collapsing, one would like to remove the centre $Z$ of the cone $X$ and one way to do this is to perturb the linearisation $f^{\ast} \Lc$ on $X$ so that the perturbed semistable locus does not meet $Z$. These types of graded actions appear in the study of instability stratifications (see \cite[$\S\S$3.2-3.3]{HoskinsKirwan}, where for the same reason the categorical quotient of an unstable stratum identifies many orbits and a similar perturbation is desirable).

\subsection{Relative quotients} \label{subsec:relative quotients}
We specify the notion of a relative quotient, and explain the obstacles to the existence of relative quotients in $\S$\ref{relquotientsandGIT}, which guides our later approach to relative GIT.

\subsubsection{Defining relative quotients} \label{relquotientdef}

Given an equivariant action of $\varphi \colon G \to H$ on a morphism of schemes $f \colon X \to Z$, it is natural to define a quotient for this action to be a $G$-quotient of $X$, an $H$-quotient of $Z$ together with a morphism from the former to the latter.

\begin{definition}[Quotients by equivariant and fibrewise actions] \label{relativequotients}
    Suppose that a homomorphism of linear algebraic groups $\varphi \colon G \to H$ acts on a morphism of schemes $f \colon X \to Z$. A \emph{relative categorical quotient} for the action of $\varphi$ on $f$ is a categorical quotient $q_G: X \to Y$ for the action of $G$ on $X$ and a categorical quotient $q_H \colon Z \to W$ for the action of $H$ on $Z$ such that $f$ induces a well-defined map $\overline{f} \colon Y \to W$ making the following diagram commute: \begin{center}  \begin{tikzcd}
X \arrow[r,"f"] \ar[d, "q_G"] & Z \ar[d, "q_H"] \\
Y  \arrow[r,"\overline{f}"] & W.  
\end{tikzcd} 
\end{center} 
If $G$ acts fibrewise on $f$, we call this data a \emph{relative categorical quotient for the fibrewise action of $G$ on $f$}. We analogously define \emph{relative good} and \emph{relative geometric quotients} by replacing categorical by good and geometric respectively in the above definition.  
\end{definition}

To emphasise when we are not working relatively, we use the terms \emph{absolute categorical (or good or geometric) quotient} when $Z = \Spec k$; this coincides with the usual notions in \cite{Newstead1978}.

\subsubsection{Relative quotients via classical GIT} \label{relquotientsandGIT}

By \cref{relativequotients}, constructing a relative quotient for an equivariant action of $\varphi \colon G \to H$ on $f \colon X \to Z$ requires in particular constructing a $G$-quotient of $X$ and an $H$-quotient of $Z$. A powerful tool for constructing such quotients is GIT, which requires a linearisation of the action. Under favourable conditions (for example for reductive groups and ample linearisations), a choice of linearisation $\cL_Z$ for the $H$-action on $Z$ determines an open semistable locus $Z^{ss}(\cL_Z) \subseteq Z$ with an absolute good $H$-quotient $q_Z:Z^{ss}(\cL_Z) \rar Z \gitq_{\hspace{-2pt} \cL_Z \hspace{1pt}} H,$ and likewise a choice of linearisation $\cL_X$ for the $G$-action on $X$ determines an open semistable locus in $X$ with an absolute $G$-quotient $q_X \colon X^{ss}(\cL_X)\rar X\gitq_{\hspace{-2pt} \cL_X \hspace{1pt}}G$. In this setting, a relative quotient will exist provided that there is an inclusion \begin{equation}
    \label{f(Zss) < Xss}  X^{ss}(\cL_X) \subseteq f^{-1}(Z^{ss}(\cL_Z)) .\end{equation} 
Thus a reasonable approach for constructing a relative quotient is to start with an $H$-linearisation $\Lc$ on $Z$, and try to construct a $G$-linearisation on $X$ satisfying \eqref{f(Zss) < Xss}. This question has been addressed in the literature in a number of special cases. The case where $f$ is (quasi-)affine and $\Lc_Z$ is ample has been addressed by Mumford in \cite[Prop 1.18]{Mumford1994}. In this case, $\Lc_X := f^{\ast} \Lc_Z$ is ample if $\Lc_Z$ is, and in general this linearisation does not satisfy \eqref{f(Zss) < Xss}. Rather the opposite containment holds, as per the following result.

\begin{proposition} \label{prop mumford's result}
  Suppose that $\varphi \colon G \rightarrow H$ is a homomorphism of reductive groups acting equivariantly on an affine morphism $f \colon X \rightarrow Z$, and let $\cL_Z$ denote an ample $H$-linearisation on $Z$. Then there is an inclusion $X^{ss}(f^*\cL_Z) \supseteq f^{-1}(Z^{ss}(\cL_Z))$. 
\end{proposition}
\begin{proof}
For $x \in f^{-1}(Z^{ss}(\cL_Z))$, as $z = f(x)$ is semistable with respect to $\cL_Z$, there is a section $\sigma \in H^0(Z, \cL_Z^{n})^H$ for some $n >0$ whose non-vanishing locus $Z_{\sigma}$ is affine and contains $z$. Since $f$ is affine, we have that $X_{f^*(\sigma)} = f^{-1}(Z_{\sigma})$ is also affine and contains $x$, proving that $x \in X^{ss}(f^*\cL_Z)$.
\end{proof}

Hence if $f$ is affine and $\Lc_Z$ is ample, then for the inclusion \eqref{f(Zss) < Xss} to hold we must have equality. However, the following example shows equality does not necessarily hold (see \cite[\S 5]{Mumford1994}). 

\begin{example}
    Let $G$ be a reductive group acting on a variety $Z$, with linearisation $\Lc$. Let $X := G \cdot x$ denote a closed and unstable orbit in $Z$, with trivial stabiliser, and let $f \colon X \rar Z$ denote  the inclusion. Then $X \cong G$ and $f^*\cL \cong \cO_G$, so that $X^{ss}(\cL) = X$, but $f^{-1}(Z^{ss}(\cL)) = Z^{ss}(\cL)\cap X  = \emptyset.$
\end{example}

In the special case where $f$ is affine and we have a \emph{graded} equivariant action, we do obtain an equality in \eqref{f(Zss) < Xss}; however, the relative quotient map is an isomorphism as the grading condition collapses the quotient of $X$ onto that of $Z$ (see the discussion after \cref{properties of graded actions and linearisations}). 

In the case where $f$ is projective, one has to perturb $f^{\ast} \Lc_Z$ for it to be ample. Reichstein studied the case of a projective morphism $f \colon X \rightarrow Z$ between projective varieties with an equivariant action of $G=H$ linearised using an ample line bundle $\Lc_Z$ on $Z$. In this case, given a $G$-linearised ample line bundle $\mathcal{K}$ on $X$, the $G$-linearisation $\mathcal{L}_X := (f^{\ast} \Lc_Z)^{n} \otimes \mathcal{K}$ is ample for $n$ sufficiently large, and \eqref{f(Zss) < Xss} holds for this linearisation by \cite[Thm.\ 2.1 (a)]{Reichstein1989}. Thus we obtain a relative quotient map in this setting. This result remains true for any surjective homomorphism $\varphi \colon G \to H$ (i.e.\ $G$ need not coincide with $H$) by \cite[Thm.\ 3.11]{Hu1996}; however the claimed statement that this remains true if $\mathcal{K}$ is only relatively ample and $X$ and $Z$ are only quasi-projective is false, as shown by a counterexample of Schmitt in which $Z$ is only quasi-projective \cite{Schmitt2017}.

While the above results show that in some special cases relative quotient maps for equivariant actions can be constructed between the absolute GIT quotients of $X$ and $Z$, the question of how to construct relative quotients in general remains. Our approach to `relative GIT' is to replace $X$ by $f^{-1}(Z^{ss}(\cL))$ to guarantee the existence of such a relative quotient map.

\subsection{Definition of the relative GIT quotient}\label{subsec: def rel GIT reductive}
In this section, we define (semi)stable loci and relative GIT quotients in the most general situation of an equivariant action of a homomorphism $\varphi \colon G \rightarrow H$ of linear algebraic groups on a projective-over-affine morphisms $f \colon X \rightarrow Z$. In this subsection, we do not a priori work with finite type schemes, as we will consider arbitrary invariant algebras. The rest of this paper will be devoted to proving this construction provides finite type relative good quotients in certain situations. 

The relative GIT quotient depends on a choice of an $f$-ample $G$-equivariant line bundle $\cL$ and a good $H$-quotient $q \colon Z' \rightarrow W$ of an open subset $Z' \subseteq Z$. We consider the base change
\begin{center}  \begin{tikzcd}
X' \arrow[r,"j_X"] \ar[d, "f'"] & X \ar[d, "f"] \\
Z'  \arrow[r,"j_Z"] & Z.  
\end{tikzcd} 
\end{center} 
As $q$ is affine, $q \circ f'$ is projective-over-affine and $j_X^*\cL$ is $(q \circ f')$-ample. As in $\S$\ref{subsec:fibrewise actions on poa}, we have \[ X' = f^{-1}(Z') =\rProj_W q_*\cA(j_X^*\cL) = \rProj_W q_*j_Z^* \cA(\cL) \quad \text{ where } \cA(\cL):= \bigoplus_{n \geq 0} f_* \cL^{n}. \]

The notion of semistability we introduce below involves finding non-vanishing $G$-invariant sections of $\cA(L)_n= f_* \cL^{n}$ for $n > 0$ over suitable open set $U \subseteq Z$. In the equivariant setting, we need $U$ to be $H$-invariant for there to be a $G$-action on $f_* \cL^{n}(U)$; see \cref{rmk cannot define invariant sections in equiv case}. Additionally in the equivariant setting, we will work with opens $U \subseteq Z'$ which are $H$-\emph{orbitwise closed} in the sense that $\overline{H \cdot u} \subseteq U$ for every $u \in U$. In particular, such open sets are saturated with respect to the $H$-quotient $q$, i.e.\ $U = q^{-1}(q(U))$ and the restriction $U \to q(U)$ to an open set $q(U) \subseteq W$. This enables us to translate a notion of semistability phrased in terms of open sets in $W$ to a notion about open orbitwise closed sets in $Z'$ (for example, as in the proof of \cref{prop equiv twisted affine}).

\begin{definition}\label{definition reductive GIT quotients}
For an equivariant action of a homomorphism $\varphi \colon G \rightarrow H$ of linear algebraic groups on a projective-over-affine morphisms $f \colon X \rightarrow Z$, the \emph{relative projective-over-affine GIT quotient} with respect to the $f$-relatively ample $G$-linearisation $\cL$ and the good $H$-quotient $q \colon Z' \rightarrow W$ of an open set $j \colon Z' \hookrightarrow Z$ is the rational $W$-morphism
\[ X' :=f^{-1}(Z') = \rProj_W q_*j_Z^*\cA(\cL) \dashrightarrow X \gitq^{f}_{\hspace{-2pt} q,\cL \hspace{2pt}} G := \rProj_W (q_{\ast}j_Z^* \cA(\cL))^G \]
given by the inclusion of the sheaf of invariants. We define the \emph{semistable locus in $X$ with respect to $(q,\cL)$} as
 \[ X^{ss}(f;q,\cL) : = \left\{ x \in X' \ \left| \ \begin{array}{c} \exists \text{ open $H$-orbitwise closed $U \subseteq Z'$ with $f(x) \in U$ and} \\ \text{$ \exists \: n>0$ and $ \sigma \in \cA(\cL)_{n}(U)^G=\cL^{n}(f^{-1}(U))^G$} \text{ with } \sigma(x) \neq 0 \end{array} \right\} \right. \]
 and the \emph{stable locus in $X$ with respect to $(q,\cL)$} as
 \[ X^{s}(f;q,\cL) : = \left\{ x \in X' \ \left| \ \begin{array}{c} \exists \text{ open $H$-orbitwise closed $U \subseteq Z'$ with $f(x) \in U$ and } \\ \exists \: n >0 \text{ and } \sigma \in \cA(\cL)_{n}(U)^G \text{ such that } \sigma(x) \neq 0,\\  \dim \Stab_G(x) = 0  \text{ and the $G$-action on } f^{-1}(U)_\sigma \text{ is closed}  \end{array} \right\}. \right.\]

If $q = \mathrm{Id_Z}$ or $\cL = \cO_X$, we omit this notation from the (semi)stable loci and relative GIT quotient; for example, we will write $ X \gitq^{f}_{\hspace{-2pt} \hspace{2pt}} G= X \gitq^{ f}_{\hspace{-2pt} \mathrm{Id}_X, \cO_X \hspace{2pt}} G$. If $\cL = (\cO_X)_\rho$ for a character $\rho \colon G \rightarrow \GG_m$, we write $ X \gitq^{f}_{\hspace{-2pt} q,\rho \hspace{2pt}} G= X \gitq^{ f}_{\hspace{-2pt} q, (\cO_X)_\rho \hspace{1pt}} G$ and similarly $X^{(s)s}(f;q,\rho)$.
\end{definition}

Note that for fibrewise actions, where $H = \{\mathrm{Id}\}$, this orbitwise closed condition disappears.

\begin{remark}[Description of the (semi)stable loci]\label{rmk relative ss loci union of absolute ss loci}\
\begin{enumerate}
    \item  The usual definition of the semistable locus \cite[Definition 1.7]{Mumford1994} also includes the condition that the non-vanishing locus of the invariant section is affine. For projective-over-affine morphisms with relatively ample line bundles as above, the nonvanishing locus of a section $\s \in \cA(\cL)_n(U)$ is automatically affine over $U$ as explained in the discussion following \cref{def affine cone relative to L},  so we do not include this condition in our definition.
    \item These (semi)stable loci can be described as a union of absolute (semi)stable loci for suitable open affine $G$-invariant subsets as follows. Let $W_\tau \subseteq W$ be a collection of open affine subsets which form a basis for the open sets in $W$. Then finitely many of the open $G$-invariant sets $X_\tau := (f\circ q)^{-1}(W_\tau)$ cover $X'=f^{-1}(Z')$. Let $X_\tau^{(s)s}(\cL|_{X_\tau})$ denote the GIT (semi)stable loci with respect to the $G$-linearisation $\cL|_{X_\tau}$, which is ample on $X_\tau$.
    
    Then
    \[ X^{ss}(f;q,\cL) = \bigcup_\tau X_\tau^{(s)s}(\cL|_{X_\tau}),\]
    as the sets $Z_\tau:=q^{-1}(W_\tau)$ form a basis for the open $H$-orbitwise closed subsets in $Z'$. 
\end{enumerate}
\end{remark}

Our goal is to prove, in certain situations, that the above semistable locus is the domain of definition of the rational map induced by the inclusion of the invariants, and the induced quotient morphism $X^{ss}(f;q,\cL) \rightarrow X \gitq^{f}_{\hspace{-2pt} q,\cL \hspace{2pt}} G$ is a good $G$-quotient and gives a relative quotient of finite type over the good $H$-quotient $W$ of $Z'$. This will not be possible in this level of generality, as problems already arise for non-reductive group actions in the absolute setting: rings of invariants may not be finitely generated, and even when they are, the induced quotient morphism may not be a good quotient (for example, see \cite[Section 3.1]{DoranKirwan} and \cite[Section 5.1]{Hoskins2023} for various examples).

 \section{Relative GIT for reductive group actions} \label{sec:relative GIT reductive}

Given an equivariant action of a homomorphism $\varphi \colon G \to H$ of reductive groups on a morphism of schemes $f \colon X \to Z$, the aim of this section is to construct a relative quotient using GIT under the assumption that $f$ is a projective-over-affine morphism using the constructions defined in $\S$\ref{subsec: def rel GIT reductive}. Our strategy is to first construct such quotients for fibrewise actions in $\S$\ref{subsec: fibrewise actions}, and then for equivariant actions in $\S$\ref{subsec: equivariant actions} by reducing to the fibrewise case. This reduction to the fibrewise case is achieved by assuming that we have a good $H$-quotient $q \colon Z \to W$ (obtained via GIT or otherwise), and considering the fibrewise $G$-action on the composition $q \circ f$. If we have a good $H$-quotient of only an open subset $Z' \subseteq Z$ instead, the same construction works if we replace $f$ by its restriction to $f^{-1}(Z') \to Z'$, which remains projective-over-affine. Finally $\S$\ref{subsec:comparison} compares our relative GIT quotients with classical (absolute) GIT quotients for reductive groups.

\subsection{Overview}

We apply $\S$\ref{subsec: def rel GIT reductive} in different cases by specifying $(f,\varphi)$ and $(q,\cL)$ as follows.
\begin{center}
\begin{tabular}{|cc|c|c|}
\hline
\multicolumn{2}{|c|}{}                                                                                                                    & Fibrewise ($\varphi \colon G \rightarrow  H = \{ e \}$)                                                     & Equivariant (any $\varphi \colon G \rightarrow  H$)                                                                           \\ \hline 
\multicolumn{1}{|c|}{\multirow{2}{*}{$f$ affine}} & \begin{tabular}[c]{@{}c@{}}Untwisted\\ (or twisted by $\rho = 0$)\end{tabular}         & \begin{tabular}[c]{@{}c@{}}$(q,\cL) =(\mathrm{Id}_Z,\cO_X)$\\ \emph{(untwisted) affine }($\S$\ref{subsec: fibrewise affine})\end{tabular}             & \begin{tabular}[c]{@{}c@{}}$(q,\cL) =(q,\cO_X)$\\ \emph{(untwisted) affine} ($\S$\ref{subseq: equiv affine})\end{tabular}      \\ \cline{2-4} 
\multicolumn{1}{|c|}{}                            & \begin{tabular}[c]{@{}c@{}}Twisted by\\ $\rho \colon G \rightarrow \GG_m$\end{tabular} & \begin{tabular}[c]{@{}c@{}}$(q,\cL)=(\mathrm{Id}_Z,(\cO_X)_\rho)$\\ $\rho$-\emph{twisted affine} ($\S$\ref{subsec: fibrewise twisted affine})\end{tabular} & \begin{tabular}[c]{@{}c@{}}$(q,\cL)=(q,(\cO_X)_\rho)$\\ $\rho$-\emph{twisted affine} ($\S$\ref{subseq: equiv affine})\end{tabular} \\ \hline
\multicolumn{2}{|c|}{\begin{tabular}[c]{@{}c@{}}$f$ projective-over-affine \\ $f$-ample $G$-line bundle $\cL$\end{tabular}}     & \begin{tabular}[c]{@{}c@{}}$(q,\cL) =(\mathrm{Id}_Z,\cL)$\\ \emph{projective-over-affine} ($\S$\ref{subsubsec:fibrewise poa})\end{tabular}         & \begin{tabular}[c]{@{}c@{}} any $(q,\cL)$\\ \emph{projective-over-affine} ($\S$\ref{subsec: equiv poa})\end{tabular}    \\ \hline
\end{tabular}
\end{center}

All entries are special cases of the bottom right entry (for example, if $f \colon X = \rSpec_Z \cA \rightarrow Z$ is affine, then we can write $X = \rProj_Z \cA[z]$ to consider $f$ as a projective-over-affine morphism). We will prove \cref{reductive main theorem}, which concerns the bottom right (and most general) entry, by incrementally progressing to this case: first we consider fibrewise actions and begin with $f$ affine, then $f$ projective-over-affine, and then we proceed to the case of equivariant actions.

\newpage
\subsection{Fibrewise actions}  \label{subsec: fibrewise actions} We treat affine and projective-over-affine morphisms consecutively.

\subsubsection{Affine quotients}\label{subsec: fibrewise affine} 

For a fibrewise action of a reductive group $G$ on an affine morphism $f\colon X \to Z$,  the \emph{relative affine GIT} quotient is defined by taking $(q,\cL) = (\mathrm{Id}_Z,\cO_X)$ as in \cref{definition reductive GIT quotients}, and denoted by $X \gitq^f G$.

\begin{proposition}[Relative affine GIT] \label{thm relative affine GIT}
For a fibrewise action of a reductive group $G$ on an affine morphism $f \colon X = \rSpec_Z \cA \rightarrow Z$, the following statements hold: 
\begin{enumerate}[(i)]
\item\label{rel aff i} $X \gitq^f G = \rSpec_Z \cA^G$ and $X^{ss}(f) = X$; 
\item The domain of definition of the relative affine GIT quotient $X \dashrightarrow X \gitq^f G$ is $X^{ss}(f) = X$;
\item\label{rel aff iii} The morphism $X \rightarrow X \gitq^f G$ is a good quotient and restricts to a geometric quotient on $X^{s}(f)$;
\item\label{rel aff iv}  $X \gitq^f G$ is affine and finite type over $Z$;
\item\label{rel aff v} The closed points in $X \gitq^f G$ are in bijection with the closed $G$-orbits in $X$, or equivalently with $S$-equivalence classes of $G$-orbits in $X$.
\end{enumerate}

\end{proposition}

\begin{proof}
    Part \eqref{rel aff i} holds as $X \gitq^f G := \rProj_Z \cA(\cO_X)^G = \rProj_Z \oplus_{n \geq 0} f_* \cO_X^G = \rSpec_Z f_*\cO_X^G = \rSpec_Z \cA^G$. For the remaining parts, we use that this construction is local over open affines in $Z$, and on each such open affine the relative affine GIT quotient is just the usual (absolute) affine GIT quotient. Thus $X^{ss}(f) = X$ is the domain of $X \dashrightarrow X \gitq^f G$ and this is a good (respectively geometric) $G$-quotient of the semistable (respectively stable) locus. We can check that $X \gitq^f G \rightarrow Z$ is of finite type over open affines in $Z$. Indeed, locally we are taking absolute GIT quotients given by finitely generated rings of invariants, as $G$ is reductive. Part \eqref{rel aff v} is a property of good quotients.
\end{proof}

We note that the first two statements hold even if $G$ is non-reductive, but \eqref{rel aff iv} may fail when $G$ is non-reductive, due to $\cA^G$ not being a finite type $\cO_Z$-algebra, and even when it is of finite type, the induced quotient in \eqref{rel aff iii} may fail to be a good quotient (see \cite[$\S$5.1]{Hoskins2023}).

\begin{remark}[Functoriality and independence of the relative affine GIT quotient] \label{deponf} 
Given a fibrewise $G$-action on an affine morphism $f \colon X \rightarrow Z$, we can precompose with a $G$-equivariant affine morphism of $G$-schemes $e \colon Y \rightarrow X$ and postcompose with any affine morphism $h \colon Z \rightarrow W$. The following commutative diagram relates the various relative quotients, where the middle vertical morphisms are induced by the universal property of the categorical quotient: 
\begin{center}  \begin{tikzcd}
Y  \ar[r] \ar[d, "e"]  & Y \gitq^{f \circ e} G  \ar[r] \ar[d] & Z \ar[d, "\parallel"] \\

X \ar[d, "\parallel"] \ar[r]  & X \gitq^{f} G \ar[d, "\wr"] \ar[r] & Z \ar[d, "h"] \\
X \ar[r] & X \gitq^{h \circ f} G  \ar[r] & W.
\end{tikzcd} 
\end{center}
In particular, the relative affine GIT quotient is functorial with respect to equivariant affine morphisms and as a $k$-scheme the relative affine GIT quotient is independent of the choice of $G$-invariant morphism $f$. The only difference lies in the fact that the two quotients are defined as schemes over potentially different bases. However, the relative affine GIT quotient cannot be defined without specifying an affine morphism $f \colon X \to Z$ and an invariant affine morphism need not exist for a given $G$-action on $X$ (e.g.\ for the transitive $\PGL(n+1)$-action on $\PP^n$, there is no such morphism).     
    
\end{remark}

\subsubsection{Twisted affine quotients}\label{subsec: fibrewise twisted affine} 

For a fibrewise action of a reductive group $G$ on an affine morphism $f \colon X \rightarrow Z$, we can use a character $\rho \colon G \rightarrow \GG_m$ to define a so-called relative `twisted' affine GIT quotient, by extending the absolute notion in \cite{King1994}. We define the \emph{relative $\rho$-twisted affine GIT} quotient by taking $(q,\cL) = (\mathrm{Id}_Z,(\cO_X)_\rho)$ as in \cref{definition reductive GIT quotients}, and denote it by $X \gitq^f_{\hspace{-2pt} \rho \hspace{2pt}} G$.

\begin{theorem} [Relative twisted affine GIT for fibrewise actions] \label{thm twisted affine}
For a fibrewise action of a reductive group $G$ on an affine morphism $f \colon X = \rSpec_Z \cA \rightarrow Z$, the relative twisted affine quotient with respect to a character $\rho \colon G \rightarrow \GG_m$ satisfies the following properties:
\begin{enumerate}[(i)]
\item\label{tw aff i} $X \gitq^f_{\hspace{-2pt} \rho \hspace{2pt}} G = \rProj_Z \cA^{G,\rho}$ for the sheaf $\cA^{G,\rho}$ of $\rho$-twisted semi-invariants (see \cref{definition induced action on pushfoward}); 
\item\label{tw aff ii}  The domain of definition of the relative affine GIT quotient $X \dashrightarrow X \gitq^f_{\hspace{-2pt} \rho \hspace{2pt}} G$ is $X^{ss}(f;\rho)$;
\item\label{tw aff iii} $X^{ss}(f;\rho) \rightarrow X \gitq^f_{\hspace{-2pt} \rho \hspace{2pt}} G$ is a good quotient and restricts to a geometric quotient on $X^{s}(f;\rho)$;
\item\label{tw aff iv}  There is a factorisation $X \gitq^f_{\hspace{-2pt} \rho \hspace{2pt}} G \rightarrow X \gitq^f G \rightarrow Z$, where the first morphism is projective and the second is affine. In particular $X \gitq^f_{\hspace{-2pt} \rho \hspace{2pt}} G$ is projective-over-affine and finite type over $Z$;
\item\label{tw aff v} The closed points of $X \gitq^f_{\hspace{-2pt} \rho \hspace{2pt}} G $ are in bijection with the closed $G$-orbits in $X^{ss}(f;\rho)$, or equivalently with S-equivalence classes of $G$-orbits in $X^{ss}(f;\rho)$.

\end{enumerate}
\end{theorem}
\begin{proof}
    Part \eqref{tw aff i} follows by definition as $\cA((\cO_X)_\rho)^G  = \oplus_{n\geq 0}(f_*(\cO_X)_\rho^{n})^G = \cA^{G,\rho}$.

    For \eqref{tw aff ii} and \eqref{tw aff iii}, the construction of the quotient is local on open affines in $Z$, so we can assume $Z$ is affine. Then the result follows from the equivalent result in the absolute setting, see \cite{King1994}.

     For \eqref{tw aff iv}, we note that $X \gitq_{\hspace{-2pt} \rho \hspace{2pt}}^f G$ is projective over the relative spectrum of the degree zero part $\cA^G$ of the graded sheaf of algebras $\cA^{G,\rho}$. This is precisely $X \gitq^f G$, which is affine over $Z$. Since $G$ is reductive, the sheaf $\cA^{G,\rho}$ of $\cO_Z$-algebras is of finite type over $Z$, and thus so is $X \gitq_{\hspace{-2pt} \rho \hspace{2pt}}^f G$.

     Finally, \eqref{tw aff v} is a consequence of \eqref{tw aff iii}.
\end{proof}

\begin{remark} \label{projective over affine}
In particular, if $\rho = 0$ is the trivial character, the relative $\rho$-twisted GIT quotient coincides with the relative affine GIT quotient. Moreover, for any character $\rho$, we have
\[ X^s(f;0) \subseteq X^s(f;\rho) \subseteq X^{ss}(f;\rho) \subseteq X^{ss}(f;0) =X \]
by a standard VGIT argument (or by the Hilbert--Mumford criterion, see \cref{HM crit for rel twisted affine GIT} below).
\end{remark}

\begin{remark}[Equivalent description of twisted affine quotients] \label{rmk equiv descr tw aff quot}
As in the absolute case, the relative $\rho$-twisted affine GIT quotient can be described in a different way. 
Consider the fibrewise action of $G$ on $h\colon X \times_Z  \mathbb{A}^1_Z = \rSpec_Z \cA[w]  \to Z$ induced by the $G$-action $g \cdot (x,y) = (g \cdot x, \rho(g)^{-1} y)$. Then $\cA[w]^G \cong \bigoplus_{n \geq 0} \cA^G_{\rho^n}$ and the relative $\rho$-twisted affine GIT quotient map coincides with the rational map of $Z$-schemes $$X= \rSpec_Z \mathcal{A} = \rProj_Z \mathcal{A}[w] \dashrightarrow   \rProj_Z \mathcal{A}[w]^G$$ induced by the inclusion $\mathcal{A}[w]^G \to \mathcal{A}[w]$ of sheaves of graded $\mathcal{O}_Z$-algebras.

\end{remark}

The final ingredient of classical reductive GIT that we recover in this setting is the Hilbert--Mumford criterion, which relies on the following well-known fact about the formation of the GIT semistable locus commuting with $G$-equivariant closed immersions. In a (fibrewise) relative setting, Seshadri states that the formation of the reductive GIT semistable locus commutes with pullback along the base \cite[Proposition (7)]{Seshadri1977}, but since the proof is not given we outline the idea below (see \cite[Proposition 4.7]{Alper} for a stacky incarnation of this).

\begin{proposition}  In the setting of Theorem \ref{thm twisted affine}, the formation of the semistable locus $X^{ss}(f;\rho)$ commutes with base change along closed immersions $\iota \colon Z' \hookrightarrow Z$. That is, if $f'\colon X'\rar Z'$ denotes the corresponding base change of $f \colon X\rar Z$ along $\iota$, then the following diagram is Cartesian:  \begin{center}  \begin{tikzcd}
(X')^{ss}(f',\rho)  \ar[d, "f'"]  \ar[r] & X^{ss}(f;\rho) \ar[d, "f"] \\
Z'  \ar[r] & Z.  \end{tikzcd}
\end{center} In particular, the semistable locus $X^{ss}(f;\rho)$ is set-theoretically the union over $z \in Z$ of the classical semistable loci $f^{-1}(z)^{ss}(\rho)$ in each fibre. \end{proposition}
\begin{proof}
Since the statement is local on the base $Z$, we can assume that $Z = \Spec B$ and $X = \Spec A$. Write $Z' = \Spec B'$ and $X' = \Spec A'$. As the action is fibrewise,  $X' \subseteq X$ is a closed $G$-invariant subscheme. By the definition of the $\rho$-semistable locus in terms of the existence of non-vanishing $\rho$-twisted semi-invariants, it suffices to show for any $\sigma \in (A')^G_{\rho^r}$ with $r > 0$, there exists $n >0$ such that $\sigma^n$ extends to $\tilde{\sigma} \in A^G_{\rho^{rn}}$. By \cref{rmk equiv descr tw aff quot}, rather than working with $\rho$-twisted semi-invariants in $A'$ and $A$, we can instead work with genuine invariants in $A'[w]$ and $A[w]$. The claim follows as $G$ is (geometrically) reductive acting on an affine scheme $\Spec A[w]$ and $\Spec A'[w]$ is a closed $G$-invariant subscheme, so for any $\sigma \in A'[w]^G$, there exists $n > 0$ such that $\sigma^n$ extends to $\tilde{\sigma} \in A[w]^G$ by \cite[Corollary 1.2]{Mumford1994}.
\end{proof}

Since $G$ acts on each fibre of $f$ and the orbit closures are also contained in the fibres of $f$, the usual Hilbert--Mumford criterion can be extended to a relative version, using the above fact that the formation of the semistable locus commutes with base change. We let $\langle \rho,\lambda \rangle$ denote the pairing between characters $\rho \colon G \rightarrow \GG_m$ and co-characters $\lambda \colon \GG_m \rightarrow G$. By a 1-parameter subgroup (1-PS) of $G$, we mean a non-trivial co-character.

\begin{corollary}[Hilbert--Mumford criterion for relative twisted affine GIT] \label{HM crit for rel twisted affine GIT}
In the setting of \cref{thm twisted affine}, a point $x \in X$ lies in $X^{ss}(f;\rho)$ if and only if for all 1-PSs $\lambda \colon \GG_m \rightarrow G$ for which $\lim_{t \rightarrow 0} \lambda(t) \cdot x$ exists, we have $\langle \rho,\lambda \rangle \geq 0$, and this point is stable if this inequality is always strict. That is, \[ X^{(s)s}(f;\rho) = \{ x \in X \ | \langle \rho, \lambda \rangle (\geq)  0 \ \text{for all 1-PS $\lambda\colon \GG_m \to G$ such that $\lim_{t \rightarrow 0} \lambda(t) \cdot x$ exists}\}.\]
Moreover, a semistable orbit is stable if and only if it is full dimensional and closed in $X^{ss}(f;\rho)$. Finally, a semistable orbit is closed if and only if it is closed under flows along 1-PS of $G$.
\end{corollary}

From this description, we immediately see that this semistable locus is independent of $f$ as we noted for the quotient in \cref{deponf}. 

\begin{corollary}[Independence of the choice of invariant morphism]
 In the setting of \cref{thm twisted affine}, the $\rho$-twisted semistable locus in $X$ relative to $f$  and the relative twisted affine GIT quotient $X \gitq_{\hspace{-2pt} \rho }^f G$ as $k$-schemes are independent of the choice of $G$-invariant affine morphism $f \colon X \rightarrow Z$.
\end{corollary}

The following simple example illustrates the construction of relative twisted affine GIT quotients, and demonstrates how the semistable loci and quotients depend on the choice of character.

\begin{example}[Variation of relative GIT for $\GG_m$] \label{example VGIT Gm}     For a fibrewise $\GG_m$-action on an affine morphism $f\colon X \to Z$, there is an induced $\ZZ$-grading of the $\mathcal{O}_Z$-algebra $\mathcal{A} = f_*\mathcal{O}_X$. Let $X^+$ (respectively $X^-$ and $X^0$) be the $Z$-subscheme of $X$ with ideal $( \mathcal{A}_i : i > 0 )$ (respectively $( \mathcal{A}_i : i < 0 )$ and $( \mathcal{A}_i : i  \neq 0 )$). Then we have morphisms given by taking the limit
    \[ p_{\pm} : X^{\pm} = \{ x \in X : \lim_{t \rightarrow 0} t^{\pm 1}x \: \mathrm{ exists} \} \longrightarrow X^0 = X^{\GG_m}.\]
    Any character of $\GG_m$ has the form $\rho (t) = t^n$ for $n \in \ZZ$ and, by the Hilbert--Mumford criterion, the semistable locus $X^{ss}(f;\rho)$ only depends on the sign of $\rho$ and is given by
    \[ X^{ss}(f;-) = X \setminus X^+ \quad \text{and} \quad X^{ss}(f;0)=X \quad \text{and} \quad X^{ss}(f;+) = X \setminus X^-. \]
    In particular, if $\mathcal{A}$ is graded in non-positive degrees with $\cA_0 = \cO_{Z}$, so that we have $Z = X^{\GG_m} = X^-$ and $X^+ = X$, then the semistable locus $X^{ss}(f;-)$ is empty, and the GIT quotient \[X^{ss}(f;0)=X \rightarrow X \gitq^f \GG_m = Z\] collapses orbits down to their limit points. In the other chamber, we obtain a geometric quotient  
    \[X^{ss}(f;+) = X \setminus Z \rightarrow X \gitq^f_{\hspace{-2pt} + \hspace{2pt}} \GG_m \cong \rProj_Z \bigoplus_{i \leq 0} \mathcal{A}_i.\]     
\end{example}

Finally, we give a class of examples of relative twisted affine GIT quotients, namely projective-over-affine morphisms. Indeed, just as a projective-over-affine scheme can be constructed as a twisted affine GIT quotient, a projective-over-affine morphism can be constructed as relative twisted affine GIT quotient for the scaling action on its affine cone (see \cref{def affine cone relative to L}). 

\begin{lemma}[Projective-over-affine morphisms as relative twisted affine GIT quotients] \label{lem affine cone}
    Let $f\colon X \rightarrow Z$ be a projective-over-affine morphism and $\cL$ be an $f$-ample line bundle. Then its affine cone $\widehat{f}_{\cL} \colon \widehat{X}_{\cL} \rightarrow Z$ admits a fibrewise scaling $\Gm$-action, and $f\colon  X\rar Z$ is the relative $\rho$-twisted affine GIT quotient for this $\GG_m$-action on $\widehat{f}_\cL$ with respect to the character $\rho = +1$. Moreover, the affine GIT quotient for the fibrewise $\GG_m$-action on $f$ coincides with the natural map $\widehat{X}_{\cL}^{\GG_m} \rar Z$.  
\end{lemma}

\begin{proof}
    The scaling $\Gm$-action on $\widehat{f}_\cL$ corresponds to a grading  $\cA:=\cA(\cL) = \oplus_{i \geq 0} \cA_i$. For the character $\rho = +1$ of $\GG_m$, we have $\cA^{\GG_m,{\rho}} =\cA$, and the sheaf $\cA^{\GG_m}$ of invariants is $\cA_0$. Hence the relative $\rho$-twisted affine GIT quotient for the fibrewise $\GG_m$-action on $\widehat{f}_{\cL}$ coincides with $f$, and the relative affine GIT quotient coincides with $Y = \rSpec_Z \cA_0 = \widehat{X}_{\cL}^{\GG_m} \to Z$.     
\end{proof}

\subsubsection{Projective-over-affine quotients} \label{subsubsec:fibrewise poa}
For a reductive group $G$ acting fibrewise on a projective-over-affine morphism $f\colon X\rar Z$, we linearise the action using a $f$-relatively ample $G$-equivariant line bundle $\cL$. Since $\cL$ is ample, we have $X = \rProj_Z \cA(\cL)$, where $\cA(\cL) = \bigoplus_{n \geq 0} f_{\ast} \Lc^{n}$ is the graded sheaf of $\cO_Z$-algebras introduced in $\S$\ref{subsec:fibrewise actions on poa}. We define the \emph{relative projective-over-affine GIT quotient} for this fibrewise action on $f$ with respect to $\cL$ by taking $(q,\cL) = (\mathrm{Id}_Z,\cL)$ as in \cref{definition reductive GIT quotients}, and denote it by $X \gitq^f_{\hspace{-2pt} \cL \hspace{2pt}} G$. 

These quotients can be obtained as relative twisted affine GIT quotients of their affine cones.

\begin{proposition}[Relative projective-over-affine quotients as relative twisted affine quotients] \label{poaastwistedaff}
    For a fibrewise action of a reductive group $G$ on a projective-over-affine morphism $f \colon X \rightarrow Z$, there is an isomorphism of schemes over $Z$ 
    \[X\gitq^f_{\hspace{-2pt} \Lc \hspace{2pt}} G  \cong \widehat{X}_{\cL} \gitq^{\widehat{f}_{\cL}}_{\hspace{-2pt} \rho \hspace{2pt}} \widehat{G},\] where the latter is the $\rho$-twisted affine GIT quotient for the fibrewise action of $\widehat{G}:=G \times \GG_m$ on the affine cone $\widehat{f}_{\cL}$ as defined at \eqref{affine cone with respect to L}, where $\rho = (0,1)$. In fact, the rational quotient map $\hX_{\cL} \dashrightarrow \widehat{X}_{\cL} \gitq^{\widehat{f}_{\Lc}}_{\hspace{-2pt} \rho \hspace{2pt}} \widehat{G}$ coincides with the composition of the rational map $\pi_{\cL} \colon \hX_{\Lc} \dashrightarrow X$ and the rational quotient map $X \dashrightarrow X\gitq^f_{\hspace{-2pt} \Lc \hspace{2pt}} G.$ In particular, $\pi_{\cL}^{-1}(X^{(s)s}(f;\cL) )= \widehat{X}_{\cL}^{(s)s}(f;\rho)$.
\end{proposition}

\begin{proof}
The result follows directly from the definitions and \cref{lem affine cone}. 
\end{proof}

Before stating the main properties of this quotient, we introduce the Hilbert--Mumford weight to give a numerical description of the (semi)stable loci. 

\begin{definition}
    Let $f\colon X \rar Z$ be a projective-over-affine morphism with a fibrewise $G$-action linearised with respect to an $f$-ample line bundle $\cL$. For $x \in X $ and a 1-PS $\lambda \colon \GG_m \rar G$ such that $x_0 := \lim_{t\rar 0} \lambda(t) \cdot x$ exists, the \emph{Hilbert--Mumford weight} is $\mu^{\cL}(x,\lambda) := - \text{wt}_\lambda (\cL)_{x_0},$ where $\text{wt}_\lambda (\cL)_{x_0}$ is the weight of the $\Gm$-action on the fibre of $\cL$ at the fixed point $x_0$.
\end{definition}
We will give an equivalent characterisation of the Hilbert-Mumford weight in terms of the affine cone $\pi_{\cL} \colon \widehat{X}_{\cL} \dashrightarrow X$ (relative to $Z$) which has an equivariant action of $\widehat{G}:= G \times \GG_m \twoheadrightarrow G$. For this, we first recall that we have a factorisation  
\begin{equation}\label{eq factorisation of p-o-a}
f \colon X \stackrel{p}{\longrightarrow} Y := \rSpec_{Z} f_* \cO_X \stackrel{h}{\longrightarrow} Z
\end{equation}
where  $p$ is projective and $h$ is affine. Since $\widehat{X}_{\cL}$ is affine over $Y$ (and thus also over $Z$), the domain of definition of $\pi_{\cL}$ contains $\widehat{X}_{\cL} \setminus Y = \widehat{X}_{\cL}^{\GG_m-ss}(+)$. We will write 1-PSs of $\widehat{G}$ as pairs $\widehat{\lambda}=(\lambda,r)$ for $\lambda$ a 1-PS of $G$ and $r \in \ZZ$ corresponding to $t \mapsto t^r$.

\begin{lemma}\label{lemma HN function in terms of affine cone}
For a fibrewise action of a reductive group $G$ on a projective-over-affine morphism $f\colon X \rar Z$ linearised by an $f$-ample  line bundle $\Lc$, let $\pi_{\cL} \colon \widehat{X}_{\cL} \dashrightarrow X$ denote the affine cone (relative to $Z$). For a 1-PS $\widehat{\lambda}=(\lambda,r)$ of $\widehat{G}:= G \times \GG_m$ and points $x \in X$ and $\hat{x} \in \pi^{-1}_{\cL}(x) \subseteq \widehat{X}_{\cL}$, the following are equivalent.
\begin{enumerate}
    \item $\lim_{t \rightarrow 0} \lambda(t) \cdot x$ exists and $\mu^{\cL}(x,\lambda) = r$;
    \item $\lim_{t \rightarrow 0} \widehat{\lambda}(t) \cdot \hat{x} = \lim_{t \rightarrow 0} t^r \lambda(t) \cdot \hat{x}$ exists and does not lie in the zero section $Y \subset \widehat{X}_{\cL}$; 
    \item $r$ is the minimal integer such that $\lim_{t \rightarrow 0} t^r \lambda(t) \cdot \hat{x}$ exists. 
\end{enumerate}
\end{lemma}
\begin{proof}
Using the factorisation \eqref{eq factorisation of p-o-a} of $f$, we will write $ y = p(x)$ for the image of $x$ under $p$. Since $p$ is proper, we see that $\lim_{t \rightarrow 0} \lambda(t) \cdot x$ exists in $X$ if and only if $\lim_{t \rightarrow 0} \lambda(t) \cdot y$ exists in $Y$. Since the statements scale under replacing $\cL$ by a positive power, we can assume that $\cL$ induces an equivariant projective embedding $X \hookrightarrow \PP^n_{Y}$ and that $\widehat{X}:=\widehat{X}_{\cL} \hookrightarrow \AA^{n+1}_Y$. We can thus write $\hat{x} \in \widehat{X}$ as $x  = (x_0,\dots,x_n;y)$, where $(x_0,\dots, x_n) \in \widehat{X}_y \hookrightarrow \AA^{n+1}_k$. Furthermore, we can assume that we have chosen our coordinates on $\AA^{n+1}_Y$ so that $\lambda(t)$ is diagonalised in the sense that  
\[   \lambda(t) \cdot \hat{x} = (t^{w_0}x_0,\dots, t^{w_n} x_n; \lambda(t) \cdot y).\]
Hence for $r \in \ZZ$ and $\widehat{\lambda} = (\lambda,r)$, we have
\[ \widehat{\lambda}(t) \cdot \hat{x} = t^r \lambda(t) \cdot \hat{x} = (t^{r+w_1}x_0,\dots, t^{r+w_n} x_n; \lambda(t) \cdot y)\]
and the limit as $t \rightarrow 0$ exists if and only (i) the limit of $\lambda(t) \cdot y$ as $t \rightarrow 0$ exists and (ii) $r+w_i \geq 0$ for all $i$ with $x_i \neq 0$. Furthermore, this limit is non-zero in $\widehat{X} \hookrightarrow \AA^{n+1}_Y$ if and only if $r + w_i = 0$ for some $i$ with $x_i \neq 0$. If this limit exists and is non-zero for some $\widehat{\lambda} =(\lambda,r)$, we let $\hat{x}_0$ denote the limit and note that $r = - \min \{ w_i : x_i \neq 0 \}$. The non-zero coordinates of $\hat{x}_0$ are the $x_i$ with $w_i = -r$. Hence $\text{wt}_\lambda (\cL)_{x_0} = -r$ and
\begin{equation}
\mu^{\cL}(x,\lambda) = - \min \{ w_i : x_i \neq 0 \} = r.
\end{equation}
By the properness of $p$, we thus obtain the equivalence of the first and second claim. We also see the second claim is equivalent to the third claim: if for some other $s \in \ZZ$, the limit of $ t^s \lambda(t) \cdot \hat{x}$ exists as $t \rightarrow 0$, then $s + w_i \geq 0$ for all $i$ with $x_i \neq 0$, thus $s \geq  - \min\{w_i: x_i \neq 0  \} =:r $. In fact $r$ is the unique integer such that the limit of $ t^r \lambda(t) \cdot \hat{x}$ exists and is non-zero.
\end{proof}

In this projective-over-affine setting, we obtain a relative quotient with the following properties.

\begin{theorem}[Fibrewise relative projective-over-affine GIT] \label{prop proj fibrewise qnts}
  For a fibrewise action of a reductive group $G$ on a projective-over-affine morphism $f\colon X \rar Z$ linearised with respect to an $f$-ample $G$-equivariant line bundle $\Lc$, the following statements hold: 
  \begin{enumerate}[(i)]
      \item\label{item fibrewise poa i} $X\gitq^f_{\hspace{-2pt} \cL \hspace{2pt}} G = \rProj_Z \cA(\cL)^G$;
      \item\label{item fibrewise poa ii} The semistable locus $X^{ss}(f;\Lc)$ is the domain of definition of the rational map $ X \dashrightarrow X\gitq^f_{\hspace{-2pt} \cL \hspace{2pt}} G$;
      \item\label{item fibrewise poa iii} $X^{ss}(f;\cL) \rightarrow X \gitq^f_{\hspace{-2pt} \cL \hspace{2pt}} G$ is a good quotient and restricts to a geometric quotient on $X^{s}(f;\cL)$;
      \item\label{item fibrewise poa iv} The structure morphism $\overline{f} \colon X \gitq^f_{\hspace{-2pt} \cL \hspace{2pt}} G \rightarrow Z$ is projective-over-affine and of finite type. If $f$ is projective, then $\overline{f}$ is also projective;
      \item\label{item fibrewise poa v} The closed points of $X\gitq^f_{\hspace{-2pt} \cL \hspace{2pt}} G$ are in bijection with the closed orbits in $X^{ss}(f;\Lc)$, or equivalently with S-equivalence classes of orbits in $X^{ss}(f;\Lc)$;
   
      \item\label{item fibrewise poa HM} The (semi)stable loci admit Hilbert--Mumford descriptions as follows \[X^{(s)s}(f;\cL) = \{x \in X \mid \text{$\mu^{\cL}(\lambda,x) (\geq) 0$ for all 1-PSs $\lambda  \colon \Gm \rightarrow G$ such that $\lim_{t\rar 0} \lambda(t) \cdot x$  exists}\}.\]
  \end{enumerate}
  \end{theorem}

  \begin{proof}
To simplify notation, let $\cA:= \cA(\cL)$. 
Part \eqref{item fibrewise poa i} follows by definition. For \eqref{item fibrewise poa ii} and \eqref{item fibrewise poa iii}, the construction of the quotient is local on open affines in $Z$, so we can assume $Z$ is affine and the result then follows from the result in the absolute setting in \cite{Mumford1994}. For \eqref{item fibrewise poa iv}, we have that  $X \gitq_{\hspace{-2pt} \Lc \hspace{2pt}}^f G \to Z$ is projective-over-affine by \eqref{item fibrewise poa i}, and this map is of finite type as $G$ is reductive. If $f$ is projective, then by definition $\cA_0$ is a finite $ \mathcal{O}_Z$-module, and as $\cA_0 = (\cA_0)^G = (\cA^G)_0$, it follows that $\overline{f}$ is also projective. Part \eqref{item fibrewise poa v} holds as the quotient is good.

Let us refer to the right side of the equation in part \eqref{item fibrewise poa HM} as the Hilbert--Mumford (semi)stable locus in $X$ and denote it by $X^{\mathrm{HM}-(s)s}(f;\cL)$. We will prove this equality by using the isomorphism $X\gitq^f_{\hspace{-2pt} \Lc \hspace{2pt}} G  \cong \widehat{X}_{\cL} \gitq^{\widehat{f}_{\cL}}_{\hspace{-2pt} \rho \hspace{2pt}} \widehat{G}$ of \cref{poaastwistedaff} involving the affine cone $\widehat{f}_{\cL} \colon \widehat{X}:=\widehat{X}_{\cL} \rightarrow Z$, together with the relative twisted affine Hilbert--Mumford criterion, which we write as $\widehat{X}_{\cL}^{(s)s}(f;\rho) = \widehat{X}_{\cL}^{\mathrm{HM}-(s)s}(f;\rho)$, see \cref{HM crit for rel twisted affine GIT}. Let $\pi \colon \widehat{X} \dashrightarrow X $ denote the projection from the affine cone. Since $\pi^{-1}(X^{(s)s}(f;\cL)) = \widehat{X}^{(s)s}(f;\rho)$ by \cref{poaastwistedaff}, it suffices to show the corresponding equality relating the corresponding Hilbert--Mumford (semi)stable loci. In fact, we will just prove the claim for semistable loci, as a similar argument gives the claim for the equality of stable loci. 

As in \cref{lemma HN function in terms of affine cone} above, we fix $x \in X$ and $\hat{x} \in \pi^{-1}(x) \subset \widehat{X}$ and write 1-PSs of $\widehat{G}$ as pairs $\widehat{\lambda}=(\lambda,r)$ for $\lambda$ a 1-PS of $G$ and $r \in \ZZ$ corresponding to $t \mapsto t^r$. First suppose that $x \notin X^{\mathrm{HM}-(s)s}(f;\cL)$, so there is a 1-PS $\lambda \colon \GG_m \rightarrow G$ such that $\lim_{t \rightarrow 0} \lambda(t) \cdot x$ exists and $\mu:=\mu^{\cL}(x,\lambda) < 0$. Then for $\widehat{\lambda}:=(\lambda,\mu)$, we have $\lim_{t \rightarrow 0} \widehat{\lambda}(t) \cdot \hat{x}$ exists in $\pi^{-1}(X)$ by \cref{lemma HN function in terms of affine cone}. Since $\rho = (0,1)$, we have $\langle \rho, \widehat{\lambda} \rangle = \mu < 0$, which proves that $\hat{x} \notin \widehat{X}^{\mathrm{HM}-(s)s}(f;\rho)$. Conversely suppose that $\hat{x} \notin \widehat{X}^{\mathrm{HM}-(s)s}(f;\rho)$, so there exists a 1-PS $\widehat{\lambda}:=(\lambda,r)$ such that $\lim_{t \rightarrow 0} \widehat{\lambda}(t) \cdot \hat{x}$ exists and $\langle \rho, \widehat{\lambda} \rangle = r < 0$. By \cref{lemma HN function in terms of affine cone} we know for some $r' \leq r$ that $\lim_{t \rightarrow 0} (\lambda,r')(t) \cdot \hat{x}$ exists and does not lie in the zero section, thus $\mu^{\cL}(x,\lambda) = r' < 0$ and $\lim_{t \rightarrow 0} \lambda(t) \cdot x$ exists, which proves that $x \notin X^{\mathrm{HM}-(s)s}(f;\cL)$. This completes the proof.
\end{proof}

\subsection{Equivariant actions}  \label{subsec: equivariant actions}

In this section we construct quotients for an equivariant action of $\varphi \colon G \to H$ on $f \colon X \to Z$ relative to a choice of good $H$-quotient of $Z$ (or possibly of an open subset $Z' \subseteq Z$). We construct this good $G$-quotient of $X$ as a relative GIT quotient for the fibrewise action of $G$ on the composition $q \circ f$, using the results of $\S$\ref{subsec: fibrewise actions}. As we will see, this construction generalises to the case where we start with a quotient of an open subset of $Z$. As in $\S$\ref{subsec:fibrewise actions affine}, we first treat the simpler case of affine morphisms.

\subsubsection{Twisted affine quotients}\label{subseq: equiv affine}

We will directly deal with twisted affine quotients in the equivariant setting, as untwisted affine quotients are the special case where the twisting character is trivial. 

Given an equivariant action of a homomorphism $\varphi\colon G \to H$ of reductive groups on an affine morphism $f\colon X \to Z$, suppose that we have a good $H$-quotient $q \colon Z \to W$ and choose a character $\rho \colon G \rightarrow \GG_m$.

  We define the \emph{relative $\rho$-twisted affine GIT quotient} for the equivariant action of $\varphi$ on $f$ with respect to the good quotient $q \colon Z \to W$ and the character $\rho$ by taking $(q,\cL) = (q, (\cO_X)_\rho)$ in \cref{definition reductive GIT quotients}, and we denote it by $X \gitq^{ f}_{\hspace{-2pt} q, \rho \hspace{2pt}} G$. If $\rho$ is the trivial character, then we call this the \emph{relative affine GIT quotient} and denote it by $X \gitq_{\hspace{-2pt} q \hspace{2pt}}^{f} G$.

\begin{proposition}[Equivariant relative twisted affine GIT]\label{prop equiv twisted affine}
   For an equivariant action of a homomorphism $\varphi \colon G \rightarrow H$ of reductive groups on an affine morphism $f \colon X = \rSpec_Z \cA \rightarrow Z$, the relative $\rho$-twisted affine quotient with respect to a good $H$-quotient $q \colon Z \rightarrow W$ and a character $\rho \colon G \rightarrow \GG_m$ coincides with the relative $\rho$-twisted affine GIT quotient for the fibrewise $G$-action on $ q \circ f$. In particular, this relative GIT quotient satisfies the following properties.
   \begin{enumerate}[(i)]
       \item\label{equiv tw aff i} $X \gitq^f_{\hspace{-2pt} q,\rho \hspace{2pt}} G  = X \gitq^{ q \circ f}_{\hspace{-2pt} \rho \hspace{2pt}} G :=\rProj_W (q_{\ast} \cA)^{G,\rho}$ and $X^{(s)s}(f;q,\rho)= X^{(s)s}(q \circ f ;\rho)$;
      
       \item\label{equiv tw aff ii}  $X^{ss}(f;q,\rho)$ is the domain of definition of the relative $\rho$-twisted affine GIT quotient;
       \item\label{equiv tw aff iii}  $X^{ss}(f;q,\rho) \rightarrow X \gitq^f_{\hspace{-2pt} q,\rho \hspace{2pt}} G$ is a good quotient and a geometric quotient restricted to $X^{s}(f;q,\rho)$;
       \item\label{equiv tw aff iv} There is a factorisation $X \gitq^f_{\hspace{-2pt} q,\rho \hspace{2pt}} G  \to X \gitq_{\hspace{-2pt} q \hspace{2pt}}^{f} G \rightarrow W$, where the first morphism is projective and the second morphism is affine;
       \item\label{equiv tw aff v} The closed points of $X \gitq^f_{\hspace{-2pt} q,\rho \hspace{2pt}} G  = X \gitq^{ q \circ f}_{\hspace{-2pt} \rho \hspace{2pt}} G$ are in bijection with the closed orbits in $X^{ss}(f;q,\rho)$, or equivalently with S-equivalence classes of orbits in $X^{ss}(f;q,\rho)$;
       \item\label{equiv tw aff vi} The (semi)stable loci admit explicit Hilbert--Mumford descriptions
       $$ X^{(s)s}(f; q,\rho) = \{ x \in X \ | \ \langle \rho, \lambda \rangle (\geq) 0 \text{ for all $\lambda\colon \GG_m \to G$ such that  $\lim_{t \rightarrow 0} \lambda(t) \cdot x$ exists} \}.$$     
   \end{enumerate}

\end{proposition}

\begin{proof}
Since the good quotient $q$ is affine, so is the composition $q \circ f : X = \rSpec_W q_* \cA \rightarrow W$, and we have $q_* \cA ((\cO_X)_\rho)^G = \oplus_{n} (q_* f_* (\cO_X)_\rho^n))^G = (q_*\cA)^{G,\rho}$, which proves that $X \gitq^f_{\hspace{-2pt} q,\rho \hspace{2pt}} G : = X \gitq^{ q \circ f}_{\hspace{-2pt} \rho \hspace{2pt}} G$ and the relative $\rho$-twisted affine GIT quotient for the equivariant $\varphi$-action on $f$ coincides with the relative $\rho$-twisted affine GIT quotient for the fibrewise $G$-action on $q \circ f$.

It suffices to establishing part \eqref{equiv tw aff i}, as then the remaining parts follow directly from \cref{thm twisted affine} and \cref{HM crit for rel twisted affine GIT}. For this, we will prove that $X^{ss}(f;q,\rho) = X^{ss}(q \circ f ;\rho)$, as the corresponding equality of stable loci follows analogously. First suppose that $x \in X^{ss}(q \circ f; \rho)$; that is, there exists an open subset $V \subseteq W$ and $n >0$ and $\tau \in (q_{\ast} \cA)^G_{\rho^n}(V) = \cO_X((q \circ f)^{-1}(V))_{\rho^n}^G$ such that $\tau(x) \neq 0$. The open subset $U:=q^{-1}(V) \subseteq Z$ is open and $H$-orbitwise closed (since it is saturated), thus $q^{\ast} \tau \in \cA(U)^G_{\rho^n} = \cO_X(f^{-1}(U))^G_{\rho^n}$ and $q^{\ast} \tau(x) \neq 0$, which proves that $x \in X^{ss}(f;q,\rho)$. 

Conversely suppose that $x \in X^{ss}(f;q,\rho)$; that is, there is an open $H$-orbitwise closed set $U \subseteq Z$ and $\sigma \in \cA(U)_{\rho^n}^G = \cO_X(f^{-1}(U))^G_{\rho^n}$ for some $n > 0$ such that $\sigma(x) \neq 0$.  Then $V:=q(U) \subseteq W$ is open and $q_*\sigma \in (q_*\cA)^G_{\rho^n}(V)$ satisfies $q_*\sigma(x) \neq 0$ which shows $x \in X^{ss}(q \circ f; \rho)$.
\end{proof}

In practice, we often do not have a good $H$-quotient of all of $Z$, but rather a quotient $q\colon Z'\rar W$ of an open subset $Z' \subseteq Z$. In this situation, we simply base change $f \colon X \rar Z$ along the inclusion $Z' \rar Z$, and can apply the above result to $f'\colon X' := f^{-1} (Z') \to Z'$ and the good quotient $q \colon Z' \to W'$ instead. By definition, we have $X \gitq_{\hspace{-2pt} q,\rho \hspace{2pt}}^{f} G = X' \gitq_{\hspace{-2pt} q,\rho \hspace{2pt}}^{f'} G $. One obtains variants of the above statements, including a variant of the Hilbert--Mumford criterion stated in Corollary \ref{HM crit for rel twisted affine GIT}, replacing $X$ by $X'$ throughout, which thus depends on $q$ (or equivalently $Z'$).

\begin{corollary}\label{HM crit red equiv tw affine wrt quotient of open in Z}
Given an equivariant action of a homomorphism $\varphi \colon G \rightarrow H$ of reductive groups on an affine morphism $f \colon X \rightarrow Z$ and a good $H$-quotient $q \colon Z' \rightarrow W$ of $Z' \subseteq Z$, we consider the restriction $f'\colon X' := f^{-1} (Z') \to Z'$. Then there is a relative $\rho$-twisted affine GIT quotient which is a good $G$-quotient
\[ X^{ss}(f; q,\rho) = (X')^{ss}(f';q,\rho) \rightarrow X \gitq^{ q \circ f}_{\hspace{-2pt} \rho \hspace{2pt}} G:= X' \gitq^{ q \circ f'}_{\hspace{-2pt} \rho \hspace{2pt}} G\]
and restricts to a geometric quotient on the stable locus $X^{s}(f; q,\rho) = (X')^{s}(f';q,\rho)$.  Moreover,  the (semi)stable loci admit the following explicit Hilbert--Mumford descriptions
$$ X^{(s)s}(f; q,\rho) = \{ x \in X' \ | \ \langle \rho, \lambda \rangle (\geq) 0 \text{ for all $\lambda\colon \GG_m \to G$ such that  $\lim_{t \rightarrow 0} \lambda(t) \cdot x$ exists in $X'$} \}.$$ 
\end{corollary}

In the case where we only have a good $H$-quotient of a subset $Z' \subsetneq Z$, the above Hilbert--Mumford description is only truly explicit if $Z'$ is explicit. In \cref{rmk Z' admits HM descr} below, we discuss the special case when $Z'$ also admits an explicit Hilbert--Mumford description.

\begin{remark} Suppose that the given good $H$-quotient $q \colon Z' \rightarrow W$ restricts to a geometric quotient on an open subset $Z^s$ (for example, in the case when $q$ is constructed via projective GIT and $Z' = Z^{ss}$, the open subset $Z^s$ could be the GIT stable locus). By construction of the relative GIT quotient, we have $f(X^{ss}(f;q,\rho)) \subseteq Z'$, but the relationship between the stable loci in $X$ and $Z$ is less clear. In general, we do not have $f(X^{s}(f;q,\rho)) \subseteq Z^s$ (for example, see \cref{rmk on NR HM semistability} \eqref{stable loci not preserved by f}).  We also do not have $f^{-1}(Z^s) \subseteq X^{s}(f;q,\rho)$; for example, if $H$ is trivial, so we have a fibrewise action and $q = \mathrm{Id}_Z$ and $Z^s = Z' = Z$, we have $X =f^{-1}(Z^s) \nsubseteq X^{s}(f;q,\rho)$ in general. 

\end{remark}

\subsubsection{Projective-over-affine quotients} \label{subsec: equiv poa}
 For an equivariant action on a projective-over-affine morphism $f \colon X \to Z$, we defined the \emph{relative projective-over-affine GIT quotient} with respect to an $f$-relatively ample $G$-equivariant line bundle $\cL$ and a good quotient $q \colon Z' \to W$ of an open subset $Z' \subseteq Z$ in Definition \ref{definition reductive GIT quotients}. 

 If $f' \colon X':= f^{-1}(Z') \rightarrow Z'$ denotes the restriction of $f$, then exactly as in $\S$\ref{subseq: equiv affine} above, this relative projective-over-affine GIT quotient for the equivariant action on $f$ coincides with the relative projective-over-affine quotient for the fibrewise $G$-action on $q \circ f'\colon X'\rar W$ studied in $\S$\ref{subsubsec:fibrewise poa}. Alternatively, as in the fibrewise case (see \cref{poaastwistedaff}), this quotient can be viewed as a twisted affine GIT quotient for an equivariant action on the associated affine cone.

\begin{proposition}\label{lem affine cone equivariant}
    Let $f\colon X \rar Z$ be a projective-over-affine morphism with an equivariant action of a homomorphism $\varphi \colon G\rar H$ of reductive groups and let $\cL$ be an $f$-relatively ample line bundle and $q: Z' \rightarrow W$ be a good $H$-quotient of an open set $Z' \subseteq Z$. Then there are isomorphisms 
\begin{equation}\label{square of isos for rel quotient}
\begin{tikzcd}
X \gitq^{f}_{\hspace{-2pt} q,\cL} G  \ar[d, "\wr"]  \ar[r, "\sim"] & \widehat{X_{\cL}} \gitq^{\widehat{f_{\cL}}}_{\hspace{-2pt} q, \rho} \widehat{G} \ar[d, "\wr"] \\
X' \gitq^{q \circ f'}_{\hspace{-2pt}\cL \hspace{2pt}} G  \ar[r, "\sim"] & \widehat{X'_{\cL}} \gitq^{\widehat{q \circ f'_{\cL}}}_{\hspace{-2pt} \rho} \widehat{G}.  \end{tikzcd}
\end{equation} 
where $f' \colon X':= f^{-1}(Z') \rightarrow Z'$ denotes the restriction of $f$ and $\widehat{f'_{\cL}} \colon \widehat{X'_{\cL}} \rar Z'$ denotes the restricted affine cone, $\widehat{G}:= G \times \GG_m$ and $\rho \colon \widehat{G} \rar \Gm$ is the character given by $(\text{id}_G,1)$. Moreover, we have equalities of (semi)stable loci
\[
X^{G-(s)s}(f;q,\cL) = (X')^{G-(s)s}(q \circ f'; \cL) = \pi_{\cL} (\widehat{X}_{\cL}^{\widehat{G}-(s)s}(\widehat{q \circ f_{\cL}};\rho)) =\pi_{\cL} (\widehat{X_{\cL}}^{\widehat{G}-(s)s}(\widehat{f_{\cL}};q,\rho)) 
\]
where $\pi_{\cL} \colon \widehat{X'_{\cL}}  \dashrightarrow X' $ denote the projection from the affine cone.
\end{proposition}
\begin{proof}
In the above diagram \eqref{square of isos for rel quotient}, the bottom isomorphism only concerns quotients for fibrewise actions and is established in \cref{poaastwistedaff}, whereas the two vertical isomorphisms are between the quotients for equivariant and fibrewise actions follow by definition (as in \cref{prop equiv twisted affine}).

For the equality of (semi)stable loci, one proves the first and last equalities as in \cref{prop equiv twisted affine}, and the middle equality follows from the proof of \cref{prop proj fibrewise qnts}. 
\end{proof}

We will use the above result to prove the following expanded version of \cref{reductive main theorem}. 

\begin{theorem}\label{thm proj equivariant qnts}
  For an equivariant action of a homomorphism $\varphi \colon G\rar H$ of reductive groups on a projective-over-affine morphism $f\colon X \rar Z$, the relative projective-over-affine GIT quotient with respect to an $f$-relatively ample $G$-equivariant line bundle $\cL$ and a good quotient $q\colon Z'\rar W$ of an open subset $Z' \subseteq Z$ satisfies the following properties.
  \begin{enumerate}[(i)]
      \item\label{thm proj equivariant qnts part i} $X^{ss}(f;q,\cL)$ is the domain of definition of the relative projective-over-affine GIT quotient map;
      \item\label{thm proj equivariant qnts part ii}  $X^{ss}(f;q,\cL)\rar X \gitq^{f}_{\hspace{-2pt} q,\cL} G $ is a good quotient and restricts to a geometric quotient on $X^{s}(f;q,\cL)$;
      \item\label{thm proj equivariant qnts part iii}  The structure map $X\gitq_{\hspace{-2pt} q, \cL \hspace{2pt}}^f G \rar W$ is projective-over-affine and of finite type;
      \item\label{thm proj equivariant qnts part iv} The closed points of $X \gitq^{f}_{\hspace{-2pt} q,\cL} G $ are in bijection with the closed orbits in $X^{ss}(f;q,\cL)$, or equivalents with S-equivalence classes of orbits in $X^{ss}(f;q,\cL)$;
      \item\label{thm proj equivariant qnts part v}  The (semi)stable loci admit the following explicit Hilbert--Mumford descriptions  \begin{equation}\label{eq HM descr red}\tag{HM} X^{(s)s}(f;q,\cL) = \left\{ x \in X':=f^{-1}(Z') \left| \begin{array}{c} \text{$\mu^{\cL}(\lambda,x)(\geq) 0$ for all 1-PSs $\lambda  \colon \Gm \rar G$} \\ \text{such that the limit exists in $X'$} \end{array} \right. \right\}.\end{equation}
  \end{enumerate}
\end{theorem} 
\begin{proof}
Parts \eqref{thm proj equivariant qnts part i} and \eqref{thm proj equivariant qnts part ii} follow by the isomorphisms of relative GIT quotients and equality of (semi)stable loci in \cref{lem affine cone equivariant} together with either \cref{prop proj fibrewise qnts} or \cref{prop equiv twisted affine}. Again using \cref{lem affine cone equivariant}, parts \eqref{thm proj equivariant qnts part iii} - \eqref{thm proj equivariant qnts part v} follow from \cref{prop proj fibrewise qnts}.
\end{proof}

All previous results in this section are special cases of this result: the case of an affine morphism $f \colon X = \rSpec_Z \cA \rightarrow Z$ is covered by taking $X = \rProj_Z \cA[w]$, a fibrewise action is the special case where $H$ is trivial, and we can incorporate twisting by a character $\rho$ by twisting the $G$-equivariant structure on $\cO_X$. However, our proof involved incrementally building on simpler cases.

\begin{remark}[The case when $Z'$ admits a Hilbert--Mumford description]\label{rmk Z' admits HM descr}
One special case is when an $H$-linearisation $\cM$ on $Z$ is used to construct the good quotient $q_{\cM} \colon Z^{ss}(\cM) \rar W = Z \gitq_{\hspace{-2pt}\cM} H$ and $Z'=Z^{ss}(\cM)$ admits a Hilbert--Mumford description (for example, $H$ is a reductive group acting on a projective-over-affine scheme $Z$ and the linearisation $\cM$ is ample). In this case, we have a more explicit (semi)stable locus for the above relative GIT quotient. Namely, a point $x \in X$ lies in $X^{ss}(f;q_{\cM},\cL)$ if and only if the following two statements hold:
\begin{enumerate}
    \item for $z = f(x)$, we have $\mu^{\cM}(\lambda,z)\geq 0$ for all 1-PSs $\lambda_H  \colon \Gm \rar H$ such that $\lim_{t \rightarrow 0} \lambda_H(t) \cdot z$ exists in $Z$ (that is $z \in Z^{ss}(\cM)$);
    \item $\mu^{\cL}(\lambda,x)\geq 0$ for all 1-PSs $\lambda_G  \colon \Gm \rar G$ such that both $x_0=\lim_{t \rightarrow 0} \lambda_G(t) \cdot z$ exists in $X$ and part (1) holds for $z_0:=f(x_0)$ (that is $z_0 \in Z^{ss}(\cM)$).
\end{enumerate}
In this setting, by construction $f$ restricts to a morphism $X^{ss}(f;q_{\cM},\cL) \rightarrow Z^{ss}(\cM)$ but the same is not true for the stable loci. If $Z^{s}(\cM)$ denotes the GIT stable locus, then in general $f$ does \emph{not} restrict to a morphism $X^{s}(f;q_{\cM},\cL) \rightarrow Z^{s}(\cM)$ and in particular, we do not get an induced morphism between the geometric quotients of the stable loci in $X$ and $Z$. Even when $f$ is affine and we take a trivial character $\rho$, this may not be the case.
\end{remark}

\begin{remark}[Variation of the quotient of the base]
For an equivariant action $\varphi \colon G \rightarrow H$ on a morphism $f \colon X = \rSpec_Z \cA \rightarrow Z $, suppose we have two good $H$-quotients $q_i : Z_i \rightarrow W_i$ where $j_i: Z_i \subseteq Z$ are open subsets for $ i= 1,2$. Suppose that we have $Z_1 \subseteq Z_2$, and so there is an induced morphism $\alpha \colon W_1 \rightarrow W_2$. Then there is an inclusion $X_1 := f^{-1}(Z_1) \subseteq X_2 :=f^{-1}(Z_2)$ which, via the universal property of good quotients, induces an inclusion of relative affine GIT quotients
\[ X \gitq^f_{\hspace{-2pt} q_1 \hspace{2pt}} G = \rSpec_{W_1} (q_{1 *} j_1^* \cA)^G \rightarrow X \gitq^f_{\hspace{-2pt} q_2 \hspace{2pt}} G = \rSpec_{W_2} (q_{2*} j_2^* \cA)^G. \]
In general, $X^{ss}(f;q_1,\rho)$ is not contained in $ X^{ss}(f;q_2,\rho)$ as there can be a 1-PS whose limit does not exist in $X_1$, but does exists in $X_2$, as per the following example.
\end{remark}

\begin{example}

Consider the action of $\pi_2 \colon G = \GG_m^2 \rightarrow H=\GG_m$ on $p_2 \colon X=\AA^2 \times \PP^1 \rightarrow Z=\PP^1$ by $(s,t) \cdot (a,[z_1:z_2]) = (sa,[tz_1:t^{-1}z_2])$. For the $H$-action on $Z = \PP^1$, we let $\cL_1  = \cO_{\PP^1}(1)$ be linearised via $\text{diag}(t,t^{-1})$, and $\cL_2$ denote the twisted linearisation given by $\text{diag}(t^2,1)$. The respective reductive GIT quotients give two maps $q_i \colon Z_i \rar \text{pt}$ where $Z_1=\AA^1\setminus \{0\} \subset Z_2= \AA^1$.  For the character $\rho = (1,-1)$ of $\Gm^2$, the semistable locus $X^{ss}(q_2,\rho) = \emptyset$, because taking the limit under $H$ as $t\rar0$ destabilises every point, as the limit always exists over $Z_2$ and we get negative pairing with $\rho$. On the other hand, $X^{ss}(q_1,\rho) =(\AA^2\setminus \{0\}) \times Z_1 \neq \emptyset $ because (for example) any cocharacter such that the limit exists in $Z_1$ must be trivial in the $H$-direction, and must be positive along $\rho$ if the limit is to exist in $X$. Thus the expected naive inclusion of semistable loci fails. 
\end{example}

\begin{remark}[Quotients with non-reductive base group $H$]\label{rmk base group can be nonred} Since by assumption we are given a good $H$-quotient $q \colon Z' \rightarrow W$ and did not concern ourselves with how it might be obtained, we can drop the requirement for $H$ to be reductive (except in \cref{rmk Z' admits HM descr} above, where we use a Hilbert--Mumford criterion for the $H$-action on $Z$). However, the assumption that $G$ is reductive is essential for the sheaves of algebras to be finitely generated. 
\end{remark}

\subsection{Comparison with classical GIT quotients} \label{subsec:comparison}
We now compare the above relative GIT quotients with absolute quotients that can be constructed using classical GIT. That is, given an equivariant action of $\varphi \colon G \to H$ with $G$ and $H$ reductive on a projective-over-affine morphism $f \colon X \to Z$, together with a good $G$-quotient $q \colon  Z \to W$, we are interested in comparing the relative GIT quotient $X \gitq^f_{\hspace{-2pt} q, \cO_X(1) \hspace{2pt}} G$ to the absolute GIT quotient $X \gitq_{\hspace{-2pt} \mathcal{O}_X(1) \hspace{2pt}} G$. Our first observation is that these two quotients need not coincide. The explanation for this is that Mumford's quotient is the relative fibrewise quotient with respect to the structure morphism $\pi\colon X \rar \Spec k$, but we take a relative quotient with respect to $f \colon X\rar Z$ and use a line bundle which is ample relative to $f$, but need not be relatively ample with respect to $\pi$ (i.e.\ it is not ample). The following simple example serves to illustrate this point.

\begin{example} Consider the projective morphism \[ \pi_2  \colon X = \PP^n\times \PP^m \rar \PP^m = Z\] given by projection onto the second factor. Then for a fibrewise action with respect to this morphism, the relative and absolute semistable loci (and hence also the quotients) can fail to coincide. Indeed, the line bundle $\cO_X(1,0) = \pi_1^*\cO_{\PP^n}(1)$ is relatively ample, but none of its global sections have affine nonvanishing loci, so for any group action the absolute semistable locus as defined by Mumford will always be empty. On the other hand, the relative semistable locus with respect to $f\colon X\rar Z$ can be non-empty. For example, if $n=1$ and $\Gm$ acts with weights $+1,-1$ on the sections $x,y$ of $\cO_X(1,0)$, we obtain relative semistable locus \[\AA^1\setminus \{0\}\times \PP^m \subseteq X\] with the good quotient being projection onto the second factor. \end{example}

Nevertheless, the absolute and relative quotients do sometimes coincide as we now explain. 

\begin{proposition}
For an equivariant action of a reductive group homomorphism $\varphi \colon G \to H$ on a morphism $f \colon X \to Z$ and a good $H$-quotient $q \colon Z \to W$, the following statements hold.
\begin{enumerate}
\item Suppose $f$ is affine and $Z$ is affine.
\begin{enumerate}
   \item\label{compare aff} $X\gitq^f_{\hspace{-2pt} q} G \cong X\gitq G := \Spec \cO_X(X)^G$ as $k$-schemes (the relative affine GIT quotient is the absolute affine GIT quotient of the affine scheme $X$);
\item\label{compare tw aff} $X\gitq^f_{\hspace{-2pt}  q,\rho} G \cong X\gitq_{\hspace{-2pt} \rho} G := \Proj \cO_X(X)^{G,\rho}$ as $k$-schemes (the relative $\rho$-twisted affine GIT quotient is the absolute $\rho$-twisted affine GIT quotient of $X$);
\end{enumerate}
\item\label{compare proj} Suppose $H$ is trivial, so we have a fibrewise $G$-action on $f$. If $f$ is projective and $Z$ is projective and $\cL$ is $f$-ample, then $X\gitq^f_{\hspace{-2pt} \cL} G \cong X/\!/_{\hspace{-2pt} \cL'}G$ where $\cL' := \cL \otimes f^*\cL_Z^N$ for any ample line bundle $\cL_Z$ on $Z$ and $N \gg 0$ so that $\cL'$ is ample.
\end{enumerate}
\end{proposition}
\begin{proof}
For \eqref{compare aff}, we can write $X =\rSpec_Z \cA = \Spec A$ where $\cA = f_*\cO_X$ and $A = \cA(X) = \cO_X(X)$. By \cref{prop equiv twisted affine} and \cref{thm relative affine GIT}, we have 
\[ X\gitq^f_{\hspace{-2pt} q} G := \rProj_W q_*\cA^{G,\rho = 0} = \rSpec_W q_* \cA^G = \Spec A^G = X\gitq G. \]
The proof of \eqref{compare tw aff} follows exactly in the same way. 

For \eqref{compare proj}, it suffices to prove $X\gitq^f_{\hspace{-2pt} \cL} G \simeq X\gitq^f_{\hspace{-2pt} \cL'} G$ where $\cL' := \cL \otimes f^*\cL_Z^N$ for any ample line bundle $\cL_Z$ on $Z$ (that this GIT quotient is independent of perturbations using line bundles pulled back from the base $Z$), as then we have 
\[X\gitq^f_{\hspace{-2pt} \cL'} G = \rProj_Z \cA(\cL')^G = \Proj \cA(\cL')^G(Z) = \Proj \oplus_{n \geq 0} H^0(X,\cL')^G = X/\!/_{\hspace{-2pt} \cL'}G.\]
To prove the claimed independence, we note that $f_*((\cL')^{n}) = (f_*\cL^{n}) \otimes \cL_Z^{nN}$ and then $f_*((\cL')^{n})^G \cong f_*(\cL^n)^G \otimes \cL_Z^{nN}$ as the action on $\cL_Z^{nN}$ is trivial. By working over an open subset $U \subseteq Z$ over which $\cL_Z$ is trivial, we can locally construct the claimed isomorphism $X\gitq^f_{\hspace{-2pt} \cL} G \simeq X\gitq^f_{\hspace{-2pt} \cL'} G$.
\end{proof}

Note that in the third case, we restricted to fibrewise actions, as if there was a non-trivial  $H$-action on the projective scheme $Z$, then we would in general only be able to construct a good quotient of an open subset $Z' \subseteq Z$, and $X' :=f^{-1}(Z')$ would no longer be projective.

\section{Relative quotients for fibrewise unipotent actions} \label{sec:rel quotients for unipotent actions}

From now on, we assume $k$ is a field of characteristic zero. In this section, we consider fibrewise unipotent actions on affine morphisms $f \colon X \rightarrow Z$ and use local slices to construct relative quotients  over an open locus $Z^\circ \subset Z$ where the $U$-stabilisers are well-behaved, provided the action is graded by a copy of $\GG_m$; this condition is crucial to produce slices. We begin with $\GG_a$-actions in $\S$\ref{sec Ga slice} and $\S$\ref{subsec:quotientsbygradedga}, before proceeding to unipotent actions in $\S$\ref{subsubsec:graded unipotent quotients}.

\subsection{$\GG_a$-quotients via slices}\label{sec Ga slice}

In the absolute case, given a $\GG_a$-action on an affine scheme $X = \Spec A$ and its corresponding locally nilpotent derivation $D$, an important concept is that of a slice: a function $s \in A$ is a \emph{slice} for $D$ if $D(s)=1$. If there exists a slice $s$, then the invariants $A^{\GG_a}$ are finitely generated and the map $X \to \Spec A^{\GG_a}$ induced from the inclusion  $A^{\GG_a} \to A$ is a trivial $\GG_a$-bundle. In fact, the slice determines a Reynolds operator $R_s \colon A \rightarrow A^{\GG_a}$, but unlike in the reductive case this is not unique and any two slices differ by a $\GG_a$-invariant function. The zero locus $S \subseteq X$ of a slice $s$ is then a geometric slice for the action of $\GG_a$ on $X$, in the sense that $X \cong S \times \GG_a$.  

Most of this can be translated to the setting of a fibrewise $\Ga$-action on an affine morphism $f\colon X \rar Z$. Recall from Proposition \ref{corresp} that such an action is equivalent to the data of an $\cO_Z$-linear locally nilpotent derivation $D \colon \cA \rar \cA$ of the sheaf of algebras $\cA = f_*\cA$. However, rather than insisting we have a global slice, we may instead have local slices on different open subsets in $Z$. This results in a $\GG_a$-quotient which is locally trivial over $Z$.

\begin{definition}[Slice locus in $X$] \label{slicelocus} Let $f\colon X = \rSpec_Z \cA \rar Z$ be an affine morphism. Given a locally nilpotent derivation $D\colon \cA \rightarrow \cA$ relative to $Z$ and an open set $U \subseteq Z$, an element $s \in \cA(U)$ is a \emph{slice over $U$} if $D(s) = 1$. The \emph{slice locus} in $Z$ is \[Z_{\sli} := \bigcup_{\exists \text{ slice over }U}\hspace{-0.5cm} U  \hspace{ 0.5cm}\subseteq Z.\] The \emph{slice locus} in $X$ is $X_{\sli} := f^{-1}(Z_{\sli}).$
\end{definition} 

Note that $f^{-1}(Z_{\sli}) = \rSpec_{Z_{\sli}} \iota^{\ast} \mathcal{A}$ where $\iota\colon Z_{\sli} \to Z$ denotes the natural open immersion, by the base-change properties of the relative spectrum \cite[01LX]{stacks-project}.

\begin{remark}[Comparison with local slices in the classical setting] For a $\Ga$-action on an affine scheme $X = \Spec A$, with corresponding locally nilpotent derivation $D\colon A\rightarrow A$, a \emph{pre-slice} (also known as a local slice) is an element $s \in \operatorname{ker} D^2 \setminus \operatorname{ker} D$ (see \cite[$\S$1.1]{Freudenburg2017}). If $A$ is integral and $s$ is a pre-slice with $t:=D(s)$, then one can show that $st^{-1}$ is a slice on the open set $U = X_t$ where $t$ is invertible. 

\end{remark}

\begin{lemma}\label{lemma global slices for add gp give Reynolds operator}
If a fibrewise $\GG_a$-action on an affine morphism $f \colon X = \rSpec_Z \cA \rightarrow Z$ has a global slice $ s\in \cA(Z)$, then there is a relative Reynolds operator $R_s \colon \cA \rightarrow \cA^{\GG_a}$. In particular $\cA^{\GG_a}$ is of finite type. Moreover $\cA \cong \cA^{\GG_a}[s]$ and $X = \rSpec_Z \cA \rightarrow X/\GG_a:= \rSpec_Z \cA^{\GG_a}$ is a trivial $\GG_a$-bundle.
\end{lemma}
\begin{proof}
Let  $D \colon \cA \rightarrow \cA$ denote the corresponding locally nilpotent derivation relative to $Z$. We claim that there is a relative Reynolds operator $R_s \colon \cA \rightarrow \cA$ given by $a \mapsto \sum_{n \geq 0}  \frac{(-s_U)^n}{n!} D_U^n(a)$ for $a \in \cA(U)$. Indeed this restricts to the identity on $\cA^{\GG_a}$ and since $D \circ R_s = 0$, the image of $R_s$ is equal to $\ker(D) =\cA^{\GG_a}$. Furthermore, since $D$ is $\cA^{\GG_a}$-linear, we see that $R_s$ is $\cA^{\GG_a}$-linear and thus also $\cO_Z$-linear as for fibrewise actions the homomorphism $\cO_Z \rightarrow \cA$ has image in $\cA^{\GG_a}$. In fact, $R_s$ is an algebra homomorphism, as $\exp D \colon  \cA \rightarrow \cA$ is an algebra homomorphism. Since $R_s$ is an algebra homomorphism, we can take finitely many (local) generators for $\cA$ and their images under $R_s$ give finitely many (local) generators for $\cA^{\GG_a}$, which proves this $\cO_Z$-algebra is of finite type.

The final claim that $\cA \cong \cA^{\GG_a}[s]$ follows exactly as in the absolute case: we can write any $f \in \cA$ as a polynomial $\sum_{n\geq0} \frac{s^n}{n!}R_s(D^n(f))$ with coefficients in $\cA^{\GG_a}$. 
\end{proof}

\begin{remark}
There are two differences between relative Reynolds operators for fibrewise actions of linearly reductive groups and for fibrewise $\GG_a$-actions with a global slice:
\begin{enumerate}
    \item While Reynolds operators for actions of linearly reductive groups are unique, they are \emph{not unique} for additive group actions. In the setting of \cref{lemma global slices for add gp give Reynolds operator} if we have another global slice $t$, then $s-t = v \in \cA^{\GG_a}(Z)$ and the Reynolds operators $R_s$ and $R_t$ do not coincide, as $R_s(s) = 0$ and $R_t(s) = R_t(t+v) = v$. This non-uniqueness of Reynolds operators for additive group actions means that it is not possible to glue Reynolds operators when we have local slices $s_i$ over open sets $U_i$ in $Z$ which do not glue. Indeed on the intersections $U_i \cap U_j$ the slices differ by non-trivial invariant functions $b_{ij} \in \cA^{\GG_a}(U_i \cap U_j)$.
    \item Relative Reynolds operators $R_s$ associated to a global slice of a fibrewise $\GG_a$-action is an \emph{algebra} homomorphism. In particular, from knowing $\cA$ is sheaf of $\cO_Z$-algebras of a finite type, by applying $R_s$, we immediately deduce that $\cA^{\GG_a}$ is sheaf of $\cO_Z$-algebras of finite type, without needing to assume that $Z$ itself is of finite type.
\end{enumerate}
\end{remark}

Although one cannot glue Reynolds operators defined by local slices, one can still glue the corresponding quotients as follows.

\begin{proposition} \label{prop slice Ga quotient} 
    Suppose that $\GG_a$ acts fibrewise on an affine morphism $f\colon X \to Z$ and let $\cA_{\sli}$ denote the pullback to $Z_{\sli}$ of $\cA := f_{\ast} \cO_X$. Then $\cA_{\sli}^{\GG_a}$ is a sheaf of $\cO_{Z_{\sli}}$-algebras of finite type, and  the map of finite type affine $Z_{\sli}$-schemes $$ X_{\sli} = \rSpec_{Z_{\sli}} (\cA_{\sli}) \to \rSpec_{Z_{\sli}} ( \cA_{\sli}^{\GG_a}) =: X_{\sli} / \GG_a$$ induced from the inclusion $\cA^{\GG_a} \subseteq \cA$ is a Zariski-locally trivial $\GG_a$-bundle. 
\end{proposition}
\begin{proof} 
For ease of notation, we shall assume $Z_{\sli} = Z$, as if not we can first base change to $Z_{\sli}$. To show that $\mathcal{A}^{\GG_a}$ is of finite type, it suffices to show for an open affine cover $\{U_i = \Spec B_i\} $ of $Z$ that $B_i = \mathcal{O}_Z(U_i) \rightarrow \mathcal{A}^{\GG_a}(U_i)$ is of finite type. By our assumption that $Z = Z_{\sli}$, we can choose an open cover $\{U_i\}$ and slices $s_i \in \cA(U_i)$. We can further refine this open cover to assume that the open sets $U_i$ are affine and write $A_i = \mathcal{A}(U_i)$. Then $s_i$ is a global slice for the locally nilpotent derivation $D_i \colon A_i \rightarrow A_i$ and $s_i$ defines a Reynolds operator $R_{s_i} \colon A_i \rightarrow A_i^{\GG_a}$ and  the invariant algebra $A_i^{\GG_a}$ is finitely generated by Lemma \ref{lemma global slices for add gp give Reynolds operator}. This proves that $\cA^{\GG_a}$ is of finite type. 

On each open subset $V_i := \Spec A_i \subseteq X$, we can write $A_i \cong A_i^{\GG_a}[s_i]$, and $V_i = \Spec A_i \rightarrow W_i:= \Spec A_i^{\GG_a}$ is a trivial $\GG_a$-bundle by Lemma \ref{lemma global slices for add gp give Reynolds operator}. The affine schemes $W_i$ form an open cover of $Y = \rSpec_Z \cA^{\GG_a}$ and we claim that the trivial bundles $V_i \rightarrow W_i$ glue to a Zariski-locally trivial $\GG_a$-bundle $X \rightarrow Y$. Indeed writing $U_{ij}:=U_i \cap U_j = \Spec A_{ij}$, we have $s_i|_{U_{ij}} - s_j|_{U_{ij}} = b_{ij} \in A_{ij}^{\GG_a}$. On the triple intersections note that $b_{ij} + b_{jk} = b_{ik}$. Thus the functions $b_{ij}$ determine transition functions $W_{ij} = \Spec A_{ij}^{\GG_a} \rightarrow \AA^{1} \simeq \GG_a$ satisfying the necessary cocycle condition to glue these trivial $\GG_a$-bundles.
\end{proof} 

\cref{prop slice Ga quotient} shows that we can construct a quotient of a fibrewise $\GG_a$-action on an affine morphism $f \colon X \to Z$ provided we restrict to the slice locus in $X$. The difficulty lies in determining the slice locus, which is usually impossible in practice. In $\S$\ref{subsec:quotientsbygradedga} we will see that under two additional assumptions the slice locus can be described explicitly.

\subsection{$\GG_a$-quotients for graded $\GG_a$-actions}\label{subsec:quotientsbygradedga}

We can explicitly describe the slice locus when there is a graded fibrewise action of $\GG_a \rtimes \GG_m$ on an affine morphism $f \colon  X = \rSpec_Z \cA \rar Z$. Recall that in this situation the $\GG_m$-action induces a $\ZZ$-grading on $\cA$ that is concentrated in non-positive degrees $\cA = \oplus_{i \leq 0} \cA_i$ with fixed subscheme $X^{\GG_m} = \rSpec_Z \cA_0 = Z$. Moreover, the $\Gm$ acts by conjugation on $\Lie \Ga$ with a single positive weight, say $\ell >0$. The corresponding locally nilpotent derivation $D : \cA \rightarrow \cA$ shifts the $\GG_m$-graded pieces by $l$ so $D(\cA_i) \subseteq \cA_{i+ l}$.

\begin{definition}
    For a graded fibrewise action of  $\GG_{a} \rtimes_l \GG_{m}$ on a morphism $f \colon X \rightarrow Z$, we denote the locus where the $\GG_a$-stabilisers are zero-dimensional in $Z$ by 
    \[ Z_0 :=  \{ z \in Z : \dim \Stab_{\GG_a}(z) = 0 \}\]
    and let $X_0 := f^{-1}(Z_0)$. 
\end{definition}

If the $\GG_m$-action on $f \colon X \rightarrow Y$ is graded, then $ f^{-1}(Z_0) \subseteq \{ x \in X : \dim \Stab_{\GG_a}(x) = 0 \} $. Indeed, one can show that $\dim \Stab_{\GG_a}(x) \leq \dim \Stab_{\GG_a}(f(x))$ due to the fact that $f(x)$ lies in the $\GG_{a} \rtimes_l \GG_{m}$-orbit closure of $x$ and also use that $f(x)$ is $\GG_m$-fixed. The following result is a relative version of \cite[Proposition 7.3]{Berczi2016}.

\begin{proposition}[Quotients by graded $\GG_a$-actions on affine morphisms]\label{prop quot graded Ga}
Let $f \colon X \rightarrow Z$ be an affine morphism of schemes with a graded fibrewise action of $\GG_{a} \rtimes_l \GG_{m}$. Then 
\[ X_{\sli} = f^{-1}(Z_0).\]
If moreover all points in $X^{\GG_m} = Z$ have trivial $\GG_a$-stabilisers, then $X_{\sli} = X$ and the morphism of finite type affine $Z$-schemes $X \rightarrow X/\GG_{a} = \rSpec_{Z}(\cA^{\GG_a})$ is a Zariski-locally trivial $\GG_a$-bundle.
\end{proposition}
\begin{proof}
By \cref{prop slice Ga quotient}, we have $Z_{\sli} \subseteq Z_{0}$ and so $X_{\sli} \subseteq f^{-1}(Z_{0})$. To show the opposite containment, by replacing $X \rightarrow Z$ by $f^{-1}(Z_{0}) \rightarrow Z_{0}$, we may as well assume that $X = X_0 :=f^{-1}(Z_{0})$ and $Z = Z_0$ and in this case prove that $X = X_{\sli}$.

The idea is to show that $D : \cA_{-l} \rightarrow \cA_0$ is surjective, so that in particular $1 \in \cA_0(Z) = \cO_Z(Z)$ is jointly hit by local sections $s_i \in \cA_{-l}(U_i)$ with $D(s_i) = 1$ where $U_i$ is an open cover of $Z$; thus each $s_i$ gives a slice over $U_i$. To prove this, consider
\[ \cI:= \oplus_{i < 0} \cA_i \oplus D(\cA_{-l}), \]
which one can check is an ideal of $\cA$. Moreover, this ideal is preserved by the $\GG_m$-action as the action on $D(\cA_{-l}) \subseteq \cA_0$ is trivial and also preserved by the additive action as $D(\cA_i) \subseteq \cA_{i+l}$ and $\cA_i = 0$ for $ i > 0$. If $\cI \subseteq \cA$ is a proper ideal, then it must be contained in a maximal ideal. Any maximal ideal $\fm_x$ of $\cA$ containing $\cI$ is given by a closed point $x \colon \Spec k \rightarrow X$, where $\fm_x$ is the kernel of the composition of  $f^* \colon \cA \rightarrow \cO_Z$ with $\cO_Z \twoheadrightarrow \cO_Z/\fm_z$ where $z := f \circ x \colon \Spec k \rightarrow Z$ denotes the corresponding closed point in $Z$. Moreover, as $\cI$ is preserved by the action, the ideal $\fm_x$ must also be preserved by the action, which means that $x \in X^{\GG_m} = Z$ is $\GG_a$-fixed and contradicts the assumption that $Z= Z_0$.

Finally, if $X^{\GG_m} = Z$, then $Z = Z_0$ and $X_{\sli}=X$, and the remaining claims follow from \cref{prop slice Ga quotient}.
\end{proof}

\subsection{Graded unipotent quotients} \label{subsubsec:graded unipotent quotients}

In this section we first generalise the above result to higher-dimensional unipotent groups by quotienting in stages by additive group actions using the results above. We then explain how the stabiliser assumptions can be weakened to obtain a more general result following the ideas of Qiao \cite{Qiao2022, Qiao2025}. 

We first prove a relative version of \cite[Proposition 7.4]{Berczi2016}.

    \begin{proposition}[Quotients by graded unipotent group actions]\label{thm:quotientsbyunip}
Let $U$ be a unipotent group and consider a graded fibrewise action of $\hU = U \rtimes \GG_{m}$ on an affine morphism $f \colon X = \rSpec_Z \cA \rightarrow Z$. If all points in $X^{\GG_m} = Z$ have trivial $U$-stabilisers, then there is a morphism of finite type affine $Z$-schemes $X \rightarrow X/U = \rSpec_{Z}\cA^U$ which is a Zariski-locally trivial $U$-bundle.
    \end{proposition} 
    \begin{proof} 
We fix a chain of subgroups $\{ \mathrm{Id} \} =U^{(0)} \subseteq U^{(1)} \subseteq \cdots \subseteq U^{(r)} = U$ such that $\GG_m$ acts on $U_i:=U^{(i)}/U^{(i-1)} \cong \GG_a$ with weight $l_i >0$. We will inductively take quotients one $\GG_a$ at a time by applying \cref{prop quot graded Ga}. Let us write $\cA^{(i)}:= \cA^{U^{(i)}}$ and $X_i:= \rSpec_Z \cA^{(i)}$; then the inclusions  $\cA^{(i)} \subseteq \cA$ and $\cA^{(i)} = (\cA^{(i-1)})^{U_i} \subseteq \cA^{(i-1)}$ induce morphisms of affine $Z$-schemes (a priori not necessarily of finite type) $q_i \colon X \rightarrow X_i$ and $\pi_i \colon X_{i-1} \rightarrow X_i$ such that $q_i = \pi_i \circ q_{i-1}$. Since the action on $f$ is graded, the $\GG_m$-action induces a grading $\cA = \oplus_{j \leq 0} \cA_j$ supported in non-positive degrees with $\cA_0 = \cO_Z$. There is an induced $\GG_m$-action on each $X_i$ coming from the grading induced on $\cA^{(i)} \subseteq \cA$. Thus the action on each $f_i \colon X_i \rightarrow Z$ is also graded, as $\cA^{(i)} = \oplus_{j \leq 0} \cA^{(i)}_j$ and $\cA^{(i)}_0 = \cO_Z$. 

For the base case, we can directly apply \cref{prop quot graded Ga} to the first group $U^{(1)} \cong \GG_a$ to obtain that $\cA^{(1)}$ is of finite type and there is a Zariski-locally trivial $U^{(1)}$-bundle $q_1 \colon X \rightarrow X_1$. Now suppose by induction that we have constructed a Zariski-locally trivial $U^{(i)}$-bundle $q_{i} \colon X \rightarrow X_{i}$. As observed above, the action of $U_{i+1} \rtimes_{l_{i+1}} \GG_m$ on $f_i \colon  X_{i} \rightarrow Z$ is graded and we claim that all points in $X_{i}^{\GG_m} = Z$ have trivial $U_{i+1}$-stabilisers, so that we can apply \cref{prop quot graded Ga} to conclude that $X_{i+1}$ is of finite type over $Z$ and $\pi_{i+1} \colon X_i \rightarrow X_{i+1}$ is a Zariski-locally trivial $U_{i+1}$-bundle. Indeed if $z \in X$ has non-trivial stabiliser for $U_{i+1} = U^{(i+1)}/U^{(i)}$, then there exists $u_{i+1} \in U^{(i+1)} \setminus U^{(i)}$ such that $u_{i+1}U^{(i)}$ fixes $z$ in $X_i = X/U^{(i)}$. This means we have $u_{i+1}z = u_iz$ in $X$ for some $u_i \in U^{(i)}$, and so would give a non-trivial element $u_i^{-1}u_{i+1}$ in the $U$-stabiliser of $z$ in $X$, which contradicts the assumption that $U$-stabilisers in $X^{\GG_m}$ are trivial. Since $q_{i+1} = \pi_{i+1} \circ q_i$, the tower of the $U^{(i)}$-bundle $q_i$ and the $U_{i+1}$-bundle $\pi_{i+1}$ is a $U^{(i+1)}$-bundle (see \cite[Lemma 3.20]{Berczi2016}).
    \end{proof}

    One can alternatively prove the above theorem using \cite[Theorem 3.10]{Greuel1993}, which allows quotienting by abelian unipotent groups that are not necessarily one-dimensional, but on which the grading $\Gm$ acts with a single weight; this approach is employed in \cite{Qiao2022, Qiao2025}.

We can first base change to the locus of points $Z$ with trivial $U$-stabilisers, if it is non-empty.

\begin{corollary}
    For a unipotent group $U$ and a graded fibrewise action of  $U \rtimes \GG_{m}$ on an affine morphism $f \colon X = \rSpec_Z \cA \rightarrow Z$, provided the open subset  
    \[ Z_0 := Z_0(U)= \{ z \in Z : \dim \Stab_{U}(z) = 0 \} \stackrel{i}{\hookrightarrow} Z,\]
    is non-empty, there is a Zariski-locally trivial $U$-bundle $X_0 := f^{-1}(Z_0) \rightarrow X_0/U= \rSpec_{Z_0}i^*\cA^U$.
\end{corollary}

If the generic $U$-stabiliser is non-trivial (and thus positive dimensional), then it is still possible to construct quotients by restricting to certain open subsets of $Z$ with well-behaved $U$-stabilisers. However, when the generic $U$-stabiliser is positive dimensional, precisely how this open subset of $Z$ is defined depends on whether or not $Z$ is reduced. Indeed, if $Z$ is not reduced, then there can be unipotent actions which are set-theoretically trivial, but not scheme theoretically trivial and so the appropriate invariant algebras may not be finitely generated (see \cite[Example 2.29]{Hoskins2021}). The case when $Z$ is reduced is described by work of B\'{e}rczi and Kirwan (although it is phrased in an absolute setting) and the more general case is described by work of Qiao \cite{Qiao2022, Qiao2025}. 

Let us outline the approach of B\'{e}rczi and Kirwan when $X$ is reduced. First suppose $U$ is abelian; then to find complementary subgroups to the stabiliser subgroups in $U$, it suffices to find a complementary subspace in the Lie algebra $\mathfrak{u}$. By first base changing to an open subset, we can assume all $U$-stabilisers in $X$ have dimension $d$, then taking the stabiliser gives a morphism $X \rightarrow \mathrm{Gr}(d,\mathfrak{u})$ and being complementary to a given codimension $d$ subgroup $U' \subseteq  U$ is an open condition in this Grassmannian, so one can work over corresponding open subsets $X_{U'}$ and take the quotient by the free action of $U'$ on $X_{U'}$. By varying $U'$ one can obtain a quotient of the $U$-action on $X$. To generalise this idea to non-abelian unipotent groups, they fix a sequence of normal subgroups whose quotients are abelian and instead restrict to the locus in $Z$ where the stabilisers for each of these subgroups is minimal (see \cref{rmk: descr Zcirc for Z reduced} below), then iteratively take quotients by the abelian subquotients using the strategy sketched above.

If $X$ is non-reduced (and for simplicity let's assume that $U$ is abelian, as otherwise we can consider a filtration as above), then one may not necessarily get a morphism to the Grassmannian as above, as this requires the stabiliser subsheaf of $\mathfrak{u} \otimes \cO_X$ to be locally free of rank $d$ (on some open). Indeed knowing the $U$-stabilisers of closed points have dimension $d$ (on some open) is not sufficient to conclude it is locally free (on this open) if $X$ is non-reduced. Therefore, Qiao instead  restricts to the locus where the appropriate sheaves encoding stabilisers are locally free (see \cite[Prop.\ 2.4]{Qiao2025}). More precisely, he considers the infinitesimal (co)action
\[ \varphi \colon \Omega_X \rightarrow \mathfrak{u}^* \otimes \cO_X\]
and restricts to an open set where the cokernel $\cQ$ of $\varphi$ is locally free (see \cref{def Zcirc for Z arbitrary} below).

Let us now describe Qiao's approach in detail in our relative perspective. Let $U$ be a unipotent group and suppose we have a graded fibrewise action of $U \rtimes \GG_m$ on an affine morphism $f \colon X \rightarrow Z$. The $\GG_m$-conjugation action on $U$ induces a grading on its Lie algebra  $\mathfrak{u} = \bigoplus_{i=1}^r \mathfrak{u}_{w_i}$ with weights $w_1 >w_2>\dots > w_r >0$. Writing $U^{(j)} = \exp( \bigoplus_{w\geq w_j} \mathfrak{u}_w)$, we obtain a filtration of $U$ by normal subgroups \begin{equation} \{ \mathrm{Id} \} \subseteq U^{(1)} \subseteq U^{(2)} \subseteq \cdots \subseteq U^{(r)} = U
\label{filtration} \end{equation} such that the Lie algebra of each subquotient $U_i = U^{(i)}/U^{(i-1)}$ has a single $\GG_m$-weight. Let $\mathfrak{u}_{(i)}$ and $\mathfrak{u}_i$ denote the Lie algebras of $U^{(i)}$ and $U_i$ respectively. We consider the infinitesimal action \[\varphi \colon \Omega_X \rar \mathfrak{u}^*\otimes \cO_X\] where $\Omega_X$ denotes the sheaf of differentials on $X$ and compose it with the projections $\fu^* \twoheadrightarrow \fu_{(i)}^*$. If $\cK_{(i)}$ and $\cQ_{(i)}$ denote the kernel and cokernel of $\Omega_X \rar \mathfrak{u}^*\otimes \cO_X \twoheadrightarrow \fu_{(i)}^* \otimes \cO_X$, then as $\fu_i =\fu_{(i)}/\fu_{(i-1)}$ we obtain an induced morphism
\[\varphi_i : \cK_{i-1} \rar \mathfrak{u}_i^* \otimes \cO_X \]
and we let $\cQ_i := \operatorname{coker} \varphi_i$, so that $\cQ_{(i-1)} \cong \cQ_{(i)}/\cQ_i$.
 
\begin{definition}\label{def Zcirc for Z arbitrary}
 In the above situation, we define \[Z^{\circ} = Z^{\circ}(U) := \{z \in Z \mid \forall i=1,\dots r,\  \text{$\cQ_i$ is a free $\cO_X$-module in a neighbourhood of $z$ in $ X$}\} \] and $X^\circ: = f^{-1}(Z^\circ)$. 
\end{definition}

The complement of this open set can be given a scheme structure defined using appropriate Fitting ideals (see \cite[Lemma 05P8]{stacks-project} and \cite[Section 3.3]{Qiao2022}). In fact, asking that all the sheaves $\cQ_i$ are locally free is equivalent to asking that the sheaves $\cQ_{(i)}$ are all locally free (see \cite[Corollary 2.9]{Qiao2022}). In particular, we note that $\cQ_i$ is locally free over $X^\circ$.

\begin{remark}[Description of $Z^{\circ}$ when $Z$ is reduced]\label{rmk: descr Zcirc for Z reduced}
If $Z$ is reduced, there is a simpler description of $Z^{\circ}$ in terms of stabilisers at closed points given in \cite[Corollary 2.9]{Qiao2022}. We then have
\[Z^\circ = \{ z \in Z \mid \forall i=1,\dots r,\  \dim \Stab_{U^i}(z) = \min_{x\in X} \dim \Stab_{U^i}(x)\}. \] If $Z$ is only generically reduced, the above locus is contained in $Z^{\circ}$. This containment may be strict, but we can still obtain $U$-quotients of the above locus. 
   
\end{remark}

    We note that we may have $Z^\circ = \emptyset$; for example,  if the minimal $U$-stabiliser dimension on $Z$ is higher than that on $X$. In this case, one could perform something similar to the \lq Modification II\rq\ of \cite{Qiao}, but we will not pursue that direction in this work.    
    
    The following result is a relative framing of \cite[Theorem 2.22]{Qiao2022}; however, we provide a self-contained and streamlined proof.

    \begin{theorem}\label{theorem U quotient with varying stabilisers}
Let $f \colon X = \rSpec_Z \cA \rightarrow Z$ be an affine morphism of schemes with a graded fibrewise action of $\hU = U \rtimes \GG_{m}$ with $U$ unipotent. Then the $U$-invariants in $\cA^\circ:= \cA|_{Z^\circ}$ are finitely generated as a sheaf of $\cO_{Z^{\circ}}$-algebras. Moreover, there is a geometric $U$-quotient $X^\circ = f^{-1}(Z^\circ) \rightarrow X^\circ/U = \rSpec_{Z^\circ}(\cA^\circ)^U$ which is a Zariski locally trivial affine space fibration. There is a cover of $Z^\circ$ by open affines such that over each of these this fibration is trivial.
    \end{theorem} 
\begin{proof}
We follow the ideas of \cite{Qiao2022, Qiao2022b, Qiao2025}. As $U$ is connected, it acts separately on each connected component of $X$, so we may assume that $Z$ is connected. Furthermore, for ease of notation, we will assume that $Z = Z^\circ$, as otherwise we first base change to $Z^\circ$.

Let us first suppose that $U$ is abelian and the $\GG_m$-action on $\fu$ has a single positive weight, as the more general case will follow from this by induction. In this case, we just write $\cQ$ for the cokernel of $\varphi$, and by our assumption that $Z^\circ = Z$ and that $Z$ is connected, we have that $\cQ$ is a locally free $\cO_X$-module of constant rank $d$. The surjection $\fu^* \otimes \cO_X \twoheadrightarrow \cQ$ is thus equivalent to a morphism $h \colon X \rightarrow \text{Gr}(\fu^*,d).$ For each $d$-dimensional quotient $\fu \twoheadrightarrow \fs$, we consider the open subset  $\text{Gr}(\fu^*,d)_{\fs}$ of quotients $\fu^* \twoheadrightarrow W$ such that $\fs^* \hookrightarrow \fu^* \twoheadrightarrow W$ is an isomorphism. Furthermore, as $U$ is abelian, the Lie subalgebra $\fu'_\fs:= \ker(\fu \twoheadrightarrow \fs) \subseteq \fu$ lifts to a subgroup $U'_\fs \subseteq U$, which we refer to as the complementary subgroup to the stabiliser $\fs$. By construction, the $U'_\fs$-action on $X_{\fs}:=h^{-1}(\text{Gr}(\fu^*,d)_{\fs})$ is free. 

In fact, we can cover the Grassmannian by only finitely many such subsets by picking a finite number of $d$-dimensional quotients $\fs_1, \dots, \fs_n$ of $\fu$ and we can assume we have chosen open affine subsets $Z_i$ covering $Z$ such that $X_i:=f^{-1}(Z_i) \subseteq X_{\fs_i}$ as in \cite[$\S$3.3.1]{Qiao2022b}. For each $1 \leq i \leq n$, we can now apply \cref{thm:quotientsbyunip} to the graded fibrewise action of $U'_{\fs_i} \rtimes \GG_m$ on $X_i \rightarrow Z_i$ as the $U'_{\fs_i}$-stabilisers are trivial, to obtain a trivial $U'_{\fs_i}$-bundle 
\[X_i = \Spec A_i \rightarrow Y_i:=\Spec A_i^{U'_{\fs_i}} = \Spec A_i^U,\] where the last equality holds as the induced action of $U / U'_{\fs_i}$ on $Y_i$ is trivial, as the Lie algebra of this quotient group is $\fs_i$ and we have $\fs_i^* \otimes \cO_{X_i} \cong \cQ|_{X_i}$, where $\cQ$ is the cokernel of the infinitesimal (co)action. In particular, the invariant ring $A_i^U$ is finitely generated. We can glue these finite type geometric quotients $X_i =\Spec A_i \rightarrow Y_i=\Spec A_i^U$ (which are affine and of finite type over $Z_i$) to obtain a geometric $U$-quotient $X = \rSpec_Z \cA \rightarrow \rSpec_Z \cA^U$ (also affine and of finite type over $Z$), which is Zariski locally trivial with fibre an affine space of dimension $\dim U -d$. This completes the proof when $U$ is abelian with a single $\GG_m$-weight. 

The more general case where $U$ is only filtered by subgroups whose subquotients are abelian with a single $\GG_m$-weight is then obtained by induction, with the above argument giving both the base case and inductive step. We omit the details as the inductive process is completely analogous to the one performed in \cref{thm:quotientsbyunip}.
\end{proof}

 \section{Relative GIT for graded actions of non-reductive groups} \label{sec:relative GIT non-reductive}

Given a graded equivariant action of a homomorphism $\varphi \colon G \to H$ on an affine morphism $f \colon X \to Z$ where $G$ is non-reductive and $H$ is reductive, the aim of this section is to construct a relative quotient using the constructions defined in $\S$\ref{subsec: def rel GIT reductive}. The grading by a copy of the multiplicative group $\GG_m$ is essential to quotient by the action of the unipotent radical $U$ of $G$. This section is split into two parts depending on whether the grading is internal or external.

\subsection{Internally graded actions}  \label{subsec:internallygradedactions}

In this section, we consider an internally graded equivariant action of $\varphi \colon G \to H$ on an affine morphism $f \colon X = \rSpec_Z \cA \to Z$ and assume that $H$ is reductive and we have a good $H$-quotient $q \colon Z' \rightarrow W$ of an open subset $Z' \subseteq Z$. By \cref{reductive H}, since $H$ is reductive we know that $H \cong G/U$ for $U$ the unipotent radical of $G$. Writing $G = U \rtimes R$ for $R$ a Levi subgroup of $G$, we may thus equivalently assume that $H = R$ and that $\varphi$ is the natural quotient map $G = U \rtimes R \to R$. Thus $U$ acts fibrewise on $f$ and there is an internal $\GG_m\subseteq G$ which grades this action, so we can apply the results of \cref{sec:rel quotients for unipotent actions}.

In order to take a quotient of the fibrewise action of the unipotent radical $U$, we will assume that the locus $Z^\circ \subseteq Z$ where the $U$-stabilisers are well-behaved contains $Z'$ and then base change to $Z'$ and construct a relative quotient by first taking a fibrewise $U$-quotient using the grading $\GG_m$ and then reducing to an equivariant action of $R$ and applying our results for relative reductive GIT quotients. 

For simplicity, we first assume that $Z = Z' = Z^\circ$ and prove a weaker version of \cref{main theorem}. In this case, following \cref{definition reductive GIT quotients}, we define the \emph{relative $\rho$-twisted affine NRGIT quotient} for the action of $\varphi$ on $f$ with respect to $q$ and a character $\rho \colon G \rightarrow \GG_m$ to be the rational map of $W$-schemes 
\begin{equation}\label{rational map for graded relative NRGIT quotient}
X \dashrightarrow X \gitq_{\hspace{-2pt} q,\rho \hspace{2pt}}^f G : = \rProj_W \bigoplus_{n \geq 0} (q_{\ast} \cA)^G_{\rho^n}
\end{equation}
induced by the inclusion of sheaves $(q_{\ast} \cA)^G_{\rho^n} \subseteq q_{\ast} \cA$. The \emph{$\rho$-twisted (semi)stable locus in $X$ relative to $f$ and $q$} is defined as in \cref{definition reductive GIT quotients} and denoted by $X^{(s)s}(f;q,\rho)$.

 \begin{theorem} \label{theorem nonred quotients graded equiv optimal assumptions on Z}
 Let $G = U \rtimes R$ be a linear algebraic group with unipotent radical $U$ and suppose we have an internally graded action of $G \twoheadrightarrow R$ on an affine morphism $f \colon X = \rSpec_Z \cA \to Z$. Additionally suppose that $Z^{\circ} = Z$ and $q \colon Z \to W$ is a good $R$-quotient. For any character $\rho \colon G \rightarrow \GG_m$, the following statements hold.
 \begin{enumerate}  [(i)]
     \item The sheaf $(q_{\ast} \cA)^{G,\rho}$ of $\rho$-twisted semi-invariants is a finitely generated sheaf of $\cO_{W}$-algebras;
     \item $X^{ss}(f;q, \rho)$ is the domain of definition of the relative $\rho$-twisted GIT quotient $X \dashrightarrow X \gitq^{f}_{\hspace{-2pt} q, \rho \hspace{2pt}} G$; 
    \item $X^{ss}(f;q,\rho) \to X \gitq^{f}_{\hspace{-2pt} q, \rho \hspace{2pt}} G $ is a good quotient and restricts to a geometric quotient on $X^{s}(f;q,\rho)$;
    \item we have $\rSpec_W (q_*\hspace{-2pt}\cA)^{G}=W$ and thus the structure morphism
    $X \gitq^{f}_{\hspace{-2pt} q, \rho  \hspace{2pt}} G \rightarrow W$ is projective. 
    \item there is an equality $$ X^{(s)s}(f;q,\rho) : = \{ x \in X \ | \  \text{$\langle \rho, \lambda \rangle (\geq) 0$ for all $\lambda\colon \GG_m \to G$ such that $\lim_{t \rightarrow 0} \lambda(t) \cdot x$ exists}\}.$$
 \end{enumerate}
 \end{theorem}
\begin{proof}
    For (i), as by assumption $Z = Z^\circ$, we know that $\cA^U$ is a finitely generated sheaf of $\cO_Z$-algebras by \cref{theorem U quotient with varying stabilisers}, and thus so is $q_{\ast} (\cA^U)$ as a sheaf of $\cO_W$-algebras. Since $R$ is reductive, the sheaf of semi-invariants $(q_{\ast} (\cA^U))^{R,\rho}$ is finitely generated. Moreover $(q_{\ast} \cA)^U = q_{\ast} (\cA^U)$ and thus $(q_{\ast} (\cA^U))^{R,\rho} =(  q_{\ast} \cA)^{G,\rho}$ and so the result follows. 

    Part (ii) is immediate, as $X^{ss}(f;q, \rho)$ is by definition the domain of definition of this rational map. For the third statement, we will factor this rational map as a good $U$-quotient followed by a good $R$-quotient (resp. a geometric $R$-quotient of the stable locus). As above, we can write the sheaf of $\rho$-twisted semi-invariants as follows
\[ \cA^{G,\rho}:= \bigoplus_{n \geq 0} (q_{\ast} \cA)^G_{\rho^n} = (q_{\ast} \cA^U)^{R,\rho} \]
and thus the rational map \eqref{rational map for graded relative NRGIT quotient} factors as
\[ X = \rSpec_Z \cA \stackrel{q_U}{\longrightarrow} Y:=X/U= \rSpec_Z \cA^U  = \rSpec_W q_*(\cA^U) \dashrightarrow  Y \gitq^h_{q,\rho} R = X \gitq^f_{q,\rho} G\]
where $h : Y = X/U \rightarrow Z$ is also affine. 
The domain of definition of the second rational map is the locus $Y^{R,ss}(h;q,\rho)= Y^{R,ss}(q \circ h; \rho)$ and thus the domain of definition of the rational map \eqref{rational map for graded relative NRGIT quotient} is 
\begin{equation}\label{eq first step to NR HM} 
X^{ss}(f;q,\rho) = q_U^{-1}(Y^{R,ss}(h;q,\rho))
\end{equation}
and we will use this to lift the Hilbert--Mumford description of $Y^{R,ss}(q \circ h;\rho)$ to $X^{ss}(f;q,\rho)$ in part (v). In particular the rational map factors as the composition of a geometric $U$-quotient by a good $R$-quotient and thus is a good $G$-quotient. Furthermore, by  \cref{thm twisted affine}, the structure map admits a factorisation \[X \gitq^f_{q,\rho} G =Y \gitq^{ h}_{q,\rho} R = Y \gitq^{q \circ h}_\rho R \longrightarrow  \rSpec_W (q_*\cA)^G = Y \gitq^{q \circ h} R \longrightarrow W\] where the first morphism is projective and the second morphism is affine. 

For part (iv), we will prove that the affine morphism $\rSpec_W (q_*\cA)^G = Y \gitq^{q \circ h} R \rightarrow W$ is the identity on $W$. Since we are working relative to $W$, we can cover $W$ by open affine subsets over which everything else is affine (as $q$ and $f$ are both affine). Thus we can reduce to the case where $W$ is affine, and thus so are $X$ and $Z$. In this case, we have the following diagram
\[ X = \Spec_k A \stackrel{f}{\longrightarrow} Z= \Spec_k B \stackrel{q}{\longrightarrow} W =\Spec_k C \]
where the graded $\GG_m$-action induces a decomposition $A = \oplus_{i \leq 0} A_i$ such that $B = A_0$, and as $q$ is a good $R$-quotient we have $C = B^R$. Since the action is graded, we have $A^G = B^R = C$ by \cref{properties of graded actions and linearisations}. This proves that $\rSpec_W (q_*\hspace{-2pt}\cA)^{G}=W$.

It remains to prove the final statement which is a non-reductive Hilbert--Mumford criterion. Our starting point is the equality \eqref{eq first step to NR HM} and it remains to prove that
\begin{equation} \label{Eq HM criterion reduction}
q_U^{-1} (Y^{R,(s)s}(h; q, \rho)) = \bigcap_{u \in U} uX^{R,(s)s}(\rho), 
\end{equation}
where by the reductive Hilbert--Mumford criterion, we have $x \in X^{R,(s)s}(\rho)$ if and only if $\langle \rho,\lambda \rangle  \: (\geq) \: 0  $ for all 1-PS $\lambda \colon \GG_m \rightarrow R$ such that $\lim_{t \rightarrow 0} \lambda(t) \cdot x$ exists in $X$. A self-contained proof of the non-reductive Hilbert--Mumford criterion will appear in \cite{Jackson2026}, but our relative approach allows us to give a short proof using an idea suggested by T. Vernet. Since we can check this locally and locally $q_U$ is constructed by taking slices over open affine subsets in $Z'$ (see \cref{theorem U quotient with varying stabilisers}), we may assume the $U$-quotient has the form $q_U \colon X \cong U' \times Y \rightarrow Y$, where $U' \subseteq  U$ is a subgroup that acts freely on $X$ (see the proof of \cref{theorem U quotient with varying stabilisers}) and so we can apply \cref{lemma existence of limits under quotients by unipotents with slices} below to deduce that \cref{Eq HM criterion reduction} holds, which completes our proof of the non-reductive Hilbert--Mumford criterion.
\end{proof}

Various non-reductive Hilbert--Mumford criteria have been stated and proved in the literature. In the projective setting, it was proved in the case \lq semistability equals stability' in \cite[Theorem 4.28]{Berczi2023} (see also \cite[Theorem 2.30]{Hoskins2021}). A proof without this condition in the projective setting was circulated to experts, and appears in \cite{Jackson2026}. A similar proof was subsequently found independently, and appears in \cite[Theorem 2.9.4]{Wiaterek}. To complete the proof of our theorem, we use the following lemma, for which we thank T. Vernet.

\begin{lemma}\label{lemma existence of limits under quotients by unipotents with slices}
For $G = U \rtimes R$ acting on a scheme $X$, suppose the $U$-action admits a slice. Let $q_U \colon X \rightarrow Y = X/ U$ denote the quotient and let $\iota \colon Y \hookrightarrow X$ denote the inclusion of the slice. Then for $y \in Y$ and a 1-PS $\lambda \colon \GG_m \rightarrow R$, the following are equivalent:
\begin{enumerate}
    \item $\lim_{t \rightarrow 0} \lambda(t) \cdot y$ exists in $Y=X/U$,
    \item there exists $x \in q_U^{-1}(y) = U \cdot \iota(y)$ such that $\lim_{t \rightarrow 0} \lambda(t) \cdot x$ exists in $X$.
\end{enumerate}
\end{lemma}
\begin{proof}
By definition of the induced $\lambda$-action on the quotient $Y = X/U$, we have $\lambda(t) \cdot y = q_U(\lambda(t) \cdot \iota(y))$ so that there is an element $u_t \in U$ such that
\[ \iota (\lambda(t) \cdot y)  = u_t \cdot (\lambda(t) \cdot \iota(y)).\]
This defines a twisted homomorphism (i.e.\ a $U$-valued 1-cocycle on $\GG_m$), as $ u_{tt'} = u_t  (\lambda(t) \cdot u_{t'})$.
Since $H^1(k,U) = 0$ (for example, this follows from \cite[Prop.\ 15.3]{Milne}, which is Prop.\ 16.3 in the online notes), this 1-cocycle is principal, so there is an element $u \in U$ such that $u_t = u^{-1} (\lambda(t) \cdot u)$. Consequently
\[ u \cdot \iota(\lambda(t) \cdot y) = \lambda(t) \cdot (u \cdot \iota(y))\]
and so $\lim_{t \rightarrow 0} \lambda(t) \cdot y$ exists if and only if $\lim_{t \rightarrow 0} \lambda(t) \cdot x$ exists, for $x = u \cdot \iota(y)$.
\end{proof}

To interpret the closed points of $X \gitq^{f}_{\hspace{-2pt} q, \rho \hspace{2pt}} G$, we recall that two orbits in $X^{ss}(f;q,\rho)$ are S-equivalent if their orbit closures meet in $X^{ss}(f;q,\rho)$.

\begin{proposition}\label{prop closed points of rel NRGIT quotient}
Under the assumptions of \cref{theorem nonred quotients graded equiv optimal assumptions on Z}, the following statements hold.
\begin{enumerate}
\item Every $G$-orbit in $X^{ss}(f;q,\rho)$ contains a unique closed orbit in its closure in $X^{ss}(f;q,\rho)$;
\item The closed points of $X \gitq^{f}_{\hspace{-2pt} q, \rho \hspace{2pt}} G$ are in bijection with the closed $G$-orbits in $X^{ss}(f;q,\rho)$ or equivalently S-equivalence classes of orbits in $X^{ss}(f;q,\rho)$;
\item A $G$-orbit is closed in $X^{ss}(f;q,\rho)$ if and only if  it is closed under all flows by 1-PSs of $G$;
\item Let $x \in X^{ss}(f;q,\rho)$ and $\lambda \colon \GG_m \rightarrow G$ be a 1-PS such that $\lim_{t \rightarrow 0} \lambda(t) \cdot x $ exists in $X$; then this limit exists in $X^{ss}(f;q,\rho)$ if and only if $\langle \rho,\lambda \rangle = 0$.
\end{enumerate}
\end{proposition}
\begin{proof}
The first two statements hold as $X^{ss}(f;q,\rho) \rightarrow X \gitq^{f}_{\hspace{-2pt} q, \rho \hspace{2pt}} G$ is a good quotient. For Part (3), we claim that $G \cdot x$ is closed in $X^{ss}(f;q,\rho)$ if and only if the $R$-orbit of $ u\cdot x$ is closed in $X^{ss}(f;q,\rho)$ for all $u \in U$. The latter condition is equivalent to the orbit being closed under all flows along 1-PS of $G$, as the orbit closure for the reductive group $R$ can be tested using 1-PSs by the Fundamental Theorem of GIT \cite[Theorem 1.4]{Kempf1978}. To prove the claim, we note that $G \cdot x$ is closed in $X^{ss}(f;q,\rho)$ if and only if $R \cdot y$ is closed in $Y^{R,ss}(h;q,\rho)$, where $y:=q_U(x)$. Since 
$R \cdot y$ is closed in $Y^{R,ss}(h;q,\rho)$ if and only if this orbit is closed under flows along 1-PS of $R$, we see that $G \cdot x \subseteq X^{ss}(f;q,\rho)$ is not closed if and only if there is a 1-PS $\lambda \colon \GG_m \rightarrow R$ such that $y \neq \lim_{t \rightarrow 0} \lambda(t) \cdot y $ exists in $ Y^{R,ss}(h;q,\rho)$ and is not equal to $y$. By \cref{lemma existence of limits under quotients by unipotents with slices}, this is equivalent to the existence of a 1-PS $\lambda \colon \GG_m \rightarrow R$ and $u \in U$ such that for $x' = u \cdot x$ we have $ \lim_{t \rightarrow 0} \lambda(t) \cdot x'$ exists in $X^{ss}(f;q,\rho)$ and is not equal to $x'$. Hence in $X^{ss}(f;q,\rho)$, the $G$-orbit of $x$ is not closed if and only if the $R$-orbit of $u \cdot x$ is not closed for some $u \in U$. In particular, this proves (3). Part (4) holds as the limit point is fixed by $\lambda$ and so this limit is semistable if and only if $\langle \rho,\lambda \rangle =0$.
\end{proof}

If we only have a $H$-quotient $q \colon Z \rightarrow W$ of an open subset $Z' \subseteq Z$, then we can apply the above result to the base change $f' \colon X' \rightarrow Z'$ of $f$ over $Z'$. We can apply \cref{theorem nonred quotients graded equiv optimal assumptions on Z} provided we assume that $Z' \subseteq Z^\circ$. In this case, we obtain the following non-reductive Hilbert--Mumford description of the (semi)stable locus.
\begin{equation}\label{eq HM descr nonred}\tag{NR-HM} X^{(s)s}(f;q,\rho) = \left\{ x \in X':=f^{-1}(Z') \left| \begin{array}{c} \text{$\langle \rho, \lambda \rangle \: (\geq) \: 0$ for all 1-PSs $\lambda  \colon \Gm \rar G$} \\ \text{such that the limit exists in $X'$} \end{array} \right. \right\}.\end{equation}

\begin{proof}[Proof of \cref{main theorem}]
This follows by applying \cref{theorem nonred quotients graded equiv optimal assumptions on Z} and \cref{prop closed points of rel NRGIT quotient} to $X' \rightarrow Z'$. 
\end{proof} 

    We should note that to really work with this criterion in practice, one would need to have a Hilbert--Mumford description of $Z' \subseteq Z$. We refer to the analogous discussion in \cref{rmk Z' admits HM descr} for further details of what can be said in that case.

\begin{remark}\label{rmk on NR HM semistability}
    In the set up of \cref{main theorem}, we make some observations about the semistable locus following this Hilbert--Mumford description.
    \begin{enumerate}
    \item Since every 1-PS in $G$ can be conjugated by an element in $U$ to have image contained in $R$ (or we can additionally conjugate by an element of $R$ so it has image in a fixed maximal torus $T \subseteq G$), we see that the (semi)stable locus (for $G$) can be computed from the semistable locus for $R$ or $T$ 
    \[ \quad  X^{G,(s)s}(f; q, \rho) = \bigcap_{u \in U} u X^{R,(s)s}(f; q,\rho) = \bigcap_{g \in G} g X^{T,(s)s}(f; q,\rho), \]
    where for a subgroup $G' \subseteq G$, we write $X^{G',(s)s}(f; q,\rho)$ to mean the locus of points in $X'$ that satisfy the same inequality as in \eqref{eq HM descr nonred} for all 1-PS in $G'$ (rather than $G$). 
        \item  As the limit of all points in $X$ under the grading subgroup $\lambda_g \colon \GG_m \hookrightarrow G$ exists, we see that we must have $\langle \rho,\lambda_g \rangle \geq 0$ for the semistable set to be non-empty. Note that if $\langle \rho,\lambda_g \rangle = 0$, then the stable locus is empty and so we will in practice want this pairing to be strictly positive. In this case, as points in $Z$ are fixed by $\lambda_g$, we see that the semistable locus is contained in $X' \setminus UZ'$ (that is all orbits that meet $Z'$ are unstable). This is precisely what we want to avoid the good $G$-quotient of (an open set in) $X'$ collapsing onto the good $R$-quotient $W$ of $Z'$.
        \item\label{stable loci not preserved by f} In the case when the good $R$-quotient $q$ is a reductive GIT quotient and $Z' = Z^{ss}$, we have by construction that $f(X^{ss}(f;q,\rho)) \subseteq Z^{ss}$. However, it is not necessarily the case that $f(X^{s}(f;q,\rho)) \subseteq Z^{s}$, as there could well be stable points in $X$ whose image in $Z$ are strictly semistable (for example, this happens in an application to moduli of quivers with multiplicities \cite{HJV}).
    \end{enumerate}
   
\end{remark}

\subsection{Externally graded actions}\label{sec ext graded}
 We now suppose that $\varphi \colon G = U \rtimes R \twoheadrightarrow R$ acts equivariantly on an affine morphism $f \colon X \to Z$, and that this action is externally graded by a $\GG_m \subseteq \Aut(U)$. As above, we fix a good $R$-quotient  $q \colon Z' \to W$ of an open subset $Z'\subseteq Z$ and a character $\rho \colon G \rightarrow \GG_m$.
 
When $Z' \subseteq Z^\circ $, we can use the grading $\GG_m$ to take a quotient of the fibrewise $U$-action on $f' \colon X'=f^{-1}(Z')\rightarrow Z'$ using \cref{theorem U quotient with varying stabilisers}; this gives a geometric $U$-quotient
\[ q_U \colon X' \rightarrow X'/U :=\rSpec_{Z'} f'_*\cO_{X'}^U,
\]
which is affine and of finite type over $Z'$. We can then forget this $\GG_m$ and take a relative $\rho$-twisted affine GIT quotient of the equivariant $R$-action on $X'/U \rightarrow Z'$ with respect to $q$. This construction yields a quotient that is projective-over-affine over $W$ as follows. 

\begin{lemma}\label{lemma ext grading forget Gm}
Let $\varphi \colon G = U \rtimes R \twoheadrightarrow R$ act equivariantly on an affine morphism $f \colon X \to Z$ and suppose this is externally graded by a $\GG_m \subseteq \Aut(U)$. Given a character $\rho \colon G \rightarrow \GG_m$ and a good $R$-quotient  $q \colon Z' \to W$ of an open subset $j \colon Z' \hookrightarrow Z^\circ \hookrightarrow  Z$, the sheaf $q_*j^*f_*\cO_X^{G,\rho}$ of $\rho$-twisted semi-invariants is of finite type over $W$. The inclusion induces a rational map
\[ q_G \colon X' = f^{-1}(Z) \dashrightarrow  X \gitq^f_{\hspace{-2pt} q, \rho \hspace{2pt}} G = \rProj_W q_*j^*f_*\cO_X^{G,\rho} \]
that gives a good $G$-quotient of its domain of definition. Furthermore, we have a factorisation $X \gitq^f_{\hspace{-2pt} q, \rho \hspace{2pt}} G \rightarrow X \gitq^f_{\hspace{-2pt} q \hspace{2pt}} G \rightarrow W$ as a projective morphism followed by an affine morphism.
\end{lemma}

This construction gives us some of the properties claimed in Theorem \ref{main theorem externally graded case}, but not the Hilbert-Mumford criterion or the projective completion. To attain these, we will show that the quotient above can also be constructed in a different manner. As in \cref{remark graded} \eqref{fromexternaltointernal}, we have an induced internally graded equivariant action of $\widetilde{\varphi} \colon \widetilde{G} : = G \rtimes \GG_m \twoheadrightarrow \widetilde{R} : = R \times \GG_m$ on $\widetilde{f} \colon \widetilde{X} := X \times \AA^1 \rightarrow Z$ to which we can thus apply \cref{main theorem}. For any character $\tilde{\rho}$ of $\widetilde{G}$, this gives a good quotient  $\widetilde{X}^{ss}(\widetilde{f};q,\widetilde{\rho}) \to \widetilde{X} \gitq^{\widetilde{f}}_{\hspace{-2pt} q, \widetilde{\rho} \hspace{2pt}} \widetilde{G}$, projective over $W$. We can use this to construct a good $G$-quotient of an open set in $X$ when $\widetilde{\rho} = (\rho, m)$ for $m >0$ as follows.

 \begin{proposition} \label{externally graded} 
     For an equivariant action of $\varphi \colon G = U \rtimes R \twoheadrightarrow R$ on an affine morphism $f \colon X \to Z$ that is externally graded by a $\GG_m$, suppose $q \colon Z' \rightarrow W$ is a good $G$-quotient of an open set with $Z' \subseteq Z^\circ \subseteq Z$ and $\rho \colon G \rightarrow \GG_m$ is a character of $G$. There is an internally graded equivariant action of $\widetilde{\varphi} \colon \widetilde{G} : = G \rtimes \GG_m \twoheadrightarrow \widetilde{R} : = R \times \GG_m$ on $\widetilde{f} : \widetilde{X} := X \times \AA^1 \rightarrow Z$ with relative NRGIT quotient $\widetilde{\pi} \colon \widetilde{X}^{(s)s}(\widetilde{f};q,\widetilde{\rho}) \rightarrow \widetilde{X} \gitq^{\widetilde{f}}_{\hspace{-2pt} q, \widetilde{\rho} \hspace{2pt}} \widetilde{G}$ where $\widetilde{\rho} = (\rho,m)$ for $m \in \ZZ_{>0}$ such that the following statements hold.
    \begin{enumerate}[(i)]
    \item Using the $G$-equivariant inclusion $\iota \colon X \hookrightarrow \widetilde{X}$ given by $x \mapsto (x,1)$ we obtain $G$-invariant open subsets $X^{(s)s}(\widetilde{f},\widetilde{\varphi};q,\rho): =  X \cap \widetilde{X}^{(s)s}(\widetilde{f};q,\widetilde{\rho})$ of $X$;
        \item there is a relative good $G$-quotient $X^{ss}(\widetilde{f},\widetilde{\varphi};q,{\rho}) \to X \gitq^{\widetilde{f},\widetilde{\varphi}}_{\hspace{-2pt} q, \rho \hspace{2pt}} {G}:= \widetilde{\pi}(X^{ss}(\widetilde{f},\widetilde{\varphi};q,{\rho}))$, which is quasi-projective over $W$, with a relative projective completion over $W$ given by $\widetilde{X} \gitq^{\widetilde{f}}_{\hspace{-2pt} q, \tilde{\rho} \hspace{2pt}} \widetilde{G}$, whose boundary component is isomorphic to $X \gitq^{f}_{q,\tilde{\rho}} \widetilde{G}$.
        \item This relative quotient  restricts to a geometric quotient on  $X^{s}(\widetilde{f},\widetilde{\varphi};q,\widetilde{\rho})$.
        \item The (semi)stable loci $X^{(s)s}(\widetilde{f},\widetilde{\varphi};q,\rho): =  X \cap \widetilde{X}^{(s)s}(\widetilde{f};q,\widetilde{\rho})$ have induced Hilbert--Mumford descriptions coming from those for $\widetilde{X}^{(s)s}(\widetilde{f};q,\widetilde{\rho})$ given in \eqref{eq HM descr nonred}.
    \end{enumerate}
 \end{proposition}

\begin{proof}
Recall from \cref{remark graded} \eqref{fromexternaltointernal} that $\widetilde{G}$ acts on $\AA^1$ by scaling by $\GG_m$ and the trivial $G$-action, and yields an internally graded equivariant action of $\widetilde{\varphi}$ on $\widetilde{f}$. By applying \cref{main theorem} to this action, we obtain the quotient $\widetilde{\pi}$.

The first claim is immediate as $G$ acts trivially on the copy of $\AA^1_{Z}$ in $\widetilde{X}$. For the second claim, let us first suppose that $R$ is trivial and $Z^\circ = Z$ so we have a fibrewise action of $\hU:= U \rtimes \GG_m$ on $\tilde{f}$ which we quotient with respect to the $\GG_m$-character $\tilde{\rho}=m$. We claim that the good $\hU$-quotient $\widetilde{\pi}_{\hU} \colon \widetilde{X}^{\hU,ss}(\widetilde{f};m) \rightarrow \widetilde{X} \gitq^{\widetilde{f}}_{\hspace{-2pt} m \hspace{2pt}} \hU$ is geometric as $m>0$. Recall this quotient is constructed by first taking a geometric $U$-quotient $\widetilde{X} \rightarrow Y:=\widetilde{X}/U:= \rSpec_Z \widetilde{f}_*\cO_{\widetilde{X}}^U$ and then taking a $\widetilde{\rho}$-twisted affine GIT quotient of the fibrewise $\GG_m$-action on $Y \rightarrow Z$. Since $Z = Y^{\GG_m}$ and the $\GG_m$-action induces a grading on $\widetilde{f}_*\cO_{\widetilde{X}}^U$ supported in non-positive degrees, we have a geometric $\GG_m$-quotient $Y^{ss}(m) = Y \setminus Z \rightarrow Y\gitq_{\hspace{-2pt} m} \GG_m$ as in  \cref{example VGIT Gm} and the $\GG_m$-orbits in $Y^{ss}(m)$ all have dimension one. Hence $\widetilde{\pi}_{\hU}$ is a geometric $\hU$-quotient. In particular, any $\hU$-invariant open subset $X'$ of $\widetilde{X}^{\hU,ss}(\widetilde{f};m)$ is automatically orbitwise closed (as all orbits in the semistable locus have the same dimension) and so the restriction $X' \rightarrow \widetilde{\pi}_{\hU}(X')$ is a geometric $\hU$-quotient whose image is an open subset in $\widetilde{X} \gitq^{\widetilde{f}}_{\hspace{-2pt} m \hspace{2pt}} \hU$. We  apply this to $X' = \GG_m \cdot \iota(X)$.

The image of the $U$-equivariant inclusion $\iota \colon X \hookrightarrow \widetilde{X}$ is contained in $\widetilde{X}^{\hU,ss}(\widetilde{f},m)$, as the coordinate on $\AA^1$ gives a $\rho$-twisted (semi)invariant section which is non-zero on any point in the image of this inclusion. Hence $X^{U,ss}(\widetilde{f},\widetilde{\varphi}): =  X \cap \widetilde{X}^{\hU,ss}(\widetilde{f};m) = X $. 
Moreover $X':=\GG_m \cdot \iota(X) \simeq  X \times (\AA^1\setminus \{0\})$ is a $\hU$-invariant open subset of $\widetilde{X}^{\hU,ss}(\widetilde{f};m)$ and so we obtain a geometric $\hU$-quotient $X' \rightarrow \widetilde{\pi}_{\hU}(X')$ as above. We claim the composition 
\[ \pi_U \colon X \rightarrow X' \rightarrow \widetilde{\pi}_{\hU}(X')=\widetilde{\pi}_{\hU}(X)\]
is a geometric $U$-quotient, where $\widetilde{\pi}_{\hU}(X)$ is an open subset of the projective $Z$-scheme $\widetilde{X} \gitq^{\widetilde{f}}_{\hspace{-2pt} m \hspace{2pt}} \hU$. Indeed, as $\pi_U$ satisfies the topological properties of being a good quotient and its fibres are $U$-orbits, it suffices to show that $\cO_{\widetilde{\pi}_{\hU}(X')} \cong (\pi_U)_* \cO_X^U$. This is the case as locally $\pi_U$ corresponds to $\Spec A \rightarrow \Spec A[x,x^{-1}] \rightarrow \Spec A[x,x^{-1}]^{\hU} = A^U$, where the last equality holds as $U$ acts trivially on $x$ and $\GG_m$ scales $x$. This completes the proof of the second claim when $R$ is trivial.

If $R$ is non-trivial, we can assume by replacing $X$ by $f^{-1}(Z')$ that $Z' = Z^\circ = Z$. The relative NRGIT quotient  $\widetilde{X} \gitq^{\widetilde{f}}_{\hspace{-2pt} \rho \hspace{2pt}} \widetilde{G}$ is constructed as a quotient of the fibrewise $\hU$-action, followed by a quotient of the equivariant $R$-action as follows
\begin{center}
\begin{tikzcd}
\widetilde{X} \arrow[r, dashrightarrow, "{\widetilde{\pi}_{\hU}}"] \arrow[rd, "{\widetilde{f}}"] & \widetilde{X} \gitq^{\widetilde{f}}_{\hspace{-2pt} m \hspace{2pt}} \hU \arrow[r, dashrightarrow, "{\widetilde{\pi}_{R}}"] \arrow[d, "h"] & ( \widetilde{X} \gitq^{\widetilde{f}}_{\hspace{-2pt} m \hspace{2pt}} \hU)\gitq^{h}_{\hspace{-2pt} q,\rho \hspace{2pt}} R = \widetilde{X} \gitq^{\widetilde{f}}_{\hspace{-2pt} \rho \hspace{2pt}} \widetilde{G}\arrow[d] \\ & Z \arrow[r, "q"] & W.
\end{tikzcd} 
\end{center}
Hence it suffices to prove that $\widetilde{\pi}_{\hU}(X) \cap (\widetilde{X} \gitq^{\widetilde{f}}_{\hspace{-2pt} m \hspace{2pt}} \hU)^{R,ss}(h;q,\rho)$ is an $R$-invariant saturated open subset of $(\widetilde{X} \gitq^{\widetilde{f}}_{\hspace{-2pt} m \hspace{2pt}} \hU)^{R,ss}(h;q,\rho)$, so the restriction of the good $R$-quotient $\widetilde{\pi}_R$ to this open set is also a good $R$-quotient. This holds as $\widetilde{\pi}_{\hU}(X) = \widetilde{\pi}_{\hU}(\GG_m \cdot X)$ and $\GG_m \cdot X$ is an $\widetilde{G}$-saturated open subset of $\widetilde{X}^{\hU,ss}(\widetilde{f},m)$: we have already shown this subset is $\hU$-saturated. As the inclusion $\iota \colon X \hookrightarrow \widetilde{X}$ is closed and $R$-invariant, and $R$ acts trivially on $\AA^1$, the subset $\GG_m \cdot \iota(X) \subset \widetilde{X}$ is $R$-orbitwise closed. Hence, we obtain a good $G$-quotient of $X^{ss}(\widetilde{f},\widetilde{\varphi};q,{\rho})$ as the composition of a geometric $U$-quotient $\pi_U$ followed by the good $R$-quotient obtained by restricting $\widetilde{\pi}_R$. This completes the proof of (ii), as the boundary corresponds to the quotient of the $\widetilde{G}$-invariant closed subset $X \times \{ 0 \} \hookrightarrow \widetilde{X}$. The remaining two items then follow immediately.
\end{proof}

    It is important that we chose the extended character $\widetilde{\rho}$ to have $\GG_m$-weight $m>0$. For example, if we instead take $\widetilde{\rho} = (\rho,0)$ then the $\GG_m$-quotient is no longer geometric, thus $\GG_m \cdot X$ will not be orbitwise closed in the semistable locus, which means the restriction of $\widetilde{\pi}_{\hU}$ to this set will no longer yield a good quotient.

We now complete the proof of Theorem \ref{main theorem externally graded case}, by showing that the quotients of \cref{lemma ext grading forget Gm} and \cref{externally graded} coincide, provided one chooses $m\gg 0 $ in the latter construction.

\begin{proof}[Proof of \cref{main theorem externally graded case}]
It suffices to prove the quotients constructed in \cref{externally graded} (for $m\gg 0 $) and \cref{lemma ext grading forget Gm} agree, as they give quotients with the desired properties stated in Theorem \ref{main theorem externally graded case}. As usual the question is local on $Z$ and we can first base change to $X' = f^{-1}(Z')$ and then additionally assume that $X = \Spec A$ is affine, where by the graded property of $f$ we have $Z = \Spec A_0$, where $A_0$ is the $0$th graded piece of $A$ for the grading $\GG_m$. Then the quotient map for the extended group $\widetilde{G} = G \rtimes \GG_m$ is given by \[  \widetilde{X} = X \times \AA^1 =\Spec A[v] \dashrightarrow \Proj A[v]^{\widetilde{G},\widetilde{\rho}} =\widetilde{X}  \gitq^{\widetilde{f}}_{\hspace{-2pt} \widetilde{\rho} \hspace{2pt}} \widetilde{G} \] is obtained via the inclusion of semi-invariants for the character $\widetilde{\rho} = (\rho,m)$ for  $m > 0$.

The ring of semi-invariants $A^{G,\rho}$ is a finitely generated $A^G$-algebra by Lemma \ref{lemma ext grading forget Gm}, so we can choose a finite generating set $S \subset A^{G,\rho}$, consisting of semi-invariants with strictly positive $\rho$-weights. Moreover, if $a \in A$ is a $\rho$-semi-invariant, then so must be each of its (grading) $\Gm$-weighted homogeneous parts, and so without loss of generality we can take all elements of $S$ to be weight vectors for the grading $\Gm$. We then define 
\[M := \max\left\{\frac{\text{wt}_{\Gm}(a)}{\text{wt}_\rho(a)}  \mid a\in S \right\}.\]
We claim that, provided we take $m > M$, we get for $t \neq 0$ 
\begin{equation} \label{ext gr claim} (x,t) \in \widetilde{X}^{ss}(\widetilde{f};q,\widetilde{\rho}) \iff x \in X^{ss}(f;q,\rho).\end{equation}
Assuming that this holds, we see that $X\cap \widetilde{X}^{ss}(\widetilde{f};q,\widetilde{\rho}) = X^{ss}(f;q,\rho)$ has two good quotients constructed by  \cref{externally graded} and \cref{lemma ext grading forget Gm}, which must coincide by the universal property of categorical quotients. 
To prove \eqref{ext gr claim}, first suppose $x \in X^{ss}(f;q,\rho)$, so there exists some $\rho$-semi-invariant $s \in S$ such that $s(x) \neq 0$. Let $w= \text{wt}_{\GG_m}(s)$ and $n = \text{wt}_{\rho}(s)$; then by choice of $m$, we have $nm > w$, so $\tilde{s}:= sv^{nm-w} \in A[v]$ and moreover is a $\widetilde{\rho}$ semi-invariant such that $\tilde{s}(x,t) \neq 0$ for $t\neq 0$. 
Conversely, suppose for some $t\neq 0$ that $(x,t) \in  \widetilde{X}^{ss}(\widetilde{f};q,\widetilde{\rho}) $. Then there is some $h \in A[v]^{\widetilde{G}}_{\widetilde{\rho}^n}$  with $n > 0$ such that $h(x,t) \neq 0$. Write $h = \sum_{i=0}^{r} h_iv^i$, then each $h_i \in A^{G}_{\rho^n}$ as $\widetilde{\rho}(g)=\rho(g)$ for $g \in G$ and $G$ acts trivially on $\AA^1 = \Spec k[v]$. Since $h(x,t) \neq 0$, there exists some $i$ with $h_i(x) \neq 0$, and since $h_i$ is a $\rho$-semi-invariant of weight $n > 0$, this proves $x \in X^{ss}(f;q,\rho)$.
\end{proof}

We illustrate the two different ways to construct the externally graded quotient using a simple example, which was studied by Gabriel da Silva Martinho using our previous work \cite{HHJ25}.

\begin{example} 
Let $X= \text{Mat}_{2\times 2}$ and consider the action of 
\[ G = U \rtimes R =\Ga\rtimes \Gm = \left\{\begin{pmatrix} s & 0 \\ u & s\end{pmatrix} \mid s \in \GG_m, u\in \GG_a\right\}\]
by left multiplication. As $G$ is not internally graded, we use the 1-PS $\lambda_g(t)= \text{diag}(1,t)$ to externally grade $G$ and consider the action of $\widetilde{G} = G \rtimes_{\lambda_g} \Gm$. The $\lambda_g$-flow is the affine map $f \colon X \rar Z =\AA^2$ projecting onto the locus of matrices $A=(a_{ij})$ where $a_{21}=a_{22}=0$.

Consider the quotient of $R = \GG_m$ acting on the $Z$ with respect to the character $\theta = +1$ \[q \colon Z^{\theta-ss} = \AA^2\setminus \{0\} \rar Z \gitq_{\! \theta  } \GG_m = \Proj k[a_{11},a_{12}]= \PP^1 =: W.\] The semistable locus $Z^{\theta-ss}$ is the locus in $Z$ with trivial $U$-stabilisers. The quotient $W$ is covered by two affine opens $W_{1},W_{2}$ where $a_{11},a_{12}$ respectively do not vanish; we denote their pre-images in $X$ by $X_{1},X_{2}$ respectively. Our quotient is constructed by quotienting these two pieces separately, and then gluing.  By symmetry, it suffices to consider $X_{1}$, which has unipotent quotient \[X_1 = \Spec k[a_{11},a_{11}^{-1},a_{12},a_{21},a_{22}] \rar \Spec \cO(X_1)^{\GG_a}= \Spec k [a_{11},a_{11}^{-1},a_{21},\Delta] =: Y_1,\]
where $\Delta := a_{11}a_{22}-a_{12}a_{21}$. The (semi)invariant ring of the $R$-action on $Y_1$ (for $\rho = +1$) is
\[ \cO(Y_1)^R = k \left[\frac{a_{12}}{a_{11}}, \frac{\Delta}{a_{11}^2}\right] = \cO(W_1)\left[\frac{\Delta}{a_{11}^2}\right] \quad  \left(\ \text{resp.} \  \cO(Y_1)^{R,\rho} = \cO(Y_1)^R [a_{11}] \ \right), \] 
where the semi-invariant $a_{11}$ has degree $1$. Hence the twisted and untwisted quotients coincide 
\[ X_1 \gitq_{\! \rho} G = Y_1 \gitq_{\! \rho} R =\Proj \cO(Y_1)^{R,\rho} = \Spec \cO(Y_1)^{R} = Y_1 \gitq R = X_1 \gitq G,\] 
and $X_1 \gitq_{\! \rho \ } G \rightarrow W_1$ is a trivial $\AA^1$-bundle, and similarly for $X_2$. Thus $X \gitq_{\! q,\rho} G$ is a Zariski-locally trivial $\AA^1$-bundle over $W$.

 By the proof of \cref{main theorem externally graded case}, $X \gitq_{\! q,\rho } G$ coincides with an open subset inside a twisted quotient of the induced internally graded action of $\widetilde{G}:= G \rtimes \GG_m$ on $\widetilde{X} := X \times \AA^1$, where $G$ acts trivially on $\AA^1$ and the grading $\GG_m$ acts with weight $1$, and we take a character $\tilde{\rho} = (\rho,m)$ for $m>0$ sufficiently large. In this example, one can check that $m=1$ works (as the $\rho$-semi-invariants are generated as an algebra over the invariants by just $a_{11}$, which has $\rho$-weight $1$ but weight zero for the grading $\Gm$). The first patch $\widetilde{X}_1 = X_1 \times \AA^1$ has $\GG_a$-quotient $\widetilde{X}_1 = X_1 \times \AA^1 = \Spec k[a_{11}^{\pm 1},a_{21},\Delta,w]$ where $w$ is the coordinate on $\AA^1$. For $\tilde{\rho} = (1,1)$, we have
 \[  \cO(\widetilde{Y}_1)^{\widetilde{R},\tilde{\rho}} =  k \left[ \frac{a_{12}}{a_{11}} \right] \left[ \frac{\Delta}{a_{11}},wa_{11} \right] = \cO(\widetilde{Y}_1)^{\widetilde{R}}\left[ \frac{\Delta}{a_{11}},wa_{11} \right] = \cO(W_1)\left[ \frac{\Delta}{a_{11}},wa_{11} \right].\]
By symmetry we can describe the second patch, and we conclude that $\widetilde{X} \gitq_{ \! q,\tilde{\rho} \ }\widetilde{G}$ is a Zariski-locally trivial $\PP^1$-bundle over $W$, which contains the $\AA^1$-bundle $X \gitq_{ \! q,\rho \ } G$ as a dense open.

\end{example}

\begin{remark}\label{rmk why R reductive}  We close this section noting that, as the good $R$-quotient $q \colon Z' \rightarrow W$ is given to us, we do not actually use or need the assumption that $R$ is reductive (analogously to \cref{rmk base group can be nonred}). We could instead consider a \lq partially graded' $\GG_m$-action on $G$ that is only strictly positive on some part of the unipotent radical, and so flows $G$ to a non-reductive group $H$ (with smaller unipotent radical), and consider an equivariant action of $\varphi \colon G \rightarrow H$ on $f$. This leads to a \lq quotienting-in-stages\rq\ approach in the style of \cite{Hoskins2021}, which can be very helpful in situations where the usual unipotent stabiliser conditions fail (such as moduli of unstable sheaves, for example), but introduces significant extra complexity. We do not pursue this, as our intended applications do not require this approach.
\end{remark}

\section{Absolute quotients from relative quotients} \label{sec:quotients by non-reductive absolute}

There are two goals in this section. Our main aim is to use our relative approach to recover known results in projective NRGIT, and prove the projective Hilbert-Mumford criterion. The other is to more generally consider a graded non-reductive group $G$ acting on a scheme $X$, and use the grading $\GG_m$ to provide an open subset $X' \subseteq X$ and an affine morphism $X' \to Z'$ with $Z' \subseteq X^{\GG_m}$ to which we can apply the results of the previous section and construct a quotient for the $G$-action on $X$ as a relative quotient for $X' \to Z'$. 

\subsection{Multiplicative actions and Bia{\l}ynicki-Birula decompositions}

Decompositions for actions of the multiplicative group have been studied in various contexts \cite{Birula1991,BS,Jurkiewicz,Konarski,Sommese}. In this section, we will summarise some of the key results we will need about Bia{\l}ynicki-Birula decompositions.

For a $\GG_m$-action on a scheme $X$ (separated and of finite type over $k$), we consider the following functors 
\[ X^0:= \Hom^{\GG_m}(-,X) \quad \text{and} \quad X^+:= \Hom^{\GG_m}(\AA^1 \times -,X) \]
of $\GG_m$-equivariant morphisms to $X$, where $\AA^1$ is given the standard scaling $\GG_m$-action. By work of Drinfeld \cite{Drinfeld}, both functors are represented by schemes, denoted $X^0$ and $X^+$, called the fixed locus and attractor. Moreover, $X^0 \subseteq X$ is a closed subscheme and there is a monomorphism $\iota \colon X^+ \rightarrow X$ and an affine morphism $f \colon X^+ \rightarrow X^0$. More precisely $X^0$ is the $\GG_m$-fixed locus and $X^+$ is the locus of points $x$ such that $\lim_{t \rar 0 } t \cdot  x$ exists. If $Z_i \subseteq X^0$ denote the connected components of $X^0$, then $X_i:=f^{-1}(Z_i)$ denote the connected components of $X^+$. When the images of $X_i$ under $\iota$ are locally closed subschemes of $X$, we will refer to them as BB strata, and will often also denote these by $X_i$. 

\begin{definition}
For a $\Gm$-action on a separated scheme $X$, we say
\begin{enumerate}
    \item the action is graded if $\lim_{t \rar 0 } t \cdot  x$ exists for all $x \in X$ (or equivalently $\iota$ is a bijection);
    \item $X^+$ admits a \emph{Bia{\l}ynicki-Birula (BB) decomposition} if the restriction of the monomorphism $\iota \colon X^+\rightarrow X$ to each connected component $X_i$ is a locally closed embedding.
    \item $X^+$ admits a \emph{BB stratification} if it admits a BB decomposition with a partial ordering on the set $\cI$ of connected components of $X^0$ such that $\overline{p(X_i)} \cap \iota(X_j) = \emptyset$ for $j < i$ (i.e.\ the closure of a stratum contains only higher strata).
    \item $X^+$ admits an \emph{open BB stratum} if there is a connected component $X_i$ on which $p$ restricts to an open embedding. 
\end{enumerate}
\end{definition}

\begin{remark}\
\begin{enumerate}
    \item Any $\GG_m$-action on a proper scheme $X$ is graded. Hausel refers to a $\GG_m$-action on a quasi-projective scheme as \emph{semi-projective} if it is graded and $X^0$ is projective. Examples include the scaling action on moduli spaces of (semistable) quiver representations and Hitchin's $\GG_m$-action on the moduli space of Higgs bundles given by scaling the Higgs field.
    \item If $X$ is smooth, then Bia{\l}ynicki-Birula proved that $X^+$ admits a BB decomposition and moreover $f_i \colon X_i \rightarrow Z_i$ are affine space fibrations \cite{BB}.
    \item If $X$ admits a $\GG_m$-equivariant locally closed embedding into a projective space $\PP^n$ with linear $\GG_m$-action (for example, a $\GG_m$-equivariant ample line bundle on $X$ gives such an embedding), then the BB stratification on $\PP^n$ restricts to a BB stratification on $X^+$.
    \item If $X$ is normal and quasi-projective, then it admits a $\GG_m$-equivariant ample line bundle by a result of Sumihiro (as a power of any ample line bundle will admit a $\GG_m$-equivariant structure), and thus $X$ admits a $\GG_m$-equivariant locally closed embedding into a projective space $\PP^n$. In particular, $X^+$ admits a BB stratification. If these assumptions fail (or even if we work with algebraic spaces rather than schemes), BB stratifications do not always exist, as summarised in the discussion in \cite[Appendix B]{Drinfeld}.
    \item If the $\GG_m$-action on $X$ is graded and admits a BB stratification, then there is an open BB stratum corresponding to the minimal index in $\cI$.
\end{enumerate}
\end{remark}

\subsection{Quotients of BB strata}

We first prove that if $G$ is a connected graded linear algebraic group acting on $X$, and the grading $\GG_m \subseteq G$ induces a BB stratification of $X$, then the retraction maps on the BB strata admit equivariant actions.

\begin{lemma}\label{lemma BB stratum internally graded}
Let $G = U\rtimes R$ be a connected linear algebraic group graded by $\GG_m \subseteq  G$. Suppose $G$ acts on a separated scheme $X$ and the $\GG_m$-action is graded and induces a BB stratification $X = \cup X_i$. Then the BB strata $X_i$ are $G$-invariant, and there are induced equivariant graded actions of $G \twoheadrightarrow R$ on each affine map $f_i \colon X_i \rar Z_i$. \end{lemma} 
\begin{proof}
By definition of $G = U \rtimes R$ being graded, $R$ is the centraliser of the grading $\GG_m$ and so the $R$-action preserves $X^0=X^{\GG_m}$, and as we additionally assume that $R$ is connected, the $R$-action fixes each connected component $Z_i$ of the fixed locus. 

For any $t$ in the grading $\GG_m$ and any $r \in R$ and $x \in X_i$, we have \[\lim_{t\rar 0} tr \cdot x = r\lim_{t\rar 0} t \cdot x = rf_i(x),\]
which proves that $r \cdot x \in X_i$ and $f_i$ is $R$-equivariant. Since $\GG_m$ grades $U$, for any $t \in \GG_m$ and $u \in U$, we have $tu = u(t)t$ for some $u(t)\in U$ such that $\lim_{t\rar 0} u(t) = e \in U$. It thus follows that for any $x \in X_i$ we have
\[\lim_{t\rar 0} tu \cdot x = \lim_{t\rar 0}u(t)t \cdot x = f_i(x)\] 
which proves that $u \cdot x \in X_i$ and $f_i$ is $U$-invariant.
\end{proof}

Given such a BB stratum $f_i \colon X_i \rightarrow Z_i$, we write $Z_i^{\circ} = Z_i^\circ(U)$ for the locus on which the cokernels of the infinitesimal maps are locally free (see \cref{def Zcirc for Z arbitrary}). 
We can then apply \cref{main theorem} to the equivariant action on $f_i$ to obtain a quotient as follows.

\begin{theorem}\label{thm quotients of BB strata}
    Let $G = U\rtimes R$ be a connected graded linear algebraic group acting on a scheme $X$ that admits a BB stratification $\{X = \bigsqcup_i X_i \rightarrow \bigsqcup_i Z_i = X^{\GG_m} \}$ for the grading $\Gm$. Suppose for some $i$ there exists a good $R$-quotient $q_i \colon Z_i'\rar W_i$ of an open subset $Z_i' \subseteq Z_i^{\circ} \subseteq Z_i$; then for any character $\rho \colon G \rightarrow \GG_m$, there is an open subset $X_i^{ss}(f_i;q_i,\rho)$ of $X_i$ and a good $G$-quotient $X_i^{ss}(f;q_i,\rho) \rightarrow X_i \gitq^{f_i}_{q_i,\rho} G$ which is projective over $W_i$. 
    
\end{theorem}
\begin{proof}
    This directly follows from \cref{main theorem}. 
\end{proof}

In particular, by applying the above theorem to an open BB stratum $X_i \rightarrow Z_i$, we obtain a good quotient of an open subset of $X$. The following theorem describes the special case where the good quotient $q_i$ is constructed via reductive GIT.

\begin{theorem} \label{Main thm on quotienting BB strata} Let $G = U\rtimes R$ be a connected graded linear algebraic group acting on a scheme $X$ that admits an open BB stratum $X_i \subseteq X$ with retraction morphism $f_i \colon X_i \rightarrow Z_i$. Let $\cL$ be an $R$-equivariant line bundle on $Z_i$ with reductive GIT quotient $q \colon Z_i^{ss}(\cL_i) \rightarrow W:=Z_i \gitq_{\cL} R$ such that $Z_i^{ss}(\cL) \subseteq Z_i^\circ$. Then for any character $\rho \colon G \rar \Gm$, the relative quotient $X_i\gitq^{f_i}_{q,\rho} G$ is projective over $W$ and is a good $G$-quotient of an open set $X_i^{ss}(f_i;q,\rho)$ of $X$ admitting the following Hilbert--Mumford description
\[ X_i^{ss}(f_i;q,\rho)= \left\{ x \in X_i \: : \: \begin{array}{c} f_i(x) \in Z_i^{ss}(\cL) \: \text{and} \; \langle \rho,\lambda \rangle \geq 0 \text{ for all } 1\text{-PS } \lambda \colon \GG_m \rightarrow G \text{ such} \\ \text{that } x_0:=\lim_{t \rightarrow 0} \lambda(t) \cdot x \text{ exists in } X_i \text{ and } f_i(x_0) \in Z_i^{ss}(\cL)  \end{array} \right\}   \]
If $Z_i$ is projective-over-affine, then $Z_i^{ss}(\cL)$ admits the following Hilbert--Mumford description
\[ Z_i^{ss}(\cL) = \{ z \in Z_i : \mu^{\cL}(x,\lambda) \geq 0 \text{ for all } 1\text{-PS } \lambda \colon \GG_m \rightarrow R \text{ such that } \lim_{t \rightarrow 0} \lambda(t) \cdot x \text{ exists in } Z_i \}.\]
Finally, if $Z_i$ is projective, then $X_i\gitq^{f_i}_{q,\rho} G$ is projective over $k$, and if $X_i':=f_i^{-1}(Z_i^{ss}(\cL))$, the relative quotient is given by 
    \[ X_i\gitq^{f_i}_{q,\rho} G = \Proj_k \bigoplus_{n\geq 0} H^0(X'_i,f_i^*\cL^{\otimes mn})_{\rho^{ln}}^G \]
for some (sufficiently divisible) choice of $l,m \gg 0$.

\end{theorem}

\begin{proof}
The only claim left to prove is the final one, describing the case when $Z_i$ itself is projective. To prove this, we adopt the following simplified notation. Let $Y:=X_i$ and $Z:=Z_i$, $f:=f_i$ and $f' \colon Y':=f^{-1}(Z') \rightarrow Z' := Z^{ss}(\cL)$. We let $Y'' := X_i^{ss}(f_i;q,\rho)$, and let $f'' \colon Y'' \rar Z'$ denote the restriction of $f'$ to $Y''$. 
We let $p \colon Y'' \rightarrow V:=Y\gitq^{f}_{q,\rho} G$ denote the relative quotient, where
    \[ V = \rProj_W \bigoplus_{n \geq 0} (q_*(f')_*\cO_{Y'})^G_{\rho^n} \quad \text{and} \quad W = \Proj \bigoplus_{r \geq 0} H^0(Z,\cL^{\otimes r})^R\]
    The relative quotient $V$ is projective over $W$, and $W$ itself is projective due to being a reductive GIT quotient of a projective scheme with respect to an ample linearisation $\cL$. Hence $V$ is projective over $k$. To describe $V$ as a projective scheme over $k$, we will produce an ample line bundle on $V$ using the ample line bundles arising in the GIT constructions of $V$ and $W$. Then we relate sections of this ample line bundle with semi-invariant sections of $f^*\cL$.

By construction of the reductive GIT quotient $q \colon Z' \rightarrow W:=Z\gitq_{\cL} R$, there is an ample line bundle $\cM_W$ on $W$ such that we have an isomorphism $q^*\cM_W \cong \cL^{\otimes N}$ as $R$-equivariant sheaves for some $N \gg 0$. For the relative quotient $p \colon Y'' \rightarrow V$, there is a line bundle $\cN_V$ on $V$ which is relatively ample with respect to the structure morphism $\pi \colon V \rightarrow W$ and such that we have an isomorphism $p^*\cN_V \cong (\cO_{Y'',\rho})^{\otimes M} \cong (\cO_{Y''})_{\rho^M}$ as $G$-equivariant sheaves for some $M \gg 0$.  Since $\cN_V$ is $\pi$-relatively ample and $\cM_W$ is ample, for some $K \gg 0$, the line bundle $
    \cN_V \otimes \pi^*\cM_W^{\otimes K}$ is ample (see \cite[Tag 01VG Lemma 29.37.4]{stacks-project}). We now pullback this line bundle to $Y''$:
    \begin{align*}
        p^*(\cN_V \otimes \pi^*\cM_W^{\otimes K}) & \simeq p^*(\cN_V) \otimes p^*\pi^*(\cM_W)^{\otimes K} \simeq (\cO_{Y''})_{\rho^M} \otimes (f'')^*q^* \cM_W^{\otimes K} \\
        & \simeq (\cO_{Y''})_{\rho^M} \otimes (f'')^*\cL^{\otimes NK}
    \end{align*}
    as $G$-equivariant line bundles. Since $p \colon Y'' \rightarrow V$ is a good quotient, we have $p_*\cO_{Y''}^G \cong \cO_V$ and moreover for a locally free sheaf $\cF$, the adjunction $\cF \rightarrow (p_*p^*\cF)^G$ is an isomorphism (see \cref{lem good qnt sheaf isom} below). Hence we have
     \[ \cN_V \otimes \pi^*\cM_W^{\otimes K} \simeq p_*p^*(\cN_V \otimes \pi^*\cM_W^{\otimes K})^G \simeq p_*((\cO_{Y''})_{\rho^M} \otimes (f'')^*\cL^{\otimes NK})^G  \simeq p_*((f'')^*\cL^{\otimes NK})^G_{\rho^M} \]
     and upon taking tensor powers and global sections, we obtain
     \[ H^0(V,(\cN_V \otimes \pi^*\cM_W^{\otimes K})^{\otimes n}) = H^0(Y'',(f'')^*\cL^{\otimes KNn})^G_{\rho^{Mn}}.  \]
   The proof is completed by \cref{lemma w discussion before using sections to trivialise locally}  below, which relates invariant sections over $Y''$ and $Y'$.
\end{proof}

To complete the proof, we prove the following two lemmas. 

\begin{lemma}\label{lem good qnt sheaf isom} 
In the setting of \cref{main theorem}, let $p \colon X^{ss}(f;q,\rho) \rightarrow V:=X \gitq^f_{\hspace{-2pt} q,\rho \hspace{2pt}} G$ denote the good $G$-quotient. Then for any line bundle $\cL$ on $V$, we have $ \cL\cong p_*p^*(\cL)^G$.
\end{lemma}
\begin{proof}
By construction, $p$ is the composition of a geometric $U$-quotient $p_U \colon X^{ss}(f;q,\rho) \rightarrow Y := X^{ss}(f;q,\rho)/U$ followed by a reductive GIT quotient $p_R \colon Y  \rightarrow V$. 

Since we are working in characteristic zero, $R$ is linearly reductive and so $ \cL\cong p_{R*}p^*_R(\cL)^R$ (a stacky version of this is \cite[Prop.\ 4.5]{Alper}, where taking invariants does not appear in the stacky version).  To give a direct and more elementary proof, we note that by working locally it suffices to check for an affine $G$-scheme $\Spec A$ and an $A^R$-module $M$ that the natural map $M \rightarrow  (M\otimes_{A^R}A)^R$ 
is an isomorphism. Since $R$ is linearly reductive, this can be checked using a Reynolds operator.

We claim that for any line bundle $\cM$ on $Y$, we have $\mathcal{M} \cong p_{U*}p_U^*(\mathcal{M})^U$.  First suppose that the $U$-action is free, so that the quotient $p_U$ is obtained by filtering $U$ by $\GG_a$'s and taking slices locally as in \cref{thm:quotientsbyunip}. Then each of these local $\GG_a$-quotients has a Reynold's operator, which as above we can use to prove this isomorphism. Otherwise we proceed as in \cref{theorem U quotient with varying stabilisers}. If $U$ is abelian, we locally take complements $U' \subset U$ to stabiliser subgroups $U''$ to obtain a free action. Then $p_U = p_{U'}$ and so it remains to check that the $U'' \cong U/U'$-equivariant structure on $p_{U*}p_U^*(\mathcal{M})^{U'}$ is trivial. Since the $U''$-action on the $U'$-quotient is trivial and unipotent groups have no characters, the only possible equivariant structure (i.e.\ linearisation) is the trivial one. Thus taking $U$-invariants is the same as taking $U'$ invariants. If $U$ is not abelian, we can filter $U$ by subgroups with abelian subquotients and proceed inductively as in the proof of \cref{theorem U quotient with varying stabilisers}.

To conclude, we combine the unipotent and reductive cases: setting $\cM = p^*_R\cL$, we have 
\[p_{R*}(p_{U*}p_U^*\cM))^U \cong p_{R*} ((p_{U*}p_U^*\cM)^U) \cong p_{R*}\cM\]
as $p_{R*}$ commutes with taking $U$-invariants, and then taking $R$-invariants gives $(p_*p^*\cL)^G \cong \cL$.
\end{proof}

The following lemma is the final result needed to complete the proof.

\begin{lemma} \label{lemma w discussion before using sections to trivialise locally} 
In the setting of \cref{Main thm on quotienting BB strata} suppose $Z:=Z_i$ is projective and write $Y=X_i$ and $f =f_i$, and let $f'\colon Y' \rightarrow Z'$ denote the base change of $f$ along $Z' :=Z^{ss}(\cL)\hookrightarrow Z$ and $f'' \colon Y'':= Y^{ss}(f;q,\rho) \hookrightarrow Y' \rightarrow Z'$ denote the restriction of $f'$. Then there is a positive integer $r$ and open affine sets $W_\tau$ covering $W :=Z\gitq_{\cL} R$ together with $G$-equivariant isomorphisms
\[ ((q\circ f')_*f'^{*}\cL^{\otimes r}|_{Z'})(W_\tau) \cong (q\circ f')_*\cO_{Y'}(W_\tau)  \cong ((q\circ f'')_*f''^{*}\cL^{\otimes r}|_{Z'})(W_\tau). \]
In particular, this implies that for any positive $m$ and $l$
\[ H^0(Y'',(f''^*\cL^{\otimes mr}))^G_{\rho^l} \cong H^0(Y',(f'^*\cL^{\otimes mr}))^G_{\rho^l}.\]
Moreover, we have isomorphisms
\[ Y \gitq_{q,\rho}^{f} G \simeq \rProj_W \bigoplus_{n \geq 0} (q_* (f')_*f'^*\cL^{\otimes rn}|_{Z'})^G_{\rho^{n}} \simeq \rProj_W \bigoplus_{n \geq 0} (q_* (f'')_*f''^*\cL^{\otimes rn}|_{Z'})^G_{\rho^n}.  \]

\end{lemma}
\begin{proof}

To find this open cover of $W$, we recall how this relative quotient is defined. The semistable set $Y''$ is the union of the non-vanishing loci of semi-invariant sections $\s \in f'_*(\cO_Y)(Z_j)^G_{\rho^{n}}$ where $W_j \subseteq W$ is an open affine and $Z_j := q^{-1}(W_j) \subseteq Z'$ and $n$ is positive. By quasi-compactness, we may assume that this is achieved using a finite affine open cover $\{W_j\subseteq W\}_{j \in \cJ}$ (and the same is true if we further refine this cover), and for each $W_j$ in the cover, we only need finitely many sections $\sigma_{j,i} \in f'_*(\cO_{Y'})(Z_j)^G_{\rho^{n}}$ where $n$ can be chosen to be independent of $W_i$ and $\s_{ij}$. Moreover, since $q$ is a GIT quotient, $W$ has a basis of affine opens $W_\tau$ given by the image under $q$ of the non-vanishing loci of an $R$-invariant section $\tau$ of a positive power of $\cL$. Hence we may assume that each $W_j$ in the cover above is of the form $W_{\tau}$ for some $\tau \in H^0(Z,\cL^{\otimes r})^R$ where again $r$ can be chosen independent of $\tau$ as the ring of invariant sections is finitely generated. Over each $Y_\tau = (q\circ f')^{-1}(W_\tau) \subseteq Y'$, we can trivialise the line bundle $f'^*\cL^{\otimes r}|_{Z'}$ using the non-vanishing section $\tau$ over $Y_\tau$, which gives an isomorphism $f'^*\cL|_{Y_\tau} \cong \cO_{Y'}|_{Y_\tau}$, which is equivariant on sections over $Y_\tau$, since $\tau$ is $R$-invariant. This proves the left isomorphism in the first displayed equation of the lemma, and the right follows verbatim for $f''$, as $Y_\tau \cap Y'' \subseteq Y_\tau$ and so we can restrict the above local equivariant trivialisations.

For the remaining two claims in the lemma, we note that $f'' = i \circ f'$, where $i \colon Y'' \hookrightarrow Y'$ denotes the inclusion. Thus the adjunction $\mathrm{Id} \rightarrow i_*i^*$ induces a homomorphism of equivariant sheaves
\[ (q \circ f')_* f'^* \cL^{\otimes r}|_{Z'} \rightarrow (q \circ f'')_* f''^* \cL^{\otimes r}|_{Z'}, \]
which is an isomorphism, as above we saw it is locally an isomorphism over the open sets $W_\tau$. The remaining statements follow from this equivariant isomorphism (after twisting by $\rho$).
\end{proof}

This concludes the proof of Theorem \ref{Main thm on quotienting BB strata}. We note in the above proof, the argument cannot be applied to $f$, as the open affines $Y_\tau$ do not cover all of $Y:=X_i$, as $q$ is not defined on all of $Z:=Z_i$ only on $Z':=Z_i^{ss}(\cL)$.

\begin{remark} The reader may be tempted to argue in the above by appealing to some general principle that invariant sections over the semistable locus in GIT must always extend over the whole scheme, since by definition all invariant sections must vanish away from there. This is, however, false: consider $\Gm$ acting on $\PP^1$ weights $+1,-1$, linearised with respect to $\cO(1)$. Then $X^{ss} = \AA^1\setminus \{0\}$ and so $\cO(1)_{\mid X^{ss}} \cong \cO_{X^{ss}}$ has a nonvanishing invariant global section, namely the constant one. However, $H^0(\PP^1,\cO(1))^G = 0$, since we only have invariants in even degree. 
\end{remark}

\subsection{Recovering Projective NRGIT}

We will now show how to recover the projective NRGIT quotient of B\'{e}rczi and Kirwan \cite{Berczi2023}, and prove the non-reductive Hilbert--Mumford criterion from our relative point of view. 

Let $G = U \rtimes R$ be a connected graded linear algebraic group with unipotent radical $U$. Let $G$ act on a projective scheme $X$ with respect to an ample linearisation $\cL$. As in \cite{Berczi2023}, we consider the BB decomposition of $X$ associated to the cocharacter that grades $G$, and let $f \colon Y \rightarrow Z$ denote the retraction of the open BB stratum onto its fixed locus component, which in \cite{Berczi2023} is denoted $p: X^0_{\min}\rar Z_{\min}$, as the subscript \lq min\rq\ refers to the fact that $Z_{\min}$ is the fixed component over which $\cL$ has minimal weight for the grading.  We assume that the ample linearisation $\cL$ is \emph{borderline}; that is, the weight of the grading on $Z$ is zero; this can be achieved  by twisting $\cL$ by an appropriate character if necessary. Let $q \colon Z' = Z^{R-ss}(\cL) \rightarrow W$ denote the projective reductive GIT quotient of $Z$ by $\overline{R} = R/\lambda_{\mathrm{gr}}(\GG_m)$ with respect to $\cL|_Z$. As in \textit{loc.\ cit.}\ assume that $\Stab_U(z) = \{e\}$ for all $z\in Z'$.

The main results in \cite{Berczi2023} concern a well-adapted linearisation, which is a notion that is somewhat inconvenient to work with due to needing to twist the borderline linearisation by $\rho^{-\varepsilon}$ where $\varepsilon$ is sufficiently small, but not known precisely. One advantage of our relative twisted affine approach is that it allows us to dispense with this notion. In fact, we rephrase this notion of well-adaptedness as follows.

\begin{definition}
    Let $\rho : G \rar \Gm$ be a character that pairs positively with the grading $\Gm$. We will say that \emph{a property holds for a well-adapted $\rho$-twist of $\cL$} if for all sufficiently divisible $N$ it holds for the twist $\cL_{\varepsilon}$ of the borderline linearisation by the character $\rho^{-\varepsilon}$, where $0 < \varepsilon = \frac{1}{N} \ll 1$.  \end{definition}
    
The main theorem of projective NRGIT is as follows. 

\begin{theorem}[Projective NRGIT Theorem {\cite{Berczi2023,Jackson2026}}]\label{thm Uhat}
Let $G = U \rtimes R$ be a graded linear algebraic group acting on a projective scheme $X$ with respect to an ample borderline linearisation $\cL$. Assume that $\Stab_U(x) = \{e\}$ for all $x\in Z'=Z^{R-ss}(\cL)$. Then the following statements hold for a well-adapted twist $\cL_\epsilon$ of the linearisation $\cL$ on $X$
\begin{enumerate}
    \item\label{part 1 Uhat} The algebra of invariant sections $\bigoplus_{r\geq 0} H^0(X, \cL_\varepsilon^{\otimes r})^G$ is finitely generated, and  \[\bigoplus_{r\geq 0} H^0(X, \cL_\varepsilon^r)^G \hookrightarrow \bigoplus_{r\geq 0} H^0(X, \cL_\varepsilon^r)\] induces a rational map to a projective scheme \[X \dashrightarrow X\gitq G := \Proj\bigoplus_{r\geq 0} H^0(X, \cL_\varepsilon^r)^G\] which restricts to a good quotient on its domain of definition, denoted by $X^{ss}(\cL_\varepsilon)$, and further restricts to a geometric quotient on an open stable set $X^{s}(\cL_\varepsilon)$.
    \item\label{part 2 Uhat} The loci $X^{(s)s}(\cL_\varepsilon)$ are explicitly determined by a Hilbert-Mumford criterion: 
    \[X^{(s)s}(\cL_\varepsilon) = \{ x\in X \mid \mu^{\cL_{\varepsilon}}(x,\lambda) \: (\geq)  \:0 \text{ for every 1-PS $\lambda \colon \Gm \rar G$}\}.\]

\end{enumerate}

\end{theorem}

The first part was proved in \cite{Berczi2023}, as was the second part under an additional \lq semistability coincides with stability\rq\ assumption in \emph{loc.\ cit}. We will provide an alternative proof of this theorem, and a full proof of the Hilbert-Mumford criterion by relating the projective NRGIT quotient described above with a relative twisted affine NRGIT quotient for the graded equivariant action on the open BB stratum
\[ (G \twoheadrightarrow R) \curvearrowright (f \colon Y \rightarrow Z).\] 
We first prove that the (semi)stable loci and quotient maps coincide.

\begin{proposition} \label{lemma inclusion of ss loci for twisted affine and proj qnts}
For $G = U \rtimes R$ acting on a projective scheme $X$ with respect to an ample borderline linearisation $\cL$, we let $f \colon Y \rightarrow Z$ denote the retraction of the open BB stratum. For a well-adapted linearisation $\cL_\varepsilon$, we have $Y^{(s)s}(f;q,\rho) = X^{(s)s}(\cL_\varepsilon)$. Moreover, the projective NRGIT quotient morphism $\varphi_X \colon X^{ss}(\cL_\varepsilon) \rightarrow X \gitq_{\cL_\varepsilon} G $ coincides with the relative twisted affine NRGIT quotient $\varphi_Y \colon Y^{ss}(f;q,\rho) \rightarrow Y\gitq_{q,\rho}^f G$ and in particular, Statement \eqref{part 1 Uhat} of \cref{thm Uhat} holds.
\end{proposition}
\begin{proof}
Recall from Lemma \ref{lemma w discussion before using sections to trivialise locally} that there is a cover of the reductive GIT quotient $W= Z\gitq_{\cL|Z} \overline{R}$ by open affines ${W_{\tau}}$ which are the image under $q \colon Z':=Z^{R-ss}(\cL) \rightarrow  W$ of the nonvanishing locus of sections $\tau \in H^0(Z,\cL^{r}|_Z)^R$, where $r$ can be chosen to be independent of $\tau$. As above, we will let $Y' := f^{-1}(Z')$ and let $Y'':= Y^{ss}(f;q,\rho)$ denote the relative twisted affine semistable locus. By construction, $Y''$ is the union of the (absolute) twisted affine $\rho$-semistable loci $ (Y_\tau)^{ss}(\rho) \subseteq Y_\tau := (q\circ f')^{-1}(W_\tau)$.  We will let $A:= \oplus_{r \geq 0} H^0(X,\cL^r)$, so that $X = \Proj A$. For a homogeneous element $s \in A$, we let $A_{(s)}$ denote the homogeneous localisation so that $X_s = \Spec A_{(s)}$. Since the linearisation $\cL$ is borderline and the action is graded, we have
\[ H^0(Z,\cL^{r}|_Z)^R \cong H^0(X,\cL^r)^G\]
(see \cref{properties of graded actions and linearisations} and \cite[Proposition 2.18]{Hoskins2021})
and, for the invariant section $\tau$, we have $X_\tau = Y_\tau$, as in fact $X^{ss}(\cL) = Y'$ by the above identification of invariant sections.

We begin by proving an equality of semistable loci, which involves relating invariant sections of $\cL_\varepsilon$ with $\rho$ semi-invariant functions. Suppose $x \in (X_\tau)^{ss}(\rho)$ for some $\tau$; then by definition there exists $n >0$ and $s \in \cO(X_
\tau)^G_{\rho^n} = (A_{(\tau)})^G_{\rho^n}$ which is non-vanishing at $x$. Then $s$ has the form $s = \frac{\sigma}{\tau^n}$ where $\sigma \in H^0(X, \cL^{rn})^G_{\rho^{n}}$ is a semi-invariant, as $\tau$ is invariant, and we have $\sigma(x) \neq 0$. We claim that from the semi-invariant section $\sigma$ of the borderline linearisation $\cL$, we can produce an invariant section of the twisted linearisation $\cL_\epsilon$ where $\epsilon = \frac{1}{N}$ for $N$ sufficiently divisible. In fact, we use that $N$ is divisible by $r$ and let $M := \frac{n N}{r}$ and consider $\tilde{\sigma}:=\sigma \tau^{M-n}$ which is also a $\rho^{n}$ semi-invariant of $\cL$, as $\tau$ is an invariant section. Since $\tilde{\sigma}$ is a section of $\cL^{Mr}$, we need to twist the action by $\rho^{-\frac{Mr}{N}} = \rho^{-n}$ and so, we conclude that $\tilde{\sigma}$ is an invariant section of the twisted linearisation $\cL_\epsilon^{Mr}$. Since this section is non-vanishing at $x$, we conclude that $x \in X^{ss}(\cL_\epsilon)$, which proves the forward inclusion.

Conversely, we let $x\in X^{ss}(\cL_\varepsilon)$ and take $\sigma \in H^0(X,\cL_{\varepsilon}^n)^G$ with $\sigma(x)\neq 0$. Since $\cL_\epsilon$ is a sufficiently small perturbation of $\cL$, we have $X^{ss}(\cL_\varepsilon) \subseteq Y' = X^{ss}(\cL)$ (see \cite[Theorem 2.30 ii)]{Hoskins2021}). Hence there is a section $\tau \in H^0(Z,i^*\cL^r)^R$ which is non-vanishing at $x$, as $Y'$ is covered by the open affine sets $Y_\tau$.  Since $N$ is sufficiently divisible, we let $m:=\frac{N}{r}$. Then $s = \frac{\sigma^{N} }{ \tau^{nm}} \in \cO(X_\tau) = A_{(\tau)}$ gives a non-vanishing semi-invariant of weight $\rho^n$.

This completes the proof that the semistable loci coincide. In fact, this moreover shows that both semistable loci can be covered by finitely many open affine sets $X_i \subseteq X_{\tau(i)}$, which can be described either as the non-vanishing of an invariant section $\sigma_i$ of a power of $\cL_\varepsilon$ or as the non-vanishing locus of a $\rho$ semi-invariant $s_i$. From this description and the definition of the stable loci, we conclude that the stable loci also coincide. 

Since the morphism $\varphi_X \colon X^{ss}(\cL_\varepsilon) \rightarrow X \gitq_{\cL_\varepsilon} G$ is induced by the inclusion of invariant rings, it is locally of the form $X_i \mapsto \Spec \cO(X_i)^G $, as both the target and the source are covered by the non-vanishing loci of the invariant sections $\sigma_i$. Since the relative NRGIT quotient map $\varphi_Y \colon Y^{ss}(f;q,\rho) \rightarrow Y\gitq_{q,\rho}^f G$, also admits the same local description on each $X_i$, we conclude that $\varphi_X = \varphi_Y$ and in particular $X \gitq_{\cL_\varepsilon} G$ is a projective scheme, and $\varphi_X$ is a good quotient and restricts to a geometric quotient on $X^{s}(\cL_\varepsilon)$.
\end{proof}

To complete our proof of the projective NRGIT Theorem, we just need to prove the Hilbert-Mumford description of the semistable locus given in Statement \eqref{part 2 Uhat} of \cref{thm Uhat}. In preparation for this, we relate the weight polytope for the projective GIT quotient, with the cone of weights of semi-invariants in affine GIT. We note that Halic related (absolute) twisted affine GIT quotients with projective GIT quotients for reductive group actions \cite{Halic2010}, and we used these ideas to construct affine non-reductive GIT quotients in \cite{HHJ25}.

\begin{definition}\label{def weight polytope nrgit}
Fix a maximal torus $T \subseteq G$ and let $X^*(T)$ denote the character lattice.
\begin{enumerate}
    \item For $x\in X$, the $T$-weight polytope of $x$ with respect to $\cL_\varepsilon$, assumed to be very ample, is
    \[ \cP_T^\varepsilon(x) = \conv \{ \chi \in X^*(T) \mid \exists \: \sigma \in H^0(X,\cL_\varepsilon)^T_\chi \text{ with } \sigma(x) \neq 0 \}.\]
    \item For $\tau \in H^0(X,L^r)^G$ and $x \in X_\tau$, the $T$-weight semi-invariant cone is
    \[ C_T(x,X_\tau):= \{w\in X^*(T) \mid \exists \: s\in \cO(X_\tau)^T_w \text{ with } s(x)\neq 0 \}.\]
\end{enumerate}
\end{definition}

Note that $\cP_T^\varepsilon(x)$ is the convex hull of \emph{minus} the weights lying over $x$ in the affine cone with respect to the equivariant projective embedding $X \hookrightarrow \PP(V)$ induced by $\cL_\varepsilon$, as the affine cone $V:= H^0(X,\cL_{\varepsilon})^*$ is the dual to the above representation. Hence for a 1-PS $\lambda$ contained in $T$, we have $\mu^{\cL_\varepsilon}(x,\lambda) = \max \{ \langle \chi,\lambda \rangle \mid  \chi \in \cP_T^\varepsilon(x) \}$. We also note that $\cP_T^0(x)$ denotes the weight polytope with respect to the borderline linearisation. 

The following lemma also appears in \cite{Jackson2026}, but we include a short self-contained proof.

\begin{lemma}\label{lem cone-polytope comparison for HM}
Let $T\subseteq G$ be a maximal torus and $\tau \in H^0(X,\cL^r)^G$, where $\cL^r$ is very ample. Then for $y\in X_\tau$, we have  
\[ C_T(y,X_\tau)= \Cone \cP^0_T(y).\] 
\end{lemma}
\begin{proof}
As in the proof of \cref{lemma inclusion of ss loci for twisted affine and proj qnts}, we let $A:= \oplus_{r \geq 0} H^0(X,\cL^r)$. If $w \in C_T(y;X_\tau)$, then there is a semi-invariant $s\in \cO(X_\tau)^T_w$ which is non-vanishing at $y$. Since $\cO(X_\tau) \cong A_{(\tau)}$, we can write $s = \frac{\s}{\tau^m}$ where $\s \in H^0(X,\cL^{rm})$ is also a semi-invariant of weight $w$, as $\tau$ is an invariant section of $\cL^r$, and is also non-vanishing at $y$. Since $\cL_\varepsilon$ is very ample, the weight polytope of $y$ with respect to $\cL_\varepsilon^{rm}$ is just the $rm$th dilation of $\cP_T^0(y)$. Hence $w \in  rm \cdot\cP^0_T(y)$ and hence $w \in \Cone \cP^0_T(y)$. 
Conversely, let $w \in \Cone \cP_T^0(y).$ Then for some $n_1,n_2 \in \ZZ_>0$ we have $n_1w \in n_2\cdot \cP_T^0(y)$. Then there is $\sigma \in H^0(X,\cL^{n_2})^T_{n_1 w}$ which is non-vanishing at $y$. Hence, the semi-invariant $s =\frac{\sigma^r}{\tau^{n_2}} \in (A_{(\tau)})^T_{rn_1w}$ is non-vanishing at $y$, so $w \in C_T(y;X_\tau)$.\end{proof}

We can now finish our independent proof of the projective NRGIT theorem.

\begin{proof}[Proof of \cref{thm Uhat}]
By \cref{lemma inclusion of ss loci for twisted affine and proj qnts}, all that remains is to prove the Hilbert--Mumford description of the semistable locus in Statement \eqref{part 2 Uhat} and we know that $X^{(s)s}(\cL_\varepsilon) = Y^{(s)s}(f;q,\rho)$, where $f \colon Y \rightarrow Z$ is the retraction from the open BB stratum and $q \colon Z' \rightarrow Z \gitq_{\cL} R$ is the reductive GIT quotient. 
Our aim is to prove that 
\[ Y^{(s)s}(f;q,\rho) = X^{\mathrm{HM}-(s)s}(\cL_\varepsilon):= \{ x \in X \mid \mu^{\cL_{\varepsilon}}(x,\lambda) (\geq) 0 \text{ for every 1-PS } \lambda:\Gm \rar G\}. \]
Since the Hilbert--Mumford weight $\mu^{\cL_{\varepsilon}}(x,\lambda)$ can be computed from the weight polytope $\cP_T^{\varepsilon}(x)$ for a maximal torus $T$ containing $\lambda$, we have $ x \in X^{\mathrm{HM}-ss}(\cL_\varepsilon)$ if and only if $0 \in \cP_T^{\varepsilon}(x)$ for every maximal torus $T \subseteq  G$.

First, we claim that these (semi)stable loci are contained in $Y':= f^{-1}(Z')$. This is true by definition for the relative twisted affine semistable locus $Y^{(s)s}(f;q,\rho)$, and for the HM-semistable locus, this holds as $\cL_\varepsilon$ is a small perturbation of $\cL$, so we have $X^{\mathrm{HM}-ss}(\cL_\varepsilon) \subseteq X^{\mathrm{HM}-ss}(\cL) = Y'$. 

As in the proof of \cref{lemma w discussion before using sections to trivialise locally} and \cref{lemma inclusion of ss loci for twisted affine and proj qnts}, we use that $Y'$ is covered by the open affine sets $X_\tau$ for $\tau \in H ^0(Z,\cL^r|_Z)^R$. Thus, to compare the (semi)stable loci, it suffices to consider a point $x$ in some $X_\tau$. Then, $x$ is Hilbert--Mumford semistable (resp.\  stable) for $\cL_\varepsilon$ if and only if $\varepsilon \rho \in  \cP^0_T(y)$ (resp.\ the interior) for all maximal tori $T \subseteq  G$. By \cite[Lemma 2.12]{HHJ25} and \cref{lem cone-polytope comparison for HM},  we conclude that $x \in X^{\mathrm{HM}-(s)s}(\cL_\varepsilon)$ if and only if $x \in X_\tau^{(s)s}(\rho)$, where this latter (semi)stable locus is the absolute $\rho$-twisted semistable locus on the affine scheme $X_\tau$. Since $Y^{(s)s}(f;q,\rho)$ is the union of the absolute semistable loci $X_\tau^{(s)s}(\rho)$ (see \cref{rmk relative ss loci union of absolute ss loci}), this completes the proof.
\end{proof}

\section{Applications}\label{sec applications}

In this section we give various applications of our relative quotients.

\subsection{Relative constructions of moduli spaces of quivers with multiplicities}\label{sec quivers with mult}

The last two authors together with T. Vernet  \cite{HJV} use the methods in this paper to construct moduli spaces of representations of a quiver with multiplicities relative to King's moduli spaces of quiver representations (without multiplicities). In this section, we briefly outline how we apply the results of this paper to construct these moduli spaces and describe their key properties.

Let us first fix notation: by a quiver $Q=(Q_0;Q_1)$, we mean a directed graph with vertex set $Q_0$ and arrow set $Q_1$. For each $i \in Q_0$ we fix a multiplicity $m_i$ and let $\mathbf{m}$ denote the tuple of multiplicities. For each $m_i$, we associate a truncated polynomial ring $k_{m_i} = k[\epsilon]/\epsilon^{m_i}$. A (locally free) representation of $(Q, \mathbf{m})$ is given by free $k_{m_i}$-modules $M_i$ for each $i \in Q_0$ and appropriate linear maps $M_i \rightarrow M_j$ for each arrow $a\colon i \rightarrow j$ (see \cite[Def.\ 2.1.5]{HJV}). The discrete invariants of $M$ are given by the tuple $\mathbf{r}$ of ranks $r_i$ of each $M_i$. The stack of rank $\mathbf{r}$ representations of $(Q, \mathbf{m})$ is thus a quotient of an affine space $R(Q,\mathbf{m};\mathbf{r})$ (obtained by choosing bases of each $M_i$) by
\[ \GL_{\mathbf{m},\mathbf{r}}= \prod_{i \in Q_0} \GL_{r_i}(k_{m_i})\]
acting by conjugation. Furthermore, this action naturally fits into an equivariant set-up
\begin{equation}\label{equivariant quiver action}
(\GL_{\mathbf{m},\mathbf{r}} \longrightarrow \GL_{\mathbf{r}}) \curvearrowright (R(Q,\mathbf{m};\mathbf{r}) \stackrel{\tau}{\longrightarrow} R(Q,\mathbf{r}))  
\end{equation}
where $\GL_{\mathbf{r}} := \GL_{\mathbf{1},\mathbf{r}}$ acting on $R(Q,\mathbf{r}):=R(Q,\mathbf{1};\mathbf{r})$ describes the stack of representations of $Q$ (without multiplicities). The morphism $\tau$ is a truncation map which associates to a quiver with multiplicities an underlying quiver without multiplicities. In fact, the truncation $\GL_{\mathbf{m},\mathbf{r}} \rightarrow \GL_{\mathbf{r}}$ is the quotient of this non-reductive group by its unipotent radical. Our goal is to take a relative NRGIT quotient with respect to King's reductive GIT quotient
\[ q_\theta \colon R(Q,\mathbf{r})^{\theta-ss} \rightarrow R(Q,\mathbf{r})\gitq_{\hspace{-2pt}\theta} \GL_{\mathbf{r}} \]
for a stability parameter $\theta$ (a tuple of integers indexed by $Q_0$ such that $\theta \cdot \mathbf{r}=0$). One of the main results of \cite{HJV} is the following construction of moduli spaces of quivers with multiplicities.

\begin{theorem}[{\cite[Theorem A]{HJV}}]
The equivariant action \eqref{equivariant quiver action} can be externally graded, and a relative NRGIT quotient can be constructed with respect to $\theta$ and an additional parameter $\rho$, provided that we assume the unipotent stabiliser condition $R(Q,\mathbf{r})^{\theta-ss} \subseteq R(Q,\mathbf{r})^\circ$. This yields a good quotient
\[ R(Q,\mathbf{m};\mathbf{r})^{\theta,\rho-ss} \rightarrow R(Q,\mathbf{m};\mathbf{r}) \gitq_{\hspace{-2pt}\theta,\rho} \GL_{\mathbf{m},\mathbf{r}}\]
with the following properties:
\begin{enumerate}
    \item $R(Q,\mathbf{m};\mathbf{r}) \gitq_{\hspace{-2pt}\theta,\rho} \GL_{\mathbf{m},\mathbf{r}}$ is projective-over-affine over King's moduli space $R(Q,\mathbf{r})\gitq_{\hspace{-2pt}\theta} \GL_{\mathbf{r}}$ of $\theta$-semistable representations of $Q$ (without multiplicities);
    \item The GIT semistable locus admits a moduli-theoretic interpretation formulated for representations of $(Q,\mathbf{m})$;
    \item The points of $R(Q,\mathbf{m};\mathbf{r}) \gitq_{\hspace{-2pt}\theta,\rho} \GL_{\mathbf{m},\mathbf{r}}$ parametrise S-equivalence classes of $(\theta,
    \rho)$-semistable representations of $(Q,\mathbf{m})$.
    \item The condition $R(Q,\mathbf{r})^{\theta-ss} \subseteq R(Q,\mathbf{r})^\circ$ is satisfied if $\theta$ is generic with respect to $\mathbf{r}$.
\end{enumerate}
\end{theorem}

More generally, we construct versions of Nakajima quiver varieties with multiplicties, and also study their cohomological purity (see \cite[Theorem B]{HJV}).

\subsection{Relative constructions of moduli of unstable objects}

For a reductive group $G$ acting on a projective-over-affine scheme $Y$ with respect to an ample linearisation $\cL$, the classical GIT semistable set $X^{ss}(\cL)$ is the open stratum in an instability stratification studied by Hesselink, Kempf, Kirwan and Ness \cite{Hesselink,Kempf1978,KirwanThesis,Ness}. The unstable strata are indexed by finitely many conjugacy classes of (rational) 1-parameter subgroups of $G$, which are \lq most responsible' for the instability of these points, where this notion (and thus the stratification) depends on a conjugation invariant norm on 1-PSs of $G$ which is used to normalise the Hilbert--Mumford weight. Let us give some basic definitions before stating the main properties of this stratification.

\begin{definition}\label{def:instab strat}
Let $G$ be a reductive group acting on a projective-over-affine scheme $Y$ with respect to an ample linearisation $\cL$ and fix a norm $|| -|| $ on conjugacy classes of 1-PSs in $G$.
    \begin{enumerate}
    \item For $x \in X$ and a 1-PS $\lambda$ of $G$ such that $x_0:=\lim_{t \rightarrow 0} \lambda(t) \cdot x$ exists, we define
    \[ \overline{\mu}(x,\lambda):= \frac{\mu^{\cL}(x,\lambda)}{||\lambda||}\]
    and $M(x) := \inf\{ \overline{\mu}(x,\lambda) \, : \: \lim_{t \rightarrow 0} \lambda(t) \cdot x 
 \text{ exists}\}$. 
 \item If $x$ is unstable, we say a primitive 1-PS $\lambda$ is \emph{adapted} to $x$ if it achieves the minimum $M(x)$ and $\lim_{t \rightarrow 0} \lambda(t) \cdot x$ exists. We let $\Lambda(x)$ denote the set of adapted 1-PSs to $x$.
\item For a 1-PS $\lambda$, we define a parabolic subgroup $P_\lambda \subseteq  G$ with Levi subgroup $L_\lambda$ by
\[ P_\lambda := \{ g \in G \: : \: \lim_{t \rightarrow 0} \lambda(t) g \lambda(t)^{-1} \text{ exists} \} \stackrel{r_{\lambda}}{\longrightarrow} L_\lambda := Z(\lambda(\GG_m)).\]
The kernel of $r_\lambda$ is the unipotent radical $U_\lambda$ of $P_\lambda$ and we have $P_\lambda = U_\lambda \rtimes L_\lambda$.
 \item For a pair $\beta = (\lambda,d)$ consisting of a 1-PS $\lambda$ and $d \in \RR_{<0}$, we will write $P_\beta:=P_\lambda$ and similarly write $U_\beta := \ker (r_\beta \colon P_\beta \rightarrow L_\beta)$. For such a pair, we define the associated blade $Y_{\beta}^{ss}$ and  limit set $Z_{\beta}^{ss}$ as the following subschemes\footnote{Here we just define these set theoretically, but they naturally admit schematic structures; see \cite{Hesselink}.} of $X$
 \[ Y_{\beta}^{ss} = \{x \in X:  M(x) = \overline{\mu}(x,\lambda)= d\} \stackrel{p_\beta}{\longrightarrow} Z_{\beta}^{ss} = \{x \in X^\lambda:  M(x) = \overline{\mu}(x,\lambda)= d\}  \]
 where $X^{\lambda}$ denotes the fixed locus.
 The limit set and blade have (relative) closures:
 \[ Z_{\beta} = \{x \in X^\lambda : \overline{\mu}(x,\lambda)= d \} \stackrel{p_\beta}{\longleftarrow
} Y_{\beta} = \{ x \in X : \overline{\mu}(x,\lambda)= d \}.  \]
In particular $Y_{\beta} =\{ x \in X : \lim_{t \rightarrow 0} \lambda(t) \cdot x \in Z_{\beta} \}$ and $p_\beta(x) = \lim_{t \rightarrow 0} \lambda(t) \cdot x$. 
 For $[\beta] := ([\lambda],d)$, we define the associated stratum
 \[ S_{[\beta]} = \{ x \in X : M(x) = d \text{ and } [\lambda] \cap \Lambda(x) \neq \emptyset \}.\]
Equivalently one can associate to $(\lambda,d)$ a rational 1-PS $\lambda_d$ such that $\mu(x,\lambda_d)=- || \lambda_d ||$ for all $x \in Y^{ss}_{\beta}$ and view $[\beta]$ as the conjugacy class of the rational 1-PS $[\lambda_d]$.
\item For a pair $\beta = (\lambda,d)$ as above, the canonical linearisation $\cL_\beta$ of the $L_\lambda$-action on $Z_{\beta}$ is obtained by twisting the linearisation $\cL$ by the rational character $\chi_\beta$ which corresponds to rational 1-PS $\lambda_d$ under the identification between the character and cocharacter lattice induced by the choice of norm $|| - ||$.
\end{enumerate}
\end{definition}

The next result collects the main properties of this instability stratification.

\begin{theorem}[Hesselink, Kempf, Kirwan, Ness]\label{thm instab strat properties}
Let $G$ be a reductive group acting on a projective-over-affine scheme $X$ with respect to an ample linearisation $\cL$. Then there is a finite stratification of $X$ into $G$-invariant locally closed subschemes 
\[ X = X^{ss}(\cL) \sqcup \bigsqcup_{[\beta] \in \cB} S_{[\beta]} \]
with the following properties:
\begin{enumerate}
    \item For a representative $\beta=(\lambda,d)$ of the conjugacy class, we have $S_{[\beta]} = G Y_\beta^{ss} \simeq G \times^{P_\lambda} Y_\beta^{ss}$.  
    \item We have $Z_\beta^{ss} = Z_\beta^{L_\beta-ss}(\cL_\beta)$ and $Y_\beta^{ss} = p_\beta^{-1}(Z_\beta^{ss})$.
    \item The index set $\cB$ can be computed from the weights of the action of a maximal torus.
\end{enumerate}
\end{theorem}

In particular, a categorical quotient of the $G$-action on (an open subset of) $S_{[\beta]}$ is equivalent to a categorical quotient of the $P_\beta$-action on (an open subset of) $Y_\beta$ for some fixed representative $\beta$.

\begin{lemma}\label{lemma grading unstable strata}
Let $\beta = (\lambda,d)$ be a representative for an index of an unstable stratum as above. Then the equivariant action of $P_\beta \twoheadrightarrow L_\beta$ on $p_\beta \colon Y_{\beta} \rightarrow Z_{\beta}$ is internally graded by $\lambda$.
\end{lemma}
\begin{proof}
    This follows by similar arguments to \cref{lemma BB stratum internally graded}.
\end{proof}

In characteristic zero, we can take a relative NRGIT quotient of the above internally graded equivariant action with respect to the good quotient $q_\beta \colon Z_\beta^{ss} \rightarrow Z_\beta \gitq_{\cL_\beta} L_\beta$ and a character $\rho \colon P_\beta \rightarrow \GG_m$ provided the locus $Z_\beta^\circ$ on which the unipotent stabilisers are well-behaved is contained in $Z_\beta^{ss}$.

\begin{theorem}\label{thm app unstable quotients}
    Let $\beta$ be a representative of an unstable stratum $S_{[\beta]}$ in an instability stratification of a projective-over-affine scheme associated to an amply linearised action of a reductive group $G$ and a choice of norm as above. Suppose that $ Z_\beta^{ss} \subseteq Z_\beta^\circ$. For any character $\rho \colon P_\beta \rightarrow \GG_m$, there are open (semi)stable subsets with explicit Hilbert--Mumford descriptions
    \begin{equation}\label{NRHM for unstable strat}
        (Y_\beta)^{ss}(p_\beta;q_\beta,\rho):= \left\{ x \in Y_\beta^{ss} \left| \begin{array}{c} \text{$\langle \rho, \lambda \rangle \: (\geq) \: 0$ for all one-parameter subgroups $\lambda  \colon \Gm \rar P_\beta$} \\ \text{such that the limit exists in $Y_\beta^{ss}$} \end{array} \right. \right\}
    \end{equation}
    and a good $P_\beta$-quotient
    \[ (Y_\beta)^{ss}(p_\beta;q_\beta,\rho) \rightarrow Y_\beta \gitq_{q_\beta,\rho}^{p_\beta} P_\beta \]
    where $Y_\beta \gitq_{q_\beta,\rho}^{p_\beta} P_\beta$ is projective over $Z_\beta \gitq_{\cL_\beta} L_\beta$. In particular, $Y_\beta \gitq_{q_\beta,\rho}^{p_\beta} P_\beta$ is a good $G$-quotient of the open subscheme
    \[ S_{[\beta]}^{ss}(p_\beta;q_\beta,\rho):= G \cdot (Y_\beta)^{ss}(p_\beta;q_\beta,\rho) \subseteq S_{[\beta]}. \]
\end{theorem}
\begin{proof}
    This follows immediately from \cref{lemma grading unstable strata} and \cref{main theorem}.
\end{proof}

In examples of interest (such as moduli of objects of fixed Harder--Narasimhan type, see $\S$\ref{sec moduli fixed HN type}) this stabiliser condition often fails: $Z_\beta^{ss} \nsubseteq Z_\beta^\circ$. However, if $Z_\beta':=Z_\beta^{\circ} \cap Z_\beta^{ss}$ is a saturated $L_\beta$-invariant subset, then we can apply \cref{main theorem} to the restricted good quotient $Z_\beta' \rightarrow q_\beta(Z_\beta')$.

\subsubsection{Moduli of objects of fixed Harder--Narasimhan type}\label{sec moduli fixed HN type}

In this subsection, we will explain how the above result can be used to construct moduli space of objects of fixed Harder--Narasimhan type in an abelian moduli problem, subject to some notion of (semi)stability. We will focus on:
\begin{enumerate}
    \item moduli of sheaves on a (polarised) projective scheme $(X,\cO(1))$;
    \item moduli of representations of a quiver $Q$.
\end{enumerate}
These techniques should also extend to other similar moduli problems, such as moduli of quiver sheaves (including Higgs bundles, pairs, and chains) and parabolic bundles or decorated bundles. 

In these moduli problems, GIT (semi)stability is interpreted moduli-theoretically as requiring a certain inequality of \lq slopes' to hold for subobjects. For a fixed notion of stability, every object admits a unique \emph{Harder--Narasimhan filtration} (HN-filtration) such that the successive quotients are semistable with strictly decreasing slopes, which is inductively built by considering the largest subobject of maximal slope \cite{HN,Reineke}. In particular, semistable objects have trivial HN filtration. The \emph{HN-type} of an object encodes the discrete invariants of the successive quotients, and one can naturally stratify the stack of all objects by HN type. This can be pulled back to families: any scheme parametrising a family of objects also admits an associated HN-stratification into locally closed subschemes of objects of fixed HN-type (for example, see \cite{Nitsure_schHN}).

In the case of sheaves, the HN stratification was (asymptotically) related to instability strata arising in the construction of moduli spaces of (Gieseker) semistable sheaves as a quotient of a subscheme of a Quot scheme by the second author and Kirwan \cite{HoskinsKirwan}. The problem of constructing moduli spaces of sheaves (or vector bundles) of fixed HN type was considered in \cite{Hoskins2021,Jackson2021}, but the approach used projective NRGIT which required a complicated blow-up procedure to achieve certain unipotent stabiliser assumptions. By using the tools of relative NRGIT developed in this paper, we can now bypass this issue. As an added bonus, the moduli spaces of objects of fixed HN type we construct are projective over the product of moduli spaces for the semistable factors appearing in the HN filtration. In thecase of quiver representations, the notion of semistability depends on a choice of stability parameter $\theta$ introduced by King \cite{King1994}. For fixed $\theta$, the instability stratification and HN stratification coincide \cite{Hoskins2014}. Since this moduli construction uses twisted affine GIT, it was not natural to apply the results of projective NRGIT, but it is natural to apply our results on relative NRGIT.

For the rest of this subsection, we will fix a length $l$ HN type $\tau = (\tau_1,\dots,\tau_l)$ either for 
\begin{enumerate}
    \item moduli of sheaves on a projective scheme $(X,\cO(1))$ with respect to Gieseker semistability;
    \item moduli of representations of $Q$ with respect to $\theta$-semistability.
\end{enumerate}
Here $\tau_i$ denotes the discrete invariants (Hilbert polynomial or dimension vector) of the $i$-th successive quotient in the length $\ell$ HN filtration. We say $\tau$ is \emph{coprime} if semistability coincides with stability for objects with invariants $\tau_i$ for $1 \leq i \leq l$. For a HN type $\tau$, there is an associated (rational) 1-PS $\lambda_\tau$ and a retraction from the blade to the limit set
\[ Y_\tau^{ss} \stackrel{p_\tau}{\longrightarrow} Z_\tau^{ss}\]
which is equivariant with respect to the retraction $P_\tau \rightarrow L_\tau$ from the associated parabolic to its Levi factor. In particular, the stack of objects of HN type $\tau$ is isomorphic to $[Y_\tau^{ss}/P_\tau]$. Our goal is to describe the semistable set for the relative quotient of the graded equivariant action
\[ (P_\tau \rightarrow L_\tau ) \curvearrowright (Y_\tau \rightarrow Z_\tau) \]
with respect to a character $\rho \colon P_\tau \rightarrow \GG_m$ and the reductive GIT quotient $q_\tau$ of the limit set with respect to the canonical linearisation $L_\tau$
\[ q_\tau : Z_\tau^{ss} \longrightarrow Z_\tau \gitq_{\cL_\tau} L_\tau = M_{\tau_1}^{ss} \times \cdots \times M_{\tau_l}^{ss}\]
where $M_{\tau_i}^{ss}$ is the moduli space for semistable objects with discrete invariants $\tau_i$. Note that $\tau$ is coprime if and only if $Z_\tau^{ss}= Z_\tau^{s}$. To understand the Hilbert--Mumford description \eqref{NRHM for unstable strat} of the semistable locus, we first consider 1-PSs of $L_\tau$ with the property that the flow under such a 1-PS remains inside $Y_\tau^{ss}$.

\begin{proposition}\label{prop HN flow exists in Ytauss}
In the above two examples, let $x \in Y_\tau^{ss}$ and let $\lambda \colon \GG_m \rightarrow L_\tau$ be a 1-PS. Then $\lim_{t \rightarrow 0} \lambda(t) \cdot x$ exists in $Y_\tau^{ss}$ if and only if $\lambda$ induces a filtration of the object corresponding to $x$ that refines the HN filtration and such that the successive quotients are semistable.
\end{proposition}
\begin{proof}
If $\overline{x} =\lim_{t \rightarrow 0} \lambda(t) \cdot x$ exists in $Y_\tau^{ss}$, then $\lim_{t \rightarrow 0} \lambda(t) \cdot p_\tau(x) = p_\tau(\overline{x}) =: \overline{z}$ exists in $Z_\tau^{ss}$, which means $\overline{x}$ lies in the orbit closure of $p_\tau(x)=(z_1,\dots,z_l) \in \prod_{i=1}^l M_{\tau_i}^{ss}$. In particular, this means $z_i$ and $\overline{z_i}$ are $S$-equivalent, and so $\lambda$ corresponds to a coarsening of the Jordan--Holder filtration of $z_i$, which proves the forwards direction. Conversely if $\lambda$ refines the HN filtration such that the successive quotients are semistable, then  $\lim_{t \rightarrow 0} \lambda(t) \cdot x$ still has HN-type $\tau$ and, as $\lambda$ is a 1-PS of $L_\tau$, we see that $\overline{x} \in Y_\tau^{ss}$.
\end{proof}

The above can be used to derive intrinsic moduli-theoretic conditions for what we might call $\tau$-semistability of sheaves and quiver representations, generalising the $\tau$-stability of \cite{Hoskins2021,Jackson2021}. This will appear in future work. For now, we observe that if $Z_\tau^{ss} = Z_\tau^s$, such a limit can never exist, and so \cref{NRHM for unstable strat} yields the following corollary.

\begin{corollary}
    If $\tau$ is a coprime HN type and $\rho$ is any character for which $ \langle \rho,\lambda_\tau \rangle > 0$, then 
    \[ (Y_\tau)^{ss}(p_\tau;q_\tau,\rho) = (Y_\tau)^{s}(p_\tau;q_\tau,\rho) = Y_\tau^{ss} \setminus U_\tau Z_\tau^{ss}.\]
    In particular, if $Z_\tau^{ss} \subseteq Z_\tau^\circ$, then this open set admits a geometric quotient which is projective over the product of moduli spaces $M_{\tau_1}^{s} \times \cdots \times M_{\tau_l}^{s}$ for the stable factors in the HN filtration.
\end{corollary}

However, when $Z_\tau^{ss}  \neq  Z_\tau^s$, the choice of $\rho$ gives different notions of (semi)stability which involve testing for subobjects which are coarsenings of a Jordan--Holder--Harder--Narasimhan filtration, as arising in \cref{prop HN flow exists in Ytauss}. In this situation, one can compute the $L_\tau$-semistable locus with respect to $\rho$ in terms of an inequality for such subobjects, and then sweep out by the $U_\tau$-action. This is exactly the strategy the last two authors used with T. Vernet to explicitly describe semistability for quivers with multiplicities \cite{HJV}.

\subsection{Relative constructions for moduli of jets}\label{sec jets}
For a smooth $d$-dimensional complex variety $Y$ and a positive integer $k$, we let $\pi \colon J_k(Y) \rightarrow Y$ denote the bundle of $k$-jets of germs of parametrised curves in $Y$. By choosing local coordinates at $y$, the fibre $\pi^{-1}(y)$ can be identified with the vector space $J_k(d)$ of $k$-jets of holomorphic maps $f \colon (\CC,0) \rightarrow (\CC^d,0)$ which is an affine space \[J_k(d) = \{(f'(0),\dots,f^{(k)}(0))\mid f^{(i)}(0) \in \CC^d,\forall i=1,\dots,k\}\cong \AA^{kd}. \] 
Since the transition functions are polynomial rather than linear, $\pi$ is not a vector bundle. We can reparametrise any jet by pre-composing with a $k$-jet $\varphi \in J_k(1)$ and this induces a fibrewise action of the group of \emph{regular} $k$-jets (i.e.\ with non-vanishing first derivative) 
\[ \mathrm{Diff}_k:=J^{\mathrm{reg}}_k(1) \curvearrowright (\pi \colon J_k(Y) \rightarrow Y).\]
This action is explicitly described in \cite[$\S$2.3]{Berczi2024}. The group $\mathrm{Diff}_k = U_k \rtimes \GG_m$ is a semi-direct product of a unipotent group $U_k$ of dimension $k-1$ with a multiplicative group, which acts on the Lie algebra of $U_k$ with positive weights $1,\dots, k-1$. Furthermore this multiplicative group acts on the fibres $\pi^{-1}(y) \cong J_k(d)$ with positive weights $1,\dots, k$, and thus grades the above fibrewise action. We cannot directly apply \cref{main theorem} as the unipotent stabiliser assumptions do not hold: points in the zero section have full stabiliser group. When $Y$ is projective, B\'erczi and Kirwan use projective non-reductive GIT to construct a quotient of this action by projectivising the jet bundle and blowing-up to obtain the appropriate unipotent stabiliser conditions. Their resulting NRGIT quotient yields spectacular progress on hyperbolicity of projective hypersurfaces \cite{Berczi2023,Berczi2024}. We can analogously solve this stabiliser problem by blowing up the zero section, which enables constructions of jet moduli for non-compact $Y$ and also simplifies the construction of B\'{e}rczi and Kirwan by eliminating the need to work with a projectivised jet bundle.

\begin{theorem}\label{thm jets}
Let $Y$ be a smooth complex variety of dimension $d$. Then relative NRGIT yields a fibrewise compactification $\cY^{\mathrm{GIT}}_k$ of $J^{\text{reg}}_k(Y)/\mathrm{Diff}_k$, which is projective over $Y$. The fibres of $\cY^{\mathrm{GIT}}_k \rar Y$ are themselves fibrations of weighted projective spaces over $\PP^{n-1}$. 
\end{theorem}

\begin{proof}
In each fibre $\pi^{-1}(y) \cong J_k(d)$, we blow-up the origin and consider the open BB-stratum \[X_y:= \{((\underline{v}_1,\dots,\underline{v}_k),[\underline{w}_1:\dots:\underline{w}_k]) \in \Bl_0 \AA^{kd} \mid \underline{w}_1 \neq 0\} \subset \Bl_0 \AA^{kd} \subset\AA^{kd}\times \PP^{kd-1}. \] Indeed the $\GG_m$-weights on $\AA^{kd}$ are all strictly positive ($\underline{v}_i$ has weight $i$), and $\underline{w}_1$ has the minimal weight on the projective factor. Hence the minimal weight space in the blow-up is $Z_y:=\PP^{d-1}$. The flow as $t\rar 0$ corresponds induces an affine morphism \[ f_y \colon X_y \rar Z_y, \quad ((\underline{v}_1,\dots,\underline{v}_k),[\underline{w}_1:\dots:\underline{w}_k])\longmapsto [\underline{w}_1].\] 
In fact $X_y$ is isomorphic to the open BB stratum in the blow-up of $\PP(J_k(d) \oplus k)$ at $[\underline{0}:1]$ considered by B\'{e}rczi and Kirwan, and $U_k$ acts freely on $X_y$ (see \cite[$\S$5]{Berczi2024}). 

Returning to the global picture over $Y$, we let $\Phi \colon \widetilde{J} \rightarrow J_k(Y)$ be the blow-up of $J_k(Y)$ along the zero-section and consider the open BB stratum $X \subset \widetilde{J}$. Flowing under the $\GG_m$-action gives an affine morphism $ f \colon X \rightarrow Z$,
where $Z$ is projective over $Y$ with fibres $Z_y =\PP^{d-1}$. Applying \cref{main theorem} to the fibrewise $\mathrm{Diff}_k$-action on $f$ (and noting that the character $\rho$ corresponds to an integer which should be chosen positive to get a non-trivial quotient as in \cref{example VGIT Gm}), we obtain a geometric quotient $X^s =X \setminus U_k Z \rightarrow X \gitq_{\hspace{-0.6mm}\rho \hspace{0.2mm} } \mathrm{Diff}_k$ which is projective over $Z$, and thus projective over $Y$. If $E \subset \widetilde{J}$ denotes the exceptional divisor, then we obtain a geometric quotient $\Phi(X^s \setminus E) = J_k(Y) \setminus Y \rightarrow (J_k(Y) \setminus Y)/ \mathrm{Diff}_k$ which is quasi-projective over $Y$, together with a natural relative projective completion $\phi \colon \cY^{\text{GIT}}_k \rar Y$. For each fibre we have a map $\phi^{-1}(y) \rar Z_y \cong \PP^{d-1}$, and as in \cite[\S5.3]{Berczi2024}, all the fibres of the latter are weighted projective spaces. 
\end{proof}

\begin{remark}[Jet versions of reductive group actions]
We can extend the situation considered in $\S$\ref{sec quivers with mult} to more generally consider an action of a reductive group $G$ on a scheme $Z$ for which there is a good quotient $q \colon Z' \rightarrow W$ of an open $G$-invariant subset $Z'$. If we take the space of $k$-jets on $X$ and the group of $k$-jets on $G$, then we get an equivariant action of
\[ (G_k:=J_k^{\mathrm{reg}}(G) \rightarrow G ) \curvearrowright (J_k(Z) \rightarrow Z)\]
and one can naturally ask if it is possible to construct a relative quotient with respect to $q$. From the invariant theory perspective, in the case where $Z$ is affine, this problem was studied in \cite{LSS}, where the invariant ring $\cO(J_k(Z))^{G_k}$ is compared with the coordinate ring of the $k$-jet space of the affine GIT quotient $W=Z\gitq G$. From the relative viewpoint, one would need to find a grading of the above action and verify the unipotent stabiliser assumptions. If $G$ is a product of general linear groups, the external grading used in \cite{HHJ25} for quivers with constant multiplicities would grade $G$. Also by varying $\rho$, one may get different quotients from $J_k(W).$

More generally, if $G$ is a product of groups and $X$ is a product of schemes (but the action is not necessarily a product), one may take jets of different orders $k_i$ and ask if a relative quotient can be constructed. This is closer to the approach in \cite{HJV} (see $\S$\ref{sec quivers with mult}), where moduli of quiver representations with varying multiplicities are studied.
\end{remark}

\bibliographystyle{alpha}
 \bibliography{references}

 \end{document}